\RequirePackage{fix-cm}
\documentclass[twocolumn]{svjour3}
\smartqed  
\usepackage{microtype}
\usepackage{graphicx}
\usepackage{amsmath}
\usepackage{amsfonts}
\usepackage{amssymb}
\usepackage{mathptmx}  
\usepackage[numbers]{natbib}
\RequirePackage[colorlinks,citecolor=blue,urlcolor=blue]{hyperref}

\usepackage[utf8]{inputenc}
\usepackage{microtype}
\usepackage{dsfont}
\usepackage[dvipsnames]{xcolor}
\usepackage{enumerate}
\usepackage{enumitem}
\usepackage{changes}
\usepackage{graphicx}
\usepackage{epstopdf}
\usepackage{multirow}
\usepackage{caption}
\usepackage{algorithm}
\usepackage{algorithmic}
\usepackage{savesym}
\usepackage{booktabs}
\usepackage{setspace}
\usepackage{subcaption}
\usepackage{nccmath}
\usepackage{bm}
\usepackage{hyperref}

\newtheorem{assumption}{Assumption}

\newcommand{\var}{\operatorname{Var}}
\newcommand{\spann}{\operatorname{Span}}
\newcommand{\prob}{\operatorname{\mathbb{P}}}
\newcommand{\expec}{\operatorname{\mathbb{E}}}

\newcommand{\argmin}{\operatornamewithlimits{\arg\min}}
\newcommand{\norm}[1]{\left\lVert{\textstyle #1}\right\rVert}
\newcommand{\abs}[1]{\left\lvert{\textstyle #1}\right\rvert}
\newcommand{\reals}{\mathbb{R}}
\newcommand{\subG}{\mathcal{G}}
\newcommand{\cG}{\mathcal{G}}

\newcommand{\Hbar}{\bar{H}}

\renewcommand{\ge}{\geqslant}
\renewcommand{\geq}{\ge}
\renewcommand{\le}{\leqslant}
\renewcommand{\leq}{\le}

\newcommand{\diff}{\mathrm{d}}
\newcommand{\1}{\mathds{1}}
\newcommand{\supp}{\operatorname{supp}}

\newcommand{\taum}{\tau}

\newcommand{\aols}{\hat{\alpha}_{n}^{\mathrm{ols}}}
\newcommand{\bols}{\hat{\beta}_{n}^{\mathrm{ols}}}
\newcommand{\alasso}{\hat{\alpha}_{n}^{\mathrm{lasso}}}
\newcommand{\alslasso}{\hat{\alpha}_{n,N}^{\mathrm{lslasso}}}
\newcommand{\blasso}{\hat{\beta}^{\mathrm{lasso}}}
\newcommand{\blslasso}{\hat{\beta}^{\mathrm{lslasso}}}
\newcommand{\blasson}{\hat{\beta}_{n}^{\mathrm{lasso}}}
\newcommand{\blassonk}{\hat{\beta}_{n,k}^{\mathrm{lasso}}}

\usepackage{color}
\definecolor{auburn}{rgb}{0.43, 0.21, 0.1}
\definecolor{britishracinggreen}{rgb}{0.0, 0.26, 0.15}
\definecolor{burntumber}{rgb}{0.54, 0.2, 0.14}
\definecolor{carmine}{rgb}{0.59, 0.0, 0.09}
\definecolor{aurometalsaurus}{rgb}{0.43, 0.5, 0.5}
\definecolor{gray}{rgb}{0.4, 0.4, 0.4 }
\definecolor{cerulean}{rgb}{0.0, 0.48, 0.65}
\definecolor{ceruleanblue}{rgb}{0.16, 0.32, 0.75}
\definecolor{cobalt}{rgb}{0.0, 0.28, 0.67}
\definecolor{darkcerulean}{rgb}{0.03, 0.27, 0.49}

\newcommand{\mdf}[1]{\bgroup\color{Sepia}{#1}\egroup}

\begin{document}

\title{Control variate selection for Monte Carlo integration}

\author{Rémi Leluc \and François Portier \and Johan Segers}

\institute{R. Leluc \at
	Télécom Paris \\
	Institut Polytechnique de Paris \\
	\email{remi.leluc@gmail.com} \\
	\and
	F. Portier \at
	Télécom Paris \\
	Institut Polytechnique de Paris \\
	\email{francois.portier@gmail.com} \\
	\and
	J. Segers \at
	UCLouvain, LIDAM/ISBA \\ 
	Voie du Roman Pays 20 \\ 
	1348 Louvain-la-Neuve \\
	\email{johan.segers@uclouvain.be}
}

\date{\today}

\maketitle

\begin{abstract}
Monte Carlo integration with variance reduction by means of control variates can be implemented by the ordinary least squares estimator for the intercept in a multiple linear regression model with the integrand as response and the control variates as covariates. Even without special knowledge on the integrand, significant efficiency gains can be obtained if the control variate space is sufficiently large. Incorporating a large number of control variates in the ordinary least squares procedure may however result in (i) a certain instability of the ordinary least squares estimator and (ii) a possibly prohibitive computation time. Regularizing the ordinary least squares estimator by preselecting appropriate control variates via the Lasso turns out to increase the accuracy without additional computational cost. The findings in the numerical experiment are confirmed by concentration inequalities for the integration error.
\end{abstract}

\section{Introduction}

Whereas the basic Monte Carlo (MC) estimate {of an integral or expectation} is given by $(1/n)\sum_i f_i$, for independent and identically distributed random variables $f_i$, the control variates method is based on $(1/n)\sum _i (f_i + h_i) $, where the {variables $h_i$}, called control variates, are constructed to have zero expectation. When the controls $h_i$ have been selected or estimated properly (based on the samples $f_i$), the use of control variates might reduce the variance of the basic MC estimate significantly. The method of control variates, already used frequently to compute prices of financial derivatives \cite{glasserman:2013,gobet2010solving}, has been employed recently in many different fields of Machine Learning and Statistics. Examples include (i) \emph{reinforcement learning} and more particularly \emph{policy gradient} methods \cite{jie+p:2010,liu2018action-dependent} where the score function permits to define many control variates; (ii) inference in complex probabilistic models \citep{ranganath+g+s:2014} where the Stein method allows to define accurate control variates (see e.g., \citep{oates+g+c:2017,brosse2018diffusion,belomestny2020variance} and the references therein); (iii) gradient based \emph{optimization} \citep{wang+c+s+x:2013,gower2018tracking}, (iv) \textit{time series analysis} when approximating the characteristic function \citep{davis+s+k:2019}, and (v) semi-supervised inference \citep{zhang2019semi}.

Suppose that $m\geq 1$ control variates are available and $n\geq 1$ samples have been generated. Any linear combination of control variates can be used as a particular control variate. In terms of the variance of the estimation error, the optimal linear combination can be estimated based on the empirical risk minimization principle applied to an ordinary least squares (OLS) regression problem [see Eq.~\eqref{eq:ols} below]. This approach, referred to as OLS, is the most common implementation of the control variates method as detailed for instance in \citep[Section~8.3]{owen:13} or \citep{portier+s:2019,south+m+d+o:2018}, although other implementations are possible, see Remark~\ref{rem:ols} below.

Asymptotically, the OLS error is bounded by the MC error and is proportional to the $L_2$ approximation error of the integrand in the linear span of control variates \citep{glynn+s:2002}. In combination with well-known approximation results in $L_p$-spaces \citep{rudin:2006}, this representation of the OLS error suggests to use an increasing number of control variates. Indeed, in \citep{portier+s:2019} it is shown that when $m$ grows with $n$, the OLS error rate can be faster than $1/\sqrt{n}$.  

However, when based on a large number of control variates, the OLS suffers from two classical problems common for least squares methods: (i) numerical instabilities when the control variates are nearly collinear, and (ii) a computational complexity in $m^3 + nm^2$, which might be prohibitive. 

To deal with these two issues, it has been proposed in \cite{south+m+d+o:2018} to regularize the OLS estimate by adding a $\ell_1$-penalty term in the minimization problem, just as in the LASSO \citep{tibshirani:1996}. Simulation results in \cite{south+m+d+o:2018} show that this approach, referred to as LASSO, provides great improvements in practice. However, those practical findings are not supported by an asymptotic error rate nor by a non-asymptotic error bound.

The main objective of the paper is to provide a non-asymptotic theory for the use of control variates in Monte Carlo simulations. The contributions are as follows.

\begin{enumerate}[label=\arabic*., leftmargin=1pc]
	\item
	A \emph{new method} called LSLASSO is proposed. In the spirit of \citep{belloni+c:2013}, it consists in selecting the best control variates via the LASSO, using subsampling to decrease the computation time, and then to apply OLS with the selected controls.\\
	\item
	{\emph{Support recovery}: the LASSO is shown to select the correct control variates with large probability.} \\
	\item
	\emph{Concentration inequalities} are derived for the OLS, LASSO and LSLASSO	integration errors. The one for the OLS highlights a compromise between the approximation error of the integrand in the linear span of control variates and the multicollinearities between the control variates. The ones for {(LS)LASSO} show	significant improvements regarding the effects of multicollinearity.
\end{enumerate}

The approach {for the proofs} combines well known sub-Gaussian concentration inequalities \citep{boucheron+l+m:2013} {along with a  lower bound for the smallest eigenvalue of an empirical Gram matrix, based on a Chernoff inequality for matrices \citep[Theorem~5.1.1]{tropp2015introduction}.} 

The outline of the paper is as follows. Section~\ref{sec:MCcv} introduces the theoretical background and the different MC estimates and provides some comments about their practical implementation and some possible alternative approaches. Section~\ref{sec:ineq} contains the statements of the theoretical results. Sections~\ref{sec:simu} and~\ref{sec:bayes_inf} describe numerical experiments on artificial and real data to illustrate the finite-sample behavior of the methods. Section~\ref{sec:conclusion} concludes the main part of the paper with a discussion of avenues for further research. Section~\ref{app:aux} contains some auxiliary results, whereas the proofs of the four theorems stated in Section~\ref{sec:ineq} are given in Sections~\ref{app:th:ols} to~\ref{app:th:lslasso}.

\section{Monte Carlo integration and control variates}
\label{sec:MCcv}

\paragraph{Background.}
Let $f \in L_2(P)$ be a square integrable, real-valued function on a probability space $(\mathcal{X}, \mathcal{A}, P)$ of which we would like to calculate the integral
$$
	P(f) = \int_{\mathcal{X}} f(x) \, P(dx).
$$
The MC estimator of $P(f)$ based on independent random variables $X_1,\ldots,X_n $ {taking values in $\mathcal{X}$ and with common distribution $P$} is
$$
\hat \alpha_n ^{\mathrm{mc}} (f) 
= P_n(f) 
= \frac{1}{n}\sum_{i=1}^n f(X_i).
$$
This estimator is unbiased and has variance $ n^{-1} \sigma_0^2(f)$, where $\sigma_0^2(f) = P[(f-P(f))^2]$. 

The control variates are functions $h_1,\ldots,h_m \in L_2(P)$ with known expectations. Without loss of generality, assume that $P(h_k)=0$ for all $k\in \{1,\ldots,m\}$. Let $h = (h_1,\ldots,h_m)^T$ denote the $\reals^m$-valued function with the $m$ control variates as elements. Let $\mathcal{F}_m = \spann\{h_1,\ldots,h_m\} = \{\beta^T h : \beta \in \reals^m \}$ denote the closed linear subspace of $L_2(P)$ generated by the control variates.

For any coefficient vector $\beta = (\beta_1,\ldots,\beta_m)^T \in \mathbb{R}^m$, we have $P(f-\beta^T h) = P(f)$, so that $P_n(f-\beta^T h)$ is an unbiased estimator of $P(f)$, with variance $ n^{-1} P[(f-P(f) - \beta^T h)^2]$. Any oracle coefficient 
$$\beta^{\star}(f)\in \underset{\beta \in \mathbb{R}^m}{\arg \min} \ P[(f - P(f) - \beta^T h)^2]$$
minimizes the variance. If such a $\beta^{\star}(f)$ would be known, the resulting oracle estimator would be
\begin{ceqn}
\begin{equation}
\label{eq:aor}
\hat \alpha_n^{\mathrm{or}}(f) 
= P_n[f-\beta^{\star} (f)^T h].
\end{equation}
\end{ceqn}
By definition, the oracle estimator achieves the minimal variance $n^{-1} \sigma_ m ^2(f) $ where $\sigma_ m ^2(f) $ is the minimum value of the variance term $ \ P[(f - P(f) - \beta^T h)^2]$ with respect to $\beta$.
For any $m' \in \{0, 1, \ldots, m\}$, if we use only the first $m'$ control variates $h_1, \ldots, h_{m'}$, or even none at all in case $m' = 0$, we have $ \sigma_ {m} ^2(f)\leq \sigma_{m'}^2(f)$. In particular, if $\beta^\star(f)$ would be known, the use of control variates would always reduce the variance of the basic Monte Carlo estimator.

As $\beta^{\star}(f)^T h$ is the $L_2(P)$-projection of $f-P(f)$ on the linear vector space $\mathcal F_m$ and since the control variates are centered, $\beta^{\star }(f)$ satisfies the normal equations $P(h h^T)\beta^{\star}(f) = P(h f)$. The integral $ P(f)$ thus appears as the intercept of a linear regression model with response $f$ and explanatory variables $h_1, \ldots, h_m$, and it can be expressed as
\begin{ceqn}
\begin{equation}
\label{eq:risk_minimization}
(P(f),  \beta^\star(f)) \in 
\argmin_{(\alpha , \beta ) \in \reals \times \reals^m} 
P [ (f -  \alpha  - \beta^{T} h )^2   ]. 
\end{equation}
\end{ceqn}

The empirical risk minimization paradigm applied to the risk function on the right-hand side of \eqref{eq:risk_minimization} will lead to the OLS and LASSO estimates, to be defined further in this section. The same paradigm suggests the use of other regression methods for MC integration such as Principal Component Regression (PCR) or Ridge Regression, {which will not be considered in this paper.}

\begin{remark}[Choice of control variates] \label{rk:construction_variates}
	Which control variates work well depends on the problem. In the Black--Scholes model, for instance, an effective control variate for the price of an option is the geometric average of the price series \citep[Example~4.1.2]{glasserman:2013}). Two generic ways to construct control variates are to be noted. Whenever $P(dx) = w(x) \, Q(dx)$, where $w:\mathcal{X} \to [0, \infty)$ and $Q$ is a probability measure on $(\mathcal{X}, \mathcal{A})$, the quantity of interest is $P(f) = Q(wf)$, so that we can use control variates for $wf$ with respect to $Q$. This trick can be useful in combination with importance sampling \citep{owen+z:2000}. If $P$ has density $p$ with respect to the Lebesgue measure and if we have access to the derivatives of $p$, Stein's method might be used to build infinitely many control functions \citep{oates+g+c:2017}.
\end{remark}

\paragraph{Ordinary Least Squares Monte Carlo.}
Replacing the distribution $P$ by the sample measure $P_n$ in \eqref{eq:risk_minimization}, we obtain the OLS estimator $\hat{\alpha}_n^{\mathrm{ols}}(f)$ of $P(f)$ as a minimizer of the empirical risk 
\begin{ceqn}
\begin{equation*}
\mathcal{R}_n(\alpha,\beta) =\norm{f^{(n)}-\alpha \1_n-H \beta}_2^{2}
\end{equation*}
\end{ceqn}
 given by
\begin{ceqn}
\begin{equation}
\label{eq:ols}
\bigl(\aols(f) , \bols(f)\bigr) \in 
\argmin_{(\alpha , \beta ) \in \reals \times \reals^m}
\mathcal{R}_n(\alpha,\beta)
\end{equation}
\end{ceqn}
where $\norm{\,\cdot\,}_2$ denotes the Euclidean norm, $\1_n = (1,\ldots, 1)^T \in \reals^n$, $f^{(n)} = (f(X_1),\ldots,f(X_n))^T \in \reals^n$ and $H$ is the random $n \times m$ matrix defined by
\begin{ceqn}
\[
H 
= \bigl( h_j(X_i) \bigr)_{\substack{i = 1,\ldots,n\\j=1,\ldots,m}}.
\]
\end{ceqn}
The minimization problem in \eqref{eq:ols} can be expressed using an OLS estimate with centered variables as
\begin{ceqn}
\begin{equation}
\label{eq:ols:c}
\left\{
\begin{split}
\aols(f) &= P_n[f - \bols(f)^T h], \\
\bols(f) &\in \argmin_{\beta \in \reals^{m}} 
\norm{f_c^{(n)}-H_c \beta}_{2}^{2},
\end{split}
\right.
\end{equation}
\end{ceqn}
where $f_c^{(n)} = f^{(n)} - \1_n (\1_n^T f^{(n)})/n$ and $H_c = H - \1_n (\1_n^T H)/n$. Indeed, for fixed $\beta \in \reals^m$, the minimizer over $\alpha \in \reals$ of the objective function in \eqref{eq:ols} is just $P_n(f - \beta^T h) = P_n(f) - \beta^T P_n(h)$, and since $P_n(f) = (\1_n^T f^{(n)}) / n$ and $P_n(h) = (\1_n^T H)/n$, the equivalence of \eqref{eq:ols} and \eqref{eq:ols:c} follows.

\begin{remark}[Variations]
	\label{rem:ols}
The solution of the linear regression problem \eqref{eq:ols:c} involves the empirical covariance matrix defined by $n^{-1} H_c^T H_c = P_n(h h^T) - P_n(h) P_n(h^T)$. Using different estimates of the Gram matrix $P(h h^T)$ leads to alternative control variate MC estimates for $P(f)$ \citep{glynn+s:2002,portier+s:2019}. For fixed $m$ and as $n \to \infty$, all these estimators are consistent and asymptotically normal. The OLS estimator, however, is the only one that can integrate both the constant functions and the control functions without error.
\end{remark}

\begin{remark}[Invariance]
	\label{scale_invariance}
	The OLS estimator does not change if we replace the control variate vector $h$ by $Ah$, where $A$ is an arbitrary invertible $m \times m$ matrix. Provided the control functions are linearly independent, the property of isotropy, i.e., $P(hh^T) = I_m$, can therefore always be enforced by an appropriate linear transformation of the vector of control variates. 
\end{remark}

\begin{remark}[Computation time]
\label{rk:comput_time}
The computation time of the OLS method is of the order $nm^2+ m^3+nt$, where $n m^2$ and $m^3$ operations are needed for computing and inverting $H_c^T H_c$ respectively and where $t$ stands for the time needed to evaluate $f$. Computational benefits occur when there are multiple integrands, since the OLS estimate can be represented as $w^Tf^{(n)}$, where the weight vector $w \in \reals^n$ does not depend on the integrands \cite{portier+s:2019}. If $q$ integrals need to be evaluated, the computing time becomes $nm^2 + m^3 + qnt$, since the matrix $H_c^T H_c$ only depends on the control variates but not on the integrand.
\end{remark}

\begin{remark}[Variance reduction]\label{rk_sigma_decay} The advantage of using a given set of $m$ control variates over standard MC can be assessed through the value of the residual standard deviation $\sigma_m(f)$. In \cite{portier+s:2019}, bounds for $\sigma_m(f)$ are computed in specific examples. For instance, if $\mathcal X = [-1,1]^d$ and the $h_k$ are tensor products of Legendre polynomials, then for any $k$-times continuously differentiable function $f$ it holds that $\sigma_m (f)  = O( m^{- k / d } )$ as $m \to \infty$. This bound emphasizes the benefits of using polynomials when the integrand is regular.
\end{remark}

\paragraph{LASSO Monte Carlo.}
The LASSO, introduced in \cite{tibshirani:1996}, is a regression technique that consists in minimizing the usual least squares loss {plus an $\ell_1$-penalty term on the vector of regression coefficients.} In contrast with OLS, {the LASSO usually produces a vector with many zero coefficients}, meaning that the corresponding variables are no longer included in the predictive model. The LASSO thus achieves estimation and variable selection at the same time. As the use of control variates in MC integration is linked with regression, the LASSO can take advantage from situations where many control variates are present but not all of them are useful. 

The LASSO estimator $\hat \alpha_n^{\mathrm{lasso}}(f)$ of $P(f)$ follows from adding a $\ell_1$-penalization to the objective function in~\eqref{eq:ols}. It is formally defined as
\begin{ceqn}
\begin{equation*}
\bigl(\hat \alpha_n^{\mathrm{lasso}}(f), \hat \beta_n^{\mathrm{lasso}}(f)\bigr)
\in \argmin\limits_{(\alpha, \beta) \in \reals \times \reals^{m}} \frac{1}{2n} \mathcal{R}_n(\alpha,\beta)
+ \lambda \norm{\beta}_{1}
\end{equation*}
\end{ceqn}
where $\norm{\,\cdot\,}_1$ denotes the $\ell_1$-norm on Euclidean space. By the same argument used to justify the equivalence of \eqref{eq:ols} and \eqref{eq:ols:c}, the LASSO can be based on centered variables via
\begin{ceqn}
\begin{equation}
\left\{
\begin{split}
\hat \alpha_n^{\mathrm{lasso}}(f) 
&= P_n [f - \hat \beta_n^{\mathrm{lasso}}(f)^T h], \\
\label{eq:lasso}
\hat \beta_n^{\mathrm{lasso}}(f) 
&\in \argmin_{\beta \in \reals^{m}}
\frac{1}{2n} \norm{f_c^{(n)} - H_c \beta}_{2}^{2} 
+ \lambda\norm{\beta}_{1}.
\end{split}
\right.
\end{equation}
\end{ceqn}

\begin{remark}[Computation]\label{rk:comput_lasso}
	For the practical implementation of the LASSO, it is commonly recommended to first center and rescale the explanatory variables empirically \citep[section 2.2]{tibshirani+w+h;2015}. The centering by the sample mean is taken care of in \eqref{eq:lasso}. However, for ease of presentation, no empirical rescaling of the control variates is considered in the theoretical analysis. This is in line with the approach proposed in \citep[Chapter~11]{tibshirani+w+h;2015}. Still, such rescaling is done in the simulation experiments reported in Section~\ref{sec:simu}.
\end{remark}

\begin{remark}[Computation time]\label{rk:comput_time_lasso}
The LASSO solution is usually computed approximately by \emph{cyclical coordinate descent}. At each iteration, this algorithm minimizes \eqref{eq:lasso} with respect to a single coordinate, say $\beta_k$, while considering other coordinates, $\beta_{(-k)}\in \mathbb R ^{m-1} $, as constant. This one-dimensional optimization problem has an explicit argmin. Let $H_{c,k}$ be the $k$-th column of $H_c$ that has been normalized such that $\|H_{c,k}\|_2 = 1$ (as indicated in the previous remark). The argmin is then simply given by $\eta_{\lambda } ( \langle z_k, H_{c,k}  \rangle )$ where $ z_k = f_c^{(n)} - H_{c,(-k)} \beta_{(-k)} $, 	$H_{c,(-k)}$ is obtained by removing $H_{c,k}$ from $H_c$ and $\eta$ is the \emph{soft-thresholding function} \citep[Section~2.4, Eq. (2.14)]{tibshirani+w+h;2015}. Since $n$ operations are needed to update $z_k$ and the same number is needed to compute the scalar product, the LASSO requires only $n D  + n t$ operations, where $D$ stands for the number of iterations conducted in the cyclical coordinate descent and $t$ represents the time needed to evaluate $f$. The value of $D$ is often imposed by a stopping rule within the algorithm but it could also be fixed by the user in order to control the computing time. The selection of the next coordinate $k$ to update can be done cyclically or at random.
\end{remark}

\paragraph{LSLASSO Monte Carlo.}
The application of ordinary least squares after model selection by the LASSO has been recently studied in \cite{belloni+c:2013}. They show, in the  setting of nonparametric regression, that OLS post-LASSO, which is also known under the name LSLASSO, performs better than the LASSO in terms of rate of convergence. Motivated by this result we propose to first use the LASSO to select the active variables among a large number of control variates and then to compute {the OLS estimate} using only the variables selected at the previous stage. We refer to this approach as the LSLASSO. To decrease the computation time when the dimensions involved in the problem, either $n$ or $m$, are large, we recommend to use sub-sampling of {size $N$ smaller than $n$} when conducting the first step.

The \emph{active set} associated to the coefficient $\beta \in \reals^m$ is $\supp(\beta) = \{ j = 1,\ldots, m\,:\, \beta_j \neq 0 \}$. Let {$\hat{S}_N = \supp(\hat{\beta}_N^{\mathrm{lasso}}(f))$} denote the active set of control variates based on the LASSO coefficient vector defined as in \eqref{eq:lasso} but using only the first $N$ random variables $X_1, \ldots, X_N$ generated. 
The LSLASSO estimate $\hat \alpha_n^{\mathrm{lslasso}}(f)$ of $P(f)$ is then defined as the OLS estimate in~\eqref{eq:ols} based on the full sample $X_1, \ldots, X_n$ but using only the control variates $h_j$ restricted to $j \in \hat{S}_N$, that is,
\begin{ceqn}
	\begin{equation*}
\bigl(\hat \alpha_n^{\mathrm{lslasso}}(f) , \hat \beta_n^{\mathrm{lslasso}}(f)\bigr) \in 
\argmin_{(\alpha , \beta ) \in \reals \times \reals^{\hat \ell}}
\norm{f^{(n)} - \alpha \1_n - H^{(n)}_{\hat S_N} \beta}_{2}^{2}
\end{equation*}
\end{ceqn}
where $H^{(n)} _ {\hat S_N} $ is the $n \times \hat{\ell}$ matrix $ ( h_j(X_i) )_{i = 1,\ldots,n,\,j\in \hat S_N}$ and $\hat \ell $ is the cardinality of $\hat{S}_N$.

\begin{remark}[Computation time]\label{rk:comput_time_lslasso}
The number of operations needed for the LSLASSO is of the order $ND + n \hat{\ell}^2 + \hat{\ell}^3 + nt$, combining the cost of selecting the control variates on the subsample of size $N$ via cyclical coordinate descent as in Remark~\ref{rk:comput_time_lasso} and running the OLS estimate based on the selected control variates for the full sample of size $n$ as in Remark~\ref{rk:comput_time}.
\end{remark}

\section{Non-asymptotic bounds}
\label{sec:ineq}
To derive concentration inequalities for the errors of the estimators proposed in Section~\ref{sec:MCcv}, we use the notion of sub-Gaussianity as defined for instance in \citep[Section~2.3]{boucheron+l+m:2013}. Recall that the moment generating function of a centered Gaussian random variable with variance $\sigma^2$ is equal to $\lambda \mapsto \exp(\lambda^2 \sigma^2 / 2)$.

\begin{definition}
	A centered random variable $Y$ is \emph{sub-Gaussian} with variance factor $\tau^2 > 0$, notation $Y \in \subG(\tau^2)$, if and only if $\log \expec[\exp(\lambda Y)] \le \lambda^2 \tau^2 / 2$ for all $\lambda \in \reals$.
\end{definition}

If $Y \in \subG(\tau^2)$, then necessarily $\var(Y) \le \tau^2$ \citep[Exercise~2.16]{boucheron+l+m:2013}. Chernoff's inequality provides exponential bounds on the tails of sub-Gaussian random variables. Moreover, the sum of independent sub-Gaussian variables is again sub-Gaussian. Centered, bounded random variables taking values in an interval $[a, b]$ are sub-Gaussian with variance factor at most $(b-a)^2/4$ \citep[Lemma~2.2]{boucheron+l+m:2013}.

The concentration inequalities for the various Monte Carlo methods with control variates will be largely due to the following assumption that requires the residuals to be sub-Gaussian.

\begin{assumption}[Sub-Gaussian residuals]\label{ass_1_sub_gauss}
	The residual function $\epsilon = f - P(f) - \beta^{\star}(f)^T h$ satisfies $\epsilon \in \subG(\taum^2)$ for some $\taum > 0$, that is, $\int_{\mathcal{X}} \exp\{\lambda \epsilon(x)\} \, P(dx) \le \exp(\lambda^2 \taum^2/2)$ for all $\lambda \in \reals$.
\end{assumption}

The estimation error of the oracle estimator in \eqref{eq:aor} is just $\hat{\alpha}_n^{\mathrm{or}}(f) - P(f) = P_n(\epsilon) = n^{-1} \sum_{i=1}^n \epsilon(X_i)$. Under Assumption~\ref{ass_1_sub_gauss}, this is a sub-Gaussian variable with variance factor $\tau^2/n$. Chernoff's inequality \citep[p.~25]{boucheron+l+m:2013} then implies that for all $\delta \in (0, 1)$ and all integer $n \ge 1$, with probability at least $1-\delta$,
\begin{ceqn}
\begin{equation}
\label{ineq:aor}
\abs{\hat{\alpha}_n^{\mathrm{or}}(f)- P(f)}
\le \sqrt{2 \log(2/\delta)} \frac{\taum}{\sqrt{n}}.
\end{equation}
\end{ceqn}
This concentration inequality provides a baseline when the best possible control variate in the space $\mathcal{F}_m$ is selected. The case $m = 0$ also covers the basic MC method: in that case, $\tau^2$ is the variance factor of the sub-Gaussian variable $f - P(f)$ on $(\mathcal{X}, \mathcal{A}, P)$.

\begin{assumption}[Bounded control variates]
	\label{ass_2_bounded_control}
	{The control variates $h_1, \ldots, h_m \in L_2(P)$ are uniformly bounded. Put $U_h := \max_{j=1,\ldots,m} \sup_{x \in \mathcal{X}} \abs{h_j(x)}$}.
\end{assumption}

{For a symmetric real matrix $A$, let $\lambda_{\min}(A)$ and $\lambda_{\max}(A)$ denote its smallest and largest eigenvalues, respectively.}

\begin{assumption}[Linear independence of control variates]
	\label{ass_1_ev_cond}
	The control variates $h_1$, $\ldots$, $h_m \in L_2(P)$ are linearly independent. As a consequence, the $m \times m$ Gram matrix $G := P(hh^T)$ is positive definite and its smallest eigenvalue $\gamma := \lambda_{\min}(G)$ is positive.
\end{assumption}

Consider the ortho-normalized vector of control variates $\hbar = (\hbar_1, \ldots, \hbar_m)^\top = G^{-1/2} h$ and put
\begin{ceqn}
\begin{equation}
\label{eq:B}
B
= \sup_{x \in \mathcal{X}} h(x)^\top G^{-1} h(x)
= \sup_{x \in \mathcal{X}} \hbar(x)^\top \hbar(x),
\end{equation}
\end{ceqn}
a finite quantity by Assumptions~\ref{ass_2_bounded_control} and~\ref{ass_1_ev_cond}. The error OLS estimation error is subject to the following concentration bound.

\begin{theorem}[Concentration inequality for OLS]
	\label{th:ols}
	Suppose Assumptions~\ref{ass_1_sub_gauss},~\ref{ass_2_bounded_control} and~\ref{ass_1_ev_cond} hold. Then for all $\delta \in (0, 1)$ and all integer $n$ such that
	\begin{ceqn}
	\begin{equation*}
	n \ge \max\bigl(18 B \log(4m / \delta), \; 75 m \log(4/\delta) \bigr)
	\end{equation*}
	\end{ceqn}
	we have, with probability at least $1 - \delta$,
	\begin{multline}
	\label{ineq:aols}
	\abs{\aols(f) - P(f)}
	\le \sqrt{2 \log(8/\delta)} \frac{\tau}{\sqrt{n}} 
	  \\ 
	+ 58 \sqrt{B m \log(8m/\delta) \log(4/\delta)}
	\frac{\tau}{n}.
	\end{multline}
\end{theorem}

Compared to the bound \eqref{ineq:aor} for the oracle estimator, the bound \eqref{ineq:aols} for the OLS estimator has an additional term. This term is due to the additional learning step that is needed to estimate the optimal control variate. 

\begin{remark}[On the factor $B$]
	\label{rem:factorB}
Defined as the supremum of the \emph{leverage function} $q_n$ in \cite[Eq.~(14)]{portier+s:2019}, the quantity $B$ plays an important role in our analysis as well as in other regression studies \citep{hsu2012random,newey1997convergence}. Just as the OLS estimate (see Remark \ref{scale_invariance}), the quantity $B$ remains invariant by invertible linear transformation of the control variates.  We have
\begin{ceqn}
\begin{equation*}	
m\le	B 
	\le \sup_{x \in \mathcal{X}} h^\top(x) h(x) / \gamma 
	\le m U_h^2/\gamma .
\end{equation*}
\end{ceqn}
\end{remark}

\begin{remark}[On the parameters $\tau$ and $\gamma$]\label{rk:param_meaning}
The parameter $\tau$ in Assumption~\ref{ass_1_sub_gauss} is by definition an upper bound of the residual variance $\sigma^2_m(f)$. In many situations, its value is not too far from $\sigma^2_m(f)$. Hence, $\tau$ should capture the adequacy between the control variate space and the integrand $f$ and should decrease with $m$. The full rank condition expressed in Assumption~\ref{ass_1_ev_cond} is not crucial as one could work with the Moore--Penrose inverse when solving \eqref{eq:ols:c}. More importantly, a large value of the minimal eigenvalue $\gamma$ of the Gram matrix $G$ reflects that the OLS problem is well-conditioned, enhancing numerical stability. As control functions are added, rows and columns are added to $G$ and so $\gamma$ cannot increase. For the Fourier basis in Example~\ref{ex:fourier}, we have $\gamma=1$, while for the Legendre polynomials in Example~\ref{ex:legendre}, we have $\gamma \simeq 1/m$.
\end{remark}

\begin{remark}[Link with OLS prediction risk analysis]
The approach taken in the proof of Theorem \ref{th:ols} requires to bound what is called the prediction risk, defined as $\|G^{1/2} ( \bols(f) -\beta^\star (f))  \|_2 $. With probability greater than $1- \delta$, we obtain an upper bound of order $\sqrt{B\tau^2\log(m/\delta) / n}  $ on the prediction risk. This makes our approach comparable to the one of the recent study \citep{hsu2012random} where concentration bounds for the OLS prediction risk (and ridge) with random design are established. In contrast to their bound, our bound involves the quantity $B$ which shares the same invariant property as the OLS estimate and we don't require the noise to be sub-Gaussian \textit{conditionally on the covariate} but just sub-Gaussian which is weaker.
\end{remark}

\begin{remark}[Rates]
	\label{rem:rate}
	Consider an asymptotic set-up where the number of control variates $m$ tends to infinity with the Monte Carlo sample size $n$. The OLS method improves upon the basic MC method ($m = 0$), which has rate $1/\sqrt{n}$, as soon as $\tau + \tau \sqrt{{mB} \log(m)   / {n}} \to 0$. To recover the same order as the one of the oracle estimator $\hat{\alpha}_{n}^{\mathrm{or}}(f)$, which has rate $\tau / \sqrt{n}$, one must have ${mB} \log(m)  = O(n)$ as $n \to \infty$, that is, $m$ must not be too large compared to $n$.
\end{remark}

\begin{remark}[{Leverage condition}]
Theorem \ref{th:ols} may be seen as a non-asymptotic version of  the asymptotic results provided in \cite{portier+s:2019} in which the \textit{leverage condition},  $\sup \{ h(x)^T G^{-1} h(x) : x \in \mathcal{X} \} = o(n/m)$, is required to obtain a similar (asymptotic) bound (see Theorem~1 therein) as the one of Theorem \ref{th:ols}. In the present non-asymptotic version, the \textit{leverage condition} is expressed through $mB$ 
when requiring that $18 B \log(4m / \delta) \le n $.
\end{remark}
LASSO takes advantage of \emph{sparse} regression models. A regression model is sparse whenever many of the coefficients of the parameter vector $\beta$ are equal to zero, i.e., many of the covariates are useless to predict the output {in the presence of the other covariates}. 
 The number of elements in the active set of the vector of regression coefficients $\beta^\star(f)$,
 \begin{ceqn}
\[
	S^\star := \supp(\beta^\star(f)),
\]
\end{ceqn}
is denoted by $\ell^\star := \abs{S^\star}$ and quantifies the level of sparsity associated to the regression model. To avoid trivialities, we tacitly assume that $S^\star$ is non-empty, so $\ell^\star \ge 1$. The factor $\ell^\star$ represents the level of sparsity of $f$ with respect to the control functions and plays an important role in describing the benefits of the LASSO over the OLS. No assumption is made on $\ell^\star$, which could be any integer in $\{1,\ldots, m\}$.

We follow the approach presented in \citep[Section~11.4.1]{tibshirani+w+h;2015} (see also \cite{bickel2009simultaneous,van2009conditions}), in which the analysis of the LASSO is carried out using a \emph{restricted eigenvalue condition}. For a vector $\beta \in \reals^m$ and for a non-empty set $S \subset \{1, \ldots, m\}$, write $\beta_S = (\beta_k)_{k \in S}$, seen as a (column) vector in $\reals^{\abs{S}}$. Define a collection of cones of interest. For $\alpha>0$ and $S \subset \{1,\ldots,m\}$, we set $\overline{S} = \{1,\ldots,m\} \setminus S$ and
\begin{ceqn}
\begin{equation*}
	\mathcal{C}(S; \alpha) 
	= \{ u\in \mathbb R^m \,:\,  \norm{u_{\overline{S}}}_1 \leq \alpha \norm{u_S}_1 \}.
\end{equation*}
\end{ceqn}

\begin{assumption}[Restricted eigenvalue condition]
	\label{ass_4_ev_cond_lasso}
There exists $\gamma^\star >0$ such that $u^T G u  \ge \gamma^\star \norm{{u}}_2^2  $ for all  $u \in \mathcal C (S^\star ; 3)$.
\end{assumption}

In practice, we do not know the active set $S^\star$, so the only way to ensure Assumption~\ref{ass_4_ev_cond_lasso} is to make sure all control variates $h_1, \ldots, h_m$ are linearly independent. The practical value of the assumption is that $\gamma^\star \ge \gamma$, yielding sharper bounds below.

Recall that the $\ell_1$-penalty of the LASSO is weighted by a regularization parameter $\lambda > 0$.

\begin{theorem}[Concentration inequality for LASSO]
	\label{th:lasso}
Suppose Assumptions~\ref{ass_1_sub_gauss},~\ref{ass_2_bounded_control} and~\ref{ass_4_ev_cond_lasso} hold. Introduce $\xi = \ell^\star(U_h^2/\gamma^\star)$. Then for all $\delta \in (0, 1)$ and all integer $n$ such that
\begin{align*}
n &\geq \max\left( 8 \xi^2 \log(8m^2/\delta);128 \xi \log(8m/\delta) \right),\\  
\lambda &\ge 7 U_h \sqrt{\log(8m/\delta)} \tau / \sqrt{n}
\end{align*}	
we have, with probability at least $1 - \delta$,
	\begin{multline}
	\label{ineq:alasso_general}
	\abs{\alasso(f) - P(f)}
	\le \sqrt{2 \log(8/\delta)} \frac{\tau}{\sqrt{n}} \\
	+ 68 \lambda \ell^\star \sqrt{\log(8m/\delta)} \frac{U_{h}/\gamma^\star}{\sqrt{n}}.
	\end{multline}
	For $\lambda$ equal to the lower bound, we have on the same event
	\begin{multline}
	\label{ineq:alasso}
	\abs{\alasso(f) - P(f)}
	\le \sqrt{2 \log(8/\delta)} \frac{\tau}{\sqrt{n}}  \\
	+ 476 \ell^\star \log(8m/\delta)(U_{h}^2/\gamma^\star) \frac{\tau}{n}.
	\end{multline}
\end{theorem}

\begin{remark}[LASSO vs OLS]
	\label{rem:comparison_OLS_LASSO}
The benefits of LASSO over OLS can be observed by comparing the bounds in~\eqref{ineq:aols} and~\eqref{ineq:alasso}. The total number $m$, of control functions has been replaced by the active number $\ell^\star$ of such functions. Further, because $\Gamma_{S^\star} = \{u\in \mathbb R^p \,: \, \|u\|_2 = 1, \, u \in \mathcal C (S^\star ; 3)\}$ is included in the unit sphere, $\gamma^\star = \inf_{ u\in \Gamma_{S^\star} } u^T G u$ in Assumption~\ref{ass_1_ev_cond} is at least as large as the smallest eigenvalue of $G$, $\gamma = \inf_{ \|u\|_2 = 1 } u^T G u$ in Assumption~\ref{ass_4_ev_cond_lasso}.
\end{remark}

The theoretical analysis of the LSLASSO estimator depends on the success of the LASSO-based model selection, i.e., the LASSO needs to correctly recover all the components of the true model. To ensure this selection step, the restricted eigenvalue condition is replaced by the two following ones.

\begin{assumption}[Linear independence of active functions]
	\label{ass_5_ev_cond_lasso}
	The active control variates $h_k$, $k \in S^\star$, are linearly independent. As a consequence, the $\ell^\star \times \ell^\star$ Gram matrix $G_{S^\star} = P(h_{S^\star} h_{S^\star}^T)$ is positive definite and its smallest eigenvalue $\gamma^{\star\star} := \lambda_{\min}(G_{S^\star})$ is strictly positive.
\end{assumption}

Note that because $ \{u\in \mathbb R^p \,: \, \|u\|_2 = 1, \, \forall k \notin S,\, u_k = 0\}\subset \Gamma_S $ (introduced in remark \ref{rem:comparison_OLS_LASSO}), we have that $\gamma^{\star\star}\geq \gamma^\star$. Finally, it is required that that the active control functions are orthogonal, in $L_2(P)$, to the inactive ones.

\begin{assumption}[Orthogonality]
	\label{ass_orth_control}
	We have $P(h_j h_k) = 0$ for all $j \in \{1,\ldots,m\} \setminus S^\star$ and all $k \in S^\star$.
\end{assumption}

{Since we do not know $S^\star$ in practice, the way to ensure Assumption~\ref{ass_orth_control} is by making all control variates orthogonal: $P(h_j h_k) = 0$ for all $j,k \in \{1, \ldots, m\}$. The Gram matrices $G$ and $G^\star$ are then diagonal. In the absence of zero control variates, Assumptions~\ref{ass_1_ev_cond} and~\ref{ass_4_ev_cond_lasso} are then satisfied as well, with $\gamma^{\star\star} = \min_{k \in S^\star} P(h_k^2) \ge \min_{k = 1, \ldots, m} P(h_k^2) = \gamma > 0$.}

\begin{theorem}[Support recovery of LASSO]
	\label{th:lasso_support_recovery}
	Suppose Assumptions~\ref{ass_1_sub_gauss}, \ref{ass_2_bounded_control}, \ref{ass_5_ev_cond_lasso} and~\ref{ass_orth_control} hold. Then for all $\delta \in (0, 1)$, all integer $n$ such that
	\begin{ceqn}
	\begin{align*}
	n \ge 70 (\ell^\star U_{h}^2/\gamma^{\star\star})^2 \log(10 \ell^\star m/\delta),
	\end{align*}
	\end{ceqn}
	and all $\lambda$ such that
	\begin{ceqn}
	\begin{equation}
	\label{eq:lambda:interval}
	13 U_{h} \sqrt{\log(10m/\delta)} \frac{\tau}{\sqrt{n}} \le \lambda \le \frac{\gamma^{\star\star}}{3\sqrt{\ell^\star}} \underset{k \in S^\star}{\min} |\beta_k^\star(f)|,
	\end{equation}
	\end{ceqn}
	it holds that, with probability at least $1-\delta$, the LASSO based solution $\hat{\beta}_n^{\mathrm{lasso}}(f)$ is unique and the true active set is recovered,
	$\supp(\hat \beta_n^{\mathrm{lasso}}(f)) = S^\star$.
\end{theorem}

The upper and lower bounds on $\lambda$ in \eqref{eq:lambda:interval} must not contradict each other, and this effectively implies an additional lower bound on $n$. Define 
$B^\star = \sup_{x \in \mathcal{X}} h_{S^\star}^T(x) G_{S^\star}^{-1} h_{S^\star}(x)$ and note that
\begin{ceqn}
\begin{align}
	B^\star 	\le \lambda_{\max}(G_{S^\star}^{-1}) 
	\sup_{x \in \mathcal{X}} h_{S^\star}^T(x) h_{S^\star}(x)  
	\le \ell^\star U_h^2 / \gamma^{\star\star}.
\end{align}	
\end{ceqn}

\begin{theorem}[Concentration inequality for LSLASSO]
	\label{th:lslasso}
	Suppose Assumptions~\ref{ass_1_sub_gauss}, \ref{ass_2_bounded_control}, \ref{ass_5_ev_cond_lasso} and \ref{ass_orth_control} hold. Write $\xi^\star = \ell^\star(U_h^2/\gamma^{\star\star})$. Then for all $\delta \in (0, 1)$ and all integer $N \in \{1, \ldots, n\}$ such that 

	\begin{ceqn}
	\begin{align*}
	N \ge 75 {\xi^\star}^2 \log(20 \ell^\star m/\delta),
	\end{align*}
	\end{ceqn}
	and all $\lambda$ such that 
	\begin{ceqn}
	\begin{equation*}
	13 U_{h} \sqrt{\log(20m/\delta)} \frac{ \tau }{ \sqrt{N}} \le \lambda \le \frac{\gamma^{\star\star}}{3\sqrt{\ell^\star}} \underset{k \in S^\star}{\min} |\beta_k^\star(f)|,
	\end{equation*}
	\end{ceqn}
	we have, with probability at least $1 - \delta$,
	\begin{multline}
	\label{ineq:alslasso}
	\abs{\hat \alpha_n^{\mathrm{lslasso}}(f) - P(f)}
	\le \sqrt{2 \log(16/\delta)} \frac{\tau}{\sqrt{n}} \\
	+ 58 \sqrt{B^\star \ell^\star \log(16 \ell^\star/\delta) \log(8/\delta)} \frac{\tau}{n}.
	\end{multline}
\end{theorem}

The logic behind Theorem~\ref{th:lslasso} is that, by Theorem~\ref{th:lasso_support_recovery}, the active set $\hat{S}_N = \supp(\hat{\beta}_N^{\mathrm{lasso}}(f))$ identified by means of the subsample of size $N$ is equal to the true active set $S^\star = \supp(\beta^\star(f))$ with large probability. On the event that the two sets coincide, the LSLASSO estimator is then the same as the OLS estimator based on the active control variates only, and the error bound follows from Theorem~\ref{th:ols}. In practice, it turns out that LSLASSO works well even when the true active set is not identified perfectly. However, to show this formally remains an open problem.

The assumptions and concentration inequalities in our theorems feature explicit rather than generic constants. Although we have worked hard to keep these constants under control [see in particular the proof of Lemma~\ref{lem:subG} as well as Step~6(ii) in the proof of Theorem~\ref{th:lasso_support_recovery}], it is likely that, at the cost of lengthier computations, sharper constants can still be found.

\begin{remark}[Bounded control variates]\label{rk_bounded_control_variates}
In Assumption~\ref{ass_2_bounded_control}, the control variates were assumed to be bounded. Even if this assumption is valid for the two classic families in Examples~\ref{ex:fourier} and~\ref{ex:legendre} below, it might fail when control variates are produced with the Stein's method as suggested in Remark~\ref{rk:construction_variates}. The boundedness assumption is needed to keep the same variance factor $\tau^2$ in the sub-Gaussian property of both variables $\epsilon(X_1)$ and $\epsilon(X_1) h(X_1) $; see, e.g., Step~3.2 in the proof of Theorem~\ref{th:ols} or Equation~\eqref{eq:2_heps_lasso} in the proof of Theorem~\ref{th:lasso}. Avoiding this assumption is thus possible at the price of more specific assumptions on the sub-Gaussianity of $\epsilon(X_1) h(X_1)$. Note finally that (different) asymptotic results are valid for unbounded control variates \cite{portier+s:2019}.
\end{remark}

\begin{remark}[Overfitting]\label{rk_overfitting}
Theorems~\ref{th:lasso} and~\ref{th:lslasso} advocate the use of the LASSO in favor of the OLS in scenarios where $\ell^*$ is smaller than $m$ or in the presence of collinearities in the design matrix making the parameter $\gamma$ close to zero; see also Remark~\ref{rem:comparison_OLS_LASSO}. Another notable advantage of the (LS)LASSO and more generally of penalization methods, is the ability to prevent over-fitting. This occurs when the number of control variates $m$ is large compared to the Monte Carlo sample size $n$ or, more generally, when the approximation space is large compared to the sample size. While the theory developed here is unable to address such phenomena, one of the objectives of the numerical experiments conducted in the next section is to empirically demonstrate the superior performance of the LASSO-based methods even in the absence of sparsity.
\end{remark}

To illustrate the application of our results in a standard framework, we consider two classic families of control functions, the Fourier basis and the Legendre polynomials.

\begin{example}[Fourier basis]
	\label{ex:fourier}
	On $\mathcal{X}=[0,1]$ equipped with the uniform distribution $P$, let $h_j(x)$ be equal to $\sqrt{2}\cos((j+1)\pi x)$ is $j$ is odd and to $\sqrt{2}\sin(j\pi x)$ is $j$ is even. The Fourier basis is orthonormal so that the Gram matrix is the identity, $G=I_m$, and $\gamma=\gamma^\star= \gamma^{\star\star}=1$. The {cosine and sine functions} being bounded by 1, a uniform bound is $U_h = \sqrt{2}$, which implies $B \leq 2m, B^\star \leq 2\ell^\star$. Under the proper assumptions, we get {from Theorems~\ref{th:ols} and~\ref{th:lslasso}} that with probability at least $1-\delta$, since $58\sqrt{2}<83$,
\begin{multline*}
    \abs{\aols(f) - P(f)}
	\le \sqrt{2 \log(8/\delta)} \frac{\tau}{\sqrt{n}} \\
	+ 83 m\sqrt{\log(8m/\delta)\log(4/\delta)}\frac{\tau}{n}
\end{multline*}
and
\begin{multline*}
	\abs{\hat \alpha_n^{\mathrm{lslasso}}(f) - P(f)}
	\le \sqrt{2 \log(16/\delta)} \frac{\tau}{\sqrt{n}} \\ 
	+ 83 \ell^\star \sqrt{\log(16 \ell^\star/\delta) \log(8/\delta)}\frac{\tau}{n}.
\end{multline*}
\end{example}

\begin{example}[Legendre polynomials]
	\label{ex:legendre}
	Suppose that $h_j = L_j$ is the Legendre polynomial of degree $j \in \{1,\ldots,m\}$. The Legendre polynomials are orthogonal on $\mathcal{X}=[-1,1]$ with respect to the uniform distribution $P$ and satisfy $|L_j(x)|\leq 1$ for $x \in [-1,1]$ with $L_j(1)=1$ and
	\begin{ceqn}
\[
	\int_{-1}^1 L_i(x) L_j(x) \, \diff x 
	= \frac{2}{2j+1}\delta_{ij}.
\]
\end{ceqn}
The Gram matrix $G=P(hh^T)$ is diagonal with entries $1/(2j+1)$, so the minimum eigenvalue is $\gamma = 1/(2m+1)$ and a uniform bound is $U_h=1$. Consequently, $B \leq 2m+1$. Similarly, considering only active control variates, we have $U_h^\star=1$, while the smallest eigenvalue, $\gamma^{\star\star}$, of $G_{S^\star}$ satisfies $ 1/(2m+1) \le \gamma^{\star\star} \le 1/(2\ell^\star +1) $. Under suitable assumptions, we get {from Theorems~\ref{th:ols} and~\ref{th:lslasso}} that with probability at least $1-\delta$,
\begin{multline*}
    \abs{\aols(f) - P(f)}
	\le \sqrt{2 \log(8/\delta)} \frac{\tau}{\sqrt{n}} \\
	+ 58 \sqrt{(2m+1) m \log(8m/\delta) \log(4/\delta)} \frac{\tau}{n},
\end{multline*}
\begin{multline*}
	\abs{\hat \alpha_n^{\mathrm{lslasso}}(f) - P(f)}
	\le \sqrt{2 \log(16/\delta)} \frac{\tau}{\sqrt{n}} \\
	+ 58 \sqrt{(2\ell^\star+1) \ell^\star \log(16 \ell^\star/\delta) \log(8/\delta)} \frac{\tau}{n}.
\end{multline*}
Compared to the Fourier basis, the improvement of LSLASSO over the OLS estimator is not only related to the number of active varables $\ell^*$ compared to $m$ but also to the place of the active variables within the set of Legendre polynomials.
\end{example}
  
\section{Numerical illustration}
\label{sec:simu}

To compare the finite-sample performance of the various control variate methods, we consider synthetic data examples involving the standard integration problem over the unit cube $[0,1]^d$. The goal is to compute $\int_{[0,1]^{d}} f(x) \, \mathrm{d} x$. We shall consider various dimensions $d\geq 1$, different integrands $f:\mathbb [0,1]^d\to \mathbb R$, and several choices for the Monte Carlo sample size, $n$, and the number of control variates, $m$. We shall focus on difficult situations where $d$ is relatively large compared to $n$.
In Section~\ref{sec:bayes_inf}, we turn to real data examples in the context of Bayesian inference. For the sake of reproducibility, the data and Python code are available online\footnotemark\footnotetext{https://github.com/RemiLELUC/ControlVariateSelection.git}.

\paragraph{Methods in competition.}
We consider all the methods presented in Section~\ref{sec:MCcv} with two different strategies regarding the sub-sample size used to compute the active set in LSLASSO. The methods in competition are OLS, LASSO, LSLASSO (sub-sample size $N=n$)  and LSLASSOX (sub-sample size $N = \lfloor 15\sqrt{n} \rfloor$). The latter choice accelerates the computation in a substantial manner without deteriorating too much the support recovery property of the LASSO. For synthetic data, because the integration domain is the unit cube $[0,1]^d$, Quasi-Monte Carlo (QMC) methods \citep{caflisch1998monte} are suitable for comparison. We run such methods in the experiments with two classical low-discrepancy sets of particles, namely Halton and Sobol sequences.

\paragraph{On the choice of $\lambda$.}
In the LASSO-step of LSLASSO(X), the choice of the regularization parameter $\lambda$ is essential since it controls the number of active variables. It is common to tune this parameter using $K$-fold cross-validation at the price of additional computations. This method, presented in general form in Algorithm~\ref{alg:k_fold}, uses the prediction error of the underlying regression problem as a proxy to calibrate the control variates estimate. In Algorithm~\ref{alg:k_fold}, the ``data'' $X$ correspond to the matrix $H$ of observed  control variables and the ``labels'' $y$ to the vector $f^{(n)}$ of observed function values. The method is computationally expensive, partitioning the training set in several folds and solving many regression problems for every value of $\lambda$ in a given grid.

\begin{algorithm}[!h]
\caption{K-fold cross-validation}
\algsetup{linenodelimiter=.}
\begin{algorithmic}[1]
\REQUIRE data $X$, labels $y$, grid search $\lambda_{\text{grid}}$, $n$, $K$.
\setstretch{1.0}
\STATE Divide $\{1,\ldots,n\}$ into $K$ folds $F_1,\ldots,F_K$.
\STATE \textbf{For} $k=1,\ldots,K$
\STATE \quad Set training folds $F_{-k}=\{F_1,\ldots,F_{k-1},F_{k+1},\ldots,F_K\}$.
\STATE \quad \textbf{For} $\lambda \in \lambda_{\text{grid}}$
\STATE \quad \quad Compute estimate $\hat{\beta}_{\lambda}^{-k}$ on training set.
\STATE \quad \quad Compute test error $e_k(\lambda) = \sum_{i \in F_k} (y_i - x_i^T \hat{\beta}_{\lambda}^{-k})^2$.
\STATE \textbf{For} $\lambda \in \lambda_{\text{grid}}$
\STATE \quad Compute average error $CV(\lambda) = \frac{1}{n} \sum_{k=1}^K e_k(\lambda)$.
\STATE \textbf{Return} $\hat{\beta}_{\lambda^\star}$ with $\lambda^\star \in \argmin_{\lambda \in \lambda_{\text{grid}}} CV(\lambda)$.
\end{algorithmic}
\label{alg:k_fold}
\end{algorithm}

To accelerate the computations, we suggest a new method based on a dichotomic search. Motivated by Eq.~\eqref{ineq:aols} and Remark~\ref{rem:rate}, the value of $\lambda$ is tuned such that the number of selected control variates is of the order $\sqrt{n}$, which is the order obtained for $m$ when equating the two terms in \eqref{ineq:aols} with $B=m$. Specifically, we enforce the number of activated control functions to lie in the range $[c_1 \sqrt{n}, c_2 \sqrt{n}]$ for constants $0 < c_1 < c_2$ to be chosen (see below). This choice offers two advantages. On the one hand, the upper bound $c_2 \sqrt{n}$ ensures that the number of selected control variates is relatively small compared to the sample size $n$, promoting stability and fast computation in the final OLS step. On the other hand, the lower bound $c_1 \sqrt{n}$ reduces the risk of excluding relevant control variates.

The full procedure for the selection of the regularization parameter using a dichotomic search is described below in Algorithm \ref{alg:dicho}. In all experiments, we set $c_1=3$ and $c_2=12$. We initialize $\lambda = \lambda_{\infty}$ to be the smallest value of $\lambda$ for which $\blasso = 0$, that is, $\lambda_{\infty} = \max_{k=1,\ldots, m} | H_{c,k}^{(N)T} f_c^{(N)} |/N$, where $H_{c,k}^{(N)}$ stands for the $k$-th column of $H_c^{(N)}$, which is the same as the matrix $H_c$ but then based on the first $N$ Monte Carlo draws \citep[Exercise 2.1]{tibshirani+w+h;2015}. Next, we decrease the value of $\lambda$, e.g., by dividing it by two, such as to incorporate more and more control variates. If too many control functions are selected, i.e., more than $c_2 \sqrt{n}$, we increase the value of $\lambda$ again, e.g., by multiplying it by two, to finally reach the desired range for the number of active variables. In the end, this procedure ensures a straightforward computation of the LSLASSO(X) because the size of the associated linear system remains reasonable. Contrary to $K$-fold cross-validation, it is not necessary to split the data into multiple folds, leading to a reduced computation time.

\begin{algorithm}[!h]
\caption{Dichotomic Search}
\algsetup{linenodelimiter=.}
\begin{algorithmic}[1]
\REQUIRE $f_c^{(n)}$, $H_c$, $n$, $N \leq n$, $(c_1,c_2)$.
\setstretch{1.2}
\STATE Initialize $\lambda = \lambda_{\infty} $ and $\hat{\ell} = 0$.
\STATE \textbf{While} $\hat{\ell} \notin [c_1 \sqrt{n}, c_2 \sqrt{n}]$
\STATE \quad $\hat \beta_N^{\lambda}(f) 
\in \argmin_{\beta \in \reals^{m}} \frac{1}{2N} \|f_c^{(N)} - H_c^{(N)} \beta\|_{2}^{2} + \lambda\norm{\beta}_{1}$.
\STATE \quad $\hat S_N = \supp( \hat \beta_N^{\lambda}(f) )$ and $\hat{\ell} = \abs{\hat{S}_N}$.
\STATE \quad \textbf{if} $\hat{\ell} < c_1 \sqrt{n}$ \textbf{then} decrease $\lambda$.
\STATE \quad \textbf{if}  $\hat{\ell} > c_2 \sqrt{n}$ \textbf{then} increase $\lambda$.
\STATE \textbf{Return} $\hat \beta_N^{\lambda}(f)$.
\end{algorithmic}
\label{alg:dicho}
\end{algorithm}

The pseudo-code of the corresponding LSLASSO(X) method is provided in Algorithm~\ref{alg:lslasso}. The regression coefficients $\bols$ and $\blasson$ for OLS and LASSO are computed using the \textsf{Scikit-Learn} library \citep{scikit-learn}, employing coordinate descent to solve the LASSO problem. 

\begin{algorithm}[!h]
\caption{Least-Squares Lasso Monte-Carlo (LSLASSO)}
\algsetup{linenodelimiter=.}
\begin{algorithmic}[1]
\begin{spacing}{0.2}
\REQUIRE $f:\mathcal{X} \rightarrow \reals$, $h_j :\mathcal{X} \rightarrow \reals, 1\leq j \leq m$, $P$, $n$, $N \leq n$.
\end{spacing}
\setstretch{1.1}
\STATE Generate $(X_i)_{i=1,\ldots , n}$ independently according to $P$.
\STATE $f^{(n)} = (f(X_1),\ldots,f(X_n))$ and $H = \bigl( h_j(X_i) \bigr)_{i = 1,\ldots,n}^{j=1,\ldots,m}$.
\STATE $f_c^{(n)} = f^{(n)} - \1_n (\1_n^T f^{(n)})/n$ and $H_c = H - \1_n (\1_n^T H)/n$.
\STATE Solve $\hat \beta_N^{\lambda}(f)$ by cross-validation or dichotomic search.
\STATE $\hat S_N = \supp( \hat \beta_N^{\lambda}(f) )$ and $\hat{\ell} = \abs{\hat{S}_N}$.
\STATE Slice $n \times \hat{\ell}$ matrix $H^{(n)} _ {c,\hat S_N} = (H^{(n)} _ {c\, ij })_{i = 1,\ldots,n,\,j\in \hat S_N}$
\STATE $\blslasso(f) \in \argmin_{\beta \in \reals^{m}} \|f_c^{(n)}-H_{c,\hat S_N}^{(n)} \beta\|_{2}^{2}$.
\STATE MC estimate $\alslasso(f) = P_n[f - \blslasso(f)^T h]$.
\end{algorithmic}
\label{alg:lslasso}
\end{algorithm}

\paragraph{Integrands.}
We consider several integrands $f$ on $[0,1]^d$:
\begin{ceqn}
\begin{align}
\label{eq:varphi}
    \varphi(x_1,\ldots,x_d) &= 1 + \sin\left(\pi\left(\frac{2}{d} \sum_{i=1}^d x_i - 1\right)\right), 
\end{align}
\end{ceqn}
and for all $j = 1, \ldots, d$,
\begin{ceqn}
\begin{align}
\label{eq:fj}    
    f_j(x_1,\ldots,x_d) &= \prod_{i=1}^j (2/\pi)^{1/2} x_i^{-1} \mathrm{e}^{-\log(x_i)^2/2},\\
\label{eq:gj}
    g_j(x_1,\ldots,x_d) &= \prod_{i=1}^j \frac{\log(2)}{2^{x_i-1}} = \log(2)^j 2^{\sum_{i=1}^j (1-x_i)}.
\end{align}
\end{ceqn}
All these functions integrate to $1$ on $[0, 1]^d$. The functions $f_j$ and $g_j$ are built using tensor products of log-normal and exponential density functions, respectively, and depend on the first $j$ coordinates only. {This construction ensures that for small $j$, the integrands $f_j$ and $g_j$ lend themselves to Monte Carlo integration based on selected control variates. In contrast, the functions} $\varphi$, $f_d$ and $g_d$ represent more difficult situations where all the coordinates are involved and the symmetry of their role makes it harder to select some meaningful control functions. None of the integrands belongs to the linear span of the control variates constructed in the next paragraph.

\paragraph{Control variates.}
Multidimensional control functions with respect to the uniform distribution over $[0,1]^d$ are easy to construct based on univariate ones. Let $(h_{1}, \ldots, h_{k})$ be a vector of one-dimensional control functions, i.e., $\int_0^1 h_j(x) \, \mathrm{d} x = 0$ for each $j=1,\ldots, k$. Let $h_0=1$ denote the constant function equal to one. Without further information on the integrand, the usual way to construct multivariate controls is by forming tensor products of the form
$$	
	h_{\ell}(x_{1}, \ldots, x_{d}) 
	= \prod_{j=1}^{d} h_{\ell_{j}}(x_{j}) 
$$
for a multi-index $\ell = (\ell_{1}, \ldots, \ell_{d})$ in $\{0,\ldots, k\}^{d} \setminus \{(0, \ldots, 0)\}$, yielding a total number of $(k+1)^{d}-1$ control functions. 

A drawback of such a construction is that the number of control functions grows quickly with $k$. Alternative approaches yielding smaller control spaces consist of imposing $\ell_{j} = 0$ for all but a small number (one or two, say) of coordinates $j = 1,\ldots,d$ or simply picking at random a desired number, say $m$, of indices $\ell = (\ell_{1}, \ldots, \ell_{d})$. 

In this study, the set of control variates at our disposal is constructed as follows. We consider different settings of dimension $d$ with $k$ univariate control functions {in each dimension.} For $j \in \{1,\ldots,k\}$, let $h_j(x) = L_j(2x-1)$ for $x \in [0, 1]$, with $L_j$ the univariate Legendre polynomial (Legendre function of the first kind) of degree $j$; see Example~\ref{ex:legendre}. We have {$\int_0^1 h_j(x) \, \diff x = 0$} for all $j=1,\ldots,m$. Because the Legendre polynomials are orthogonal, they provide some numerical stability when inverting the Gram matrix. The multivariate control functions are sorted in ascending order according to the total degree $  \sum_{j=1}^d \ell_{j} \in \{1,\ldots, ,kd\}$ of the polynomial. In the experiments, the number of control functions $m$ is increased by progressively including all polynomials whose total degree is lower than or equal to a fixed threshold $\mathit{deg}$.

\paragraph{Settings.}
For the triple $(d, k, n)$ we consider $d \in \{3, 5, 8\}$, $k \in \{12, 10, 3\}$ and $n \in \{2\,000,\; 5\,000,\;10\,000\} $. For each choice of $(d, k)$, the number of control variates $m$ with a total degree lower than or equal to a fixed threshold $\mathit{deg}$ are given in Table \ref{tab:cv_config}. The case $d=8$ represents a difficult situation as the number of points $n$ is relatively small compared to the dimension. For instance, a grid made of only {four} points in each direction would already comprise $65\,536$ points.

\begin{table}[h]
\center
   \begin{tabular}{cc|rrrrr}
    \toprule
     \multirow{2}{*}{$d$} & \multirow{2}{*}{$k$} & \multicolumn{5}{c}{Degree threshold ($\mathit{deg}$)}  \\ 
     & & 1 & 3 & 5 & 10 & 12 \\
     \midrule
     3 & 12 & 3 & 19 & 55 & 285 & 454 \\
     5 & 10 & 5 & 55 & 251 & 3\,001 & 6\,157\\
     8 & 3 & 8 & 164 & 1\,214 & 20\,993 & 36\,813\\
     \bottomrule
    \end{tabular}
    \caption{\label{tab:cv_config}Number of control variates $m$ by degree threshold $\mathit{deg}$ in dimension $d$ constructed out of tensor products of $k$ univariate polynomials.}
\end{table}
\begin{table}[h]
\center
   \begin{tabular}{cccc}
    \toprule
     $n$ & $N$ & $\lfloor 3\sqrt{n} \rfloor$ & $\lfloor 12\sqrt{n} \rfloor$ \\
    \midrule
     \phantom{0}2\,000 & \phantom{0\,}700 & 134 & \phantom{0\,}536 \\
     \phantom{0}5\,000 & 1\,000 & 212 & \phantom{0\,}848 \\
     10\,000 & 2\,000 & 300 & 1\,200 \\
     \bottomrule
    \end{tabular}
   \caption{\label{tab:params} Sample sizes $n$ and sub-sample sizes $N$ together with the range $[c_1 \sqrt{n},c_2 \sqrt{n}]$ corresponding to the imposed number of selected control variates in LSLASSO.}
\end{table}

\paragraph{Results.}

The different Monte Carlo estimates are compared on the basis of their mean squared error (MSE). Figure~\ref{fig:toy_data} presents the boxplots obtained over $100$ replications of the values returned by each of the methods. In Tables~\ref{tab:mse_phi} to~\ref{tab:mse_g4}, we provide the ratio $MSE(\mathrm{vanilla})/MSE(\cdot)$, the MSE of the vanilla Monte Carlo estimate divided by the MSE for the current method, as a measure of statistical efficiency of the method relative to naive Monte Carlo integration. The four tables correspond to the four panels (a) to (d) in Figure~\ref{fig:toy_data}. For a given number of control variates $m$, the most efficient method is indicated in bold. For the Lasso-based methods, the results for the $\lambda$ selection based on cross-validation (Algorithm~\ref{alg:k_fold}) and dichotomic search (Algorithm~\ref{alg:dicho}) did not differ much; for the sake of brevity, the figures and the tables report the results associated to the dichotomic search.

Figures~\ref{fig:1} and \ref{fig:2} highlight the success or failure of the OLS estimator depending on the size of $m$ compared to $n$. In Figures~\ref{fig:3} and \ref{fig:4}, we consider larger values of $m$ and only compare the Lasso-based methods as it takes too much time to solve the OLS. In all our experiments, the LSLASSOX is the clear winner as it has the highest accuracy in almost all configurations. {Moreover, the LSLASSOX can be computed much faster than the LSLASSO: in our implementation, preselecting the control variates based on a smaller subsample led to a reduction op the computation time by a factor between three and twenty.} 

In Figure~\ref{fig:1}, boxplots of the values returned by each of the methods are provided for $\varphi$ in \eqref{eq:varphi} when $d=3$ and $n=10\,000$. In this situation, where $m$ is small compared to $n$, the OLS performs very well and the LSLASSO procedure selects almost all control variates so it performs as well as OLS. In Figure~\ref{fig:2},  boxplots of the values returned by each of the methods are provided for $g_3$ in \eqref{eq:gj} when $d=5$, $n=2\,000$, and $N=700$. In this case, the OLS estimator starts to break down as soon as the number, $m$, of control variates is of the same order as $n$. It is then necessary to perform some control variate selection, which is succesfully carried out by the LASSO and LSLASSO. Both of these estimators give the best results. Although the number of sample points used in the selection step of LSLASSOX has been reduced compared to the LSLASSO, the stability of the active set is barely affected. Accordingly, the error distributions for LSLASSO and LSLASSOX are quite similar.

Figures \ref{fig:3} and \ref{fig:4} reveal the benefits of selecting appropriate control variates before applying the OLS estimator. Figure~\ref{fig:3} covers the function $f_1$ in \eqref{eq:fj} when $d=5$ and $n=5\,000$, while Figure~\ref{fig:4} deals with the function $g_4$ in \eqref{eq:gj} when $d=8$ and $n=2\,000$. In the latter case, the number of control variates, $m=36\,813$, is huge compared to the sample size $n=2\,000$. However, the Lasso-based methods perform remarkably well in those settings. More precisely, in dimension $d=5$ with the function $f_1$, the mean square error of the naive Monte Carlo estimator is of the order $10^{-5}$ whereas the one of the LSLASSOX is of the order $10^{-10}$. {Similarly, in dimension $d=8$ with the function $g_4$, the mean square error goes down from $10^{-4}$ to $10^{-8}$.} Table~\ref{tab:mse_g4} highlights the benefits of the LSLASSO over the LASSO in difficult situations.  

In the recent study \cite{south+m+d+o:2018}, the authors investigate the use of \textit{regularization} in computing control variates estimates. They focus on the LASSO and ridge regression and they show, based on several examples, that the LASSO generally outperforms the ridge. In the applications they consider, they found that polynomials with relatively small degrees in each direction ($k$ equal to $2$ and $3$) give the best performance. The examples considered here show a similar pattern as the results do not generally improve beyond degree $k = 3$. 

\begin{figure*}
	\begin{subfigure}[h]{0.49\textwidth}
         \centering
         \includegraphics[width=\textwidth]{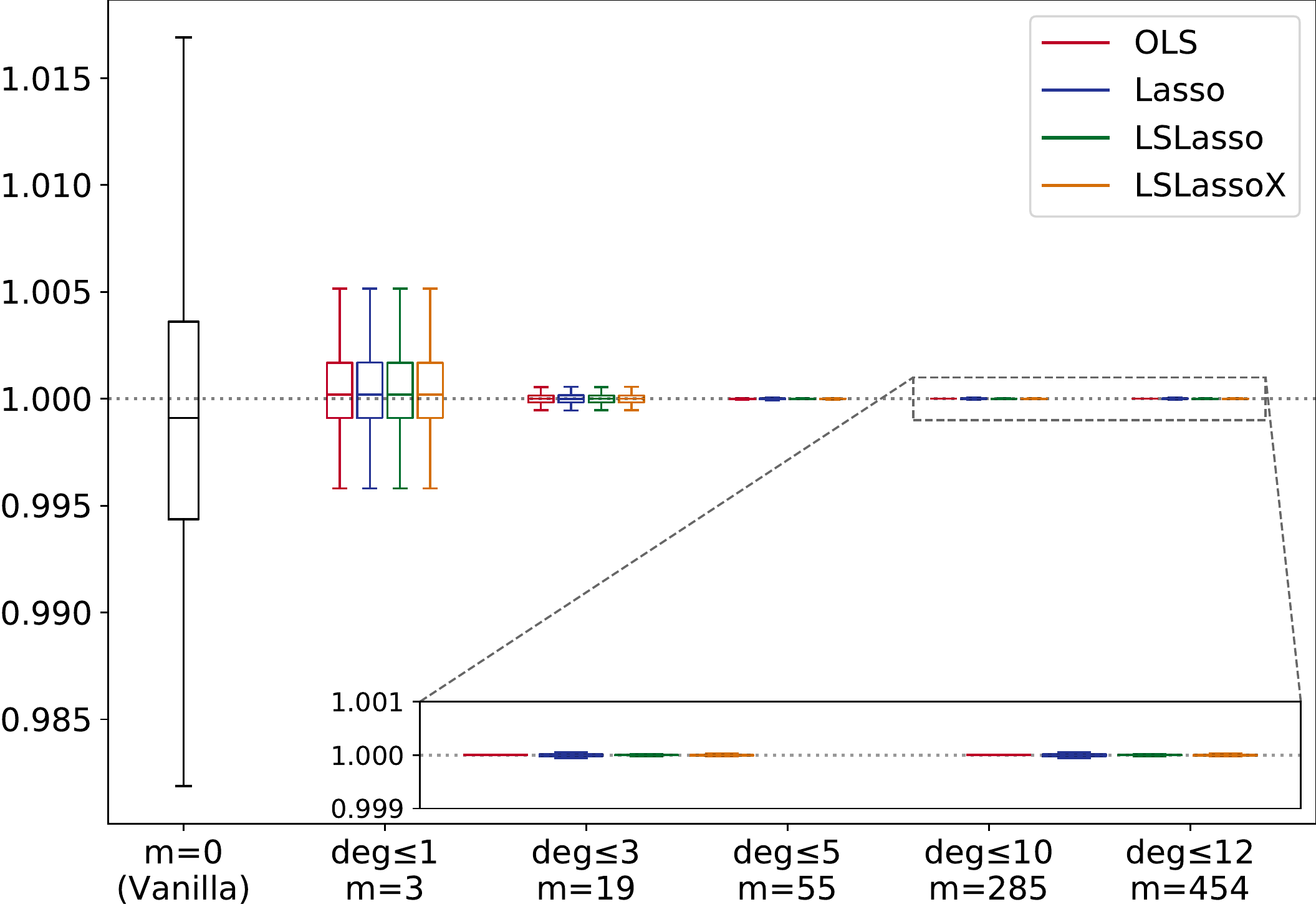}
	\caption{Boxplots for $\varphi$ in \eqref{eq:varphi} with $d=3,\ n =10\,000,\ N=2\,000$.}
	\label{fig:1}
     \end{subfigure}
	\hfill
	\begin{subfigure}[h]{0.49\textwidth}
         \centering
	\includegraphics[width=\textwidth]{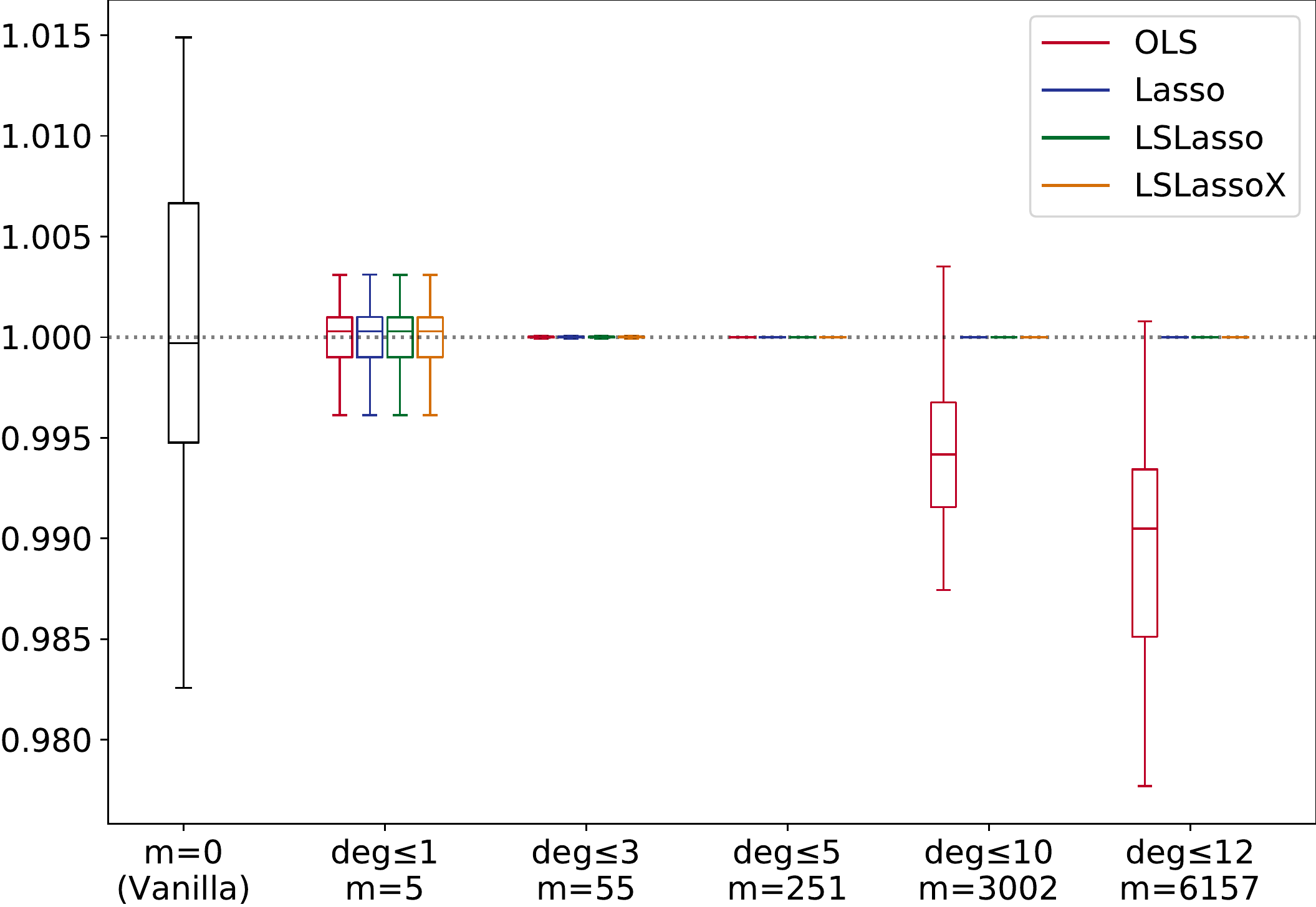}
	\caption{Boxplots for $g_3$ in \eqref{eq:gj} with $d=5,\ n =2\,000,\ N=700$.}\label{fig:2}
	\end{subfigure}

	\begin{subfigure}[h]{0.49\textwidth}
         \centering
         \includegraphics[width=\textwidth]{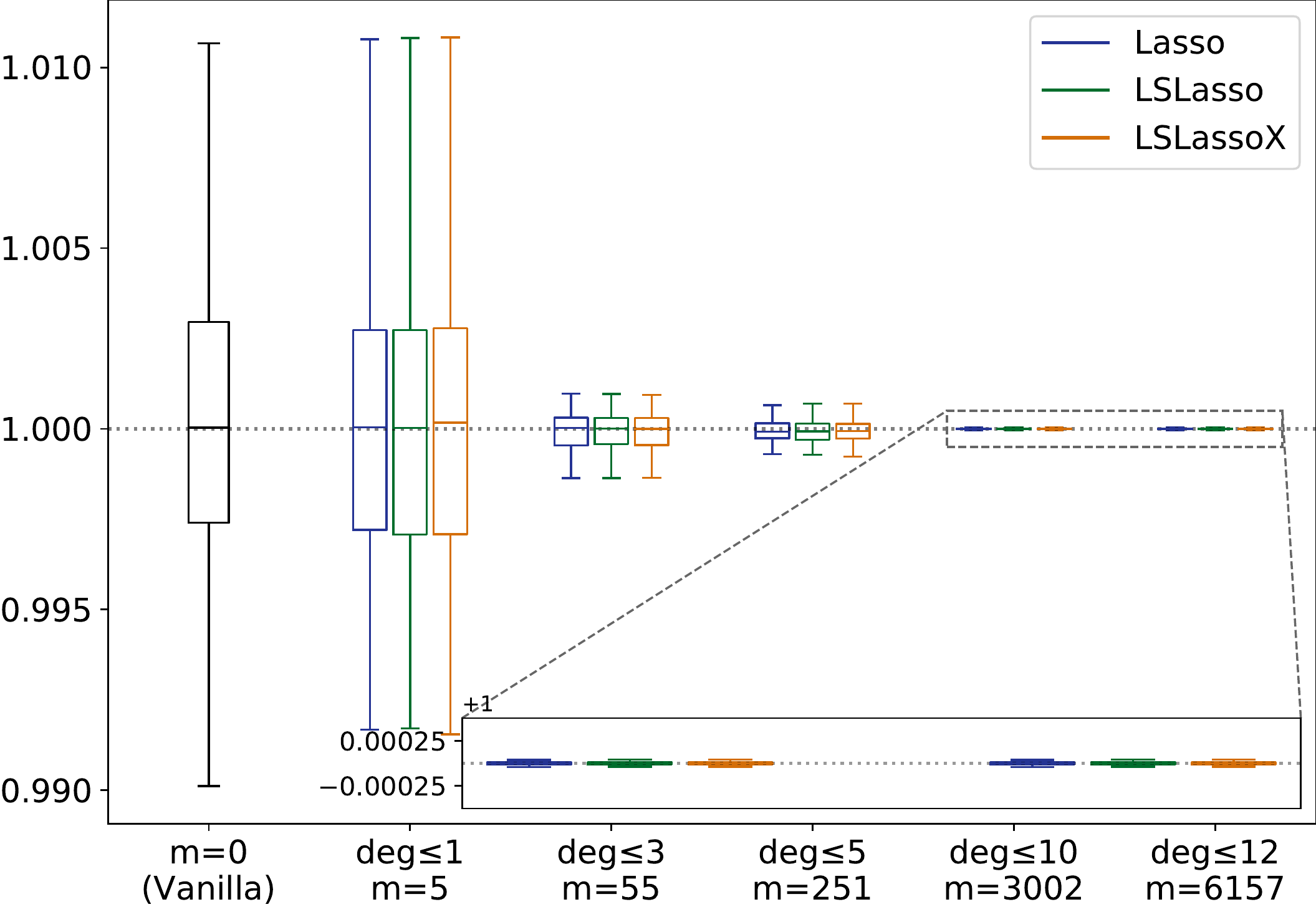}
	\caption{Boxplots for $f_1$ in \eqref{eq:fj} with $d=5,\ n =5\,000,\ N=1\,000$.}\label{fig:3}
     \end{subfigure}
	\hfill
	\begin{subfigure}[h]{0.49\textwidth}
         \centering
	\includegraphics[width=\textwidth]{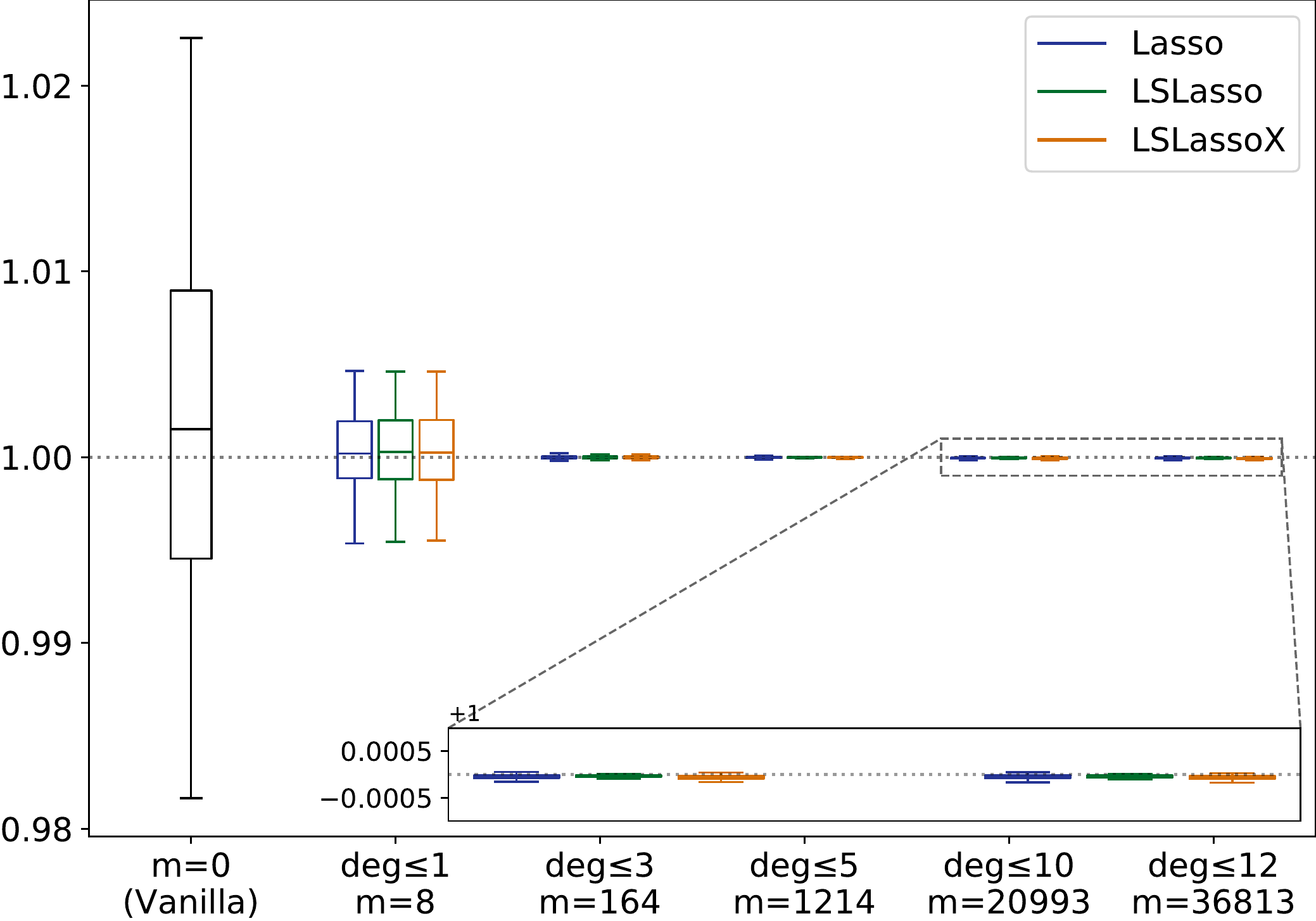}
	\caption{Boxplots for $g_4$ in \eqref{eq:gj} with $d=8,\ n =2\,000,\ N=700$.}\label{fig:4}
	\end{subfigure}
	\caption{Boxplots (based on 100 runs) of the values returned by each of the methods for functions $\varphi,g_3,f_1,g_4$ in \eqref{eq:varphi}--\eqref{eq:gj}.}
	\vspace{1mm}
	\label{fig:toy_data}
	\vspace{0.5cm}
\end{figure*}

\begin{table*}
\centering
\begin{minipage}[h]{\columnwidth}
\centering 
\resizebox{\textwidth }{!}{  
 \begin{tabular}{lccccc}
    \toprule
    \multicolumn{1}{c}{$m=$} & 3 & 19 & 55 & 285 & 454 \\
    \midrule
     OLS & 8.42e00 & 8.56e02 & \textbf{2.12e05} & 2.49e11 & \textbf{5.27e14}\\[0.5ex]
    LASSO & 8.42e00 & 8.53e02 & 6.72e04 & 7.71e04 & 7.71e04 \\[0.5ex]
    LSL & 8.42e00 & \textbf{8.58e02} & 2.10e05 & 6.26e05 & 1.37e06\\[0.5ex]
    LSLX & 8.42e00 & 8.51e02 & 2.09e05 & \textbf{2.49e11} & 2.91e05 \\[0.5ex]
    QMC & \multicolumn{5}{c}{Halton: 8.76e01 \qquad Sobol: 3.29e02} \\
     \bottomrule
    \end{tabular}}
    \vspace{-0.25cm}
   \caption{\label{tab:mse_phi} Statistical efficiency for $\varphi$; see also Figure~\ref{fig:1}.}
\end{minipage}\hfill
\begin{minipage}[h]{\columnwidth}
\centering 
\resizebox{\textwidth }{!}{  
   \begin{tabular}{lccccc}
    \toprule
     \multicolumn{1}{c}{$m=$} & 5 & 55 & 251 & 3002 & 6157 \\
    \midrule
     OLS & 2.45e01 & 5.75e04 & 7.48e08 & 1.42e00 & 4.94e-1 \\[0.5ex]
    LASSO & 2.45e01 & 5.75e04 & 4.19e06 & 4.83e05 & 4.31e05 \\[0.5ex]
    LSL & 2.45e01 & 5.75e04 & \textbf{7.79e08} & \textbf{4.83e06} & \textbf{4.54e06} \\[0.5ex]
    LSLX & 2.45e01 & 5.75e04 & 1.87e08 & 1.71e06 & 5.54e05 \\[0.5ex]
    QMC & \multicolumn{5}{c}{Halton: 3.75e00 \qquad Sobol: 1.57e01} \\
     \bottomrule
    \end{tabular}
    }
    \vspace{-0.25cm}
   \caption{\label{tab:mse_g3} Statistical efficiency for $g_3$; see also Figure~\ref{fig:2}.}
\end{minipage}

\vspace{1cm}

\begin{minipage}[t]{\columnwidth}
\centering 
\resizebox{\textwidth}{!}{  
   \begin{tabular}{lccccc}
    \toprule
    \multicolumn{1}{c}{$m=$} & 5 & 55 & 251 & 3002 & 6157 \\
    \midrule
    LASSO & 1.11e00 & \textbf{6.60e01} & \textbf{1.79e02} & 8.17e04 & 8.56e04\\[0.5ex]
    LSL & 1.11e00 & 6.59e01 & 1.76e02 & 6.77e04 & 6.83e04 \\[0.5ex]
    LSLX & 1.11e00 & 6.59e01 & 1.78e02 & \textbf{8.97e04} & \textbf{9.24e04} \\[0.5ex]
    QMC & \multicolumn{5}{c}{Halton: 4.60e00 \qquad Sobol: 7.21e01} \\
     \bottomrule
    \end{tabular}
    }
    \vspace{-0.25cm}
   \caption{\label{tab:mse_f1} Statistical efficiency for $f_1$; see also Figure~\ref{fig:3}.}
\end{minipage}\hfill
\begin{minipage}[t]{\columnwidth}
\centering 
\resizebox{\textwidth}{!}{  
   \begin{tabular}{lccccc}
    \toprule
    \multicolumn{1}{c}{$m=$} & 8 & 164 & 1214 & 20993 & 36813 \\
    \midrule
     LASSO & 1.98e01 & 1.52e04 & 7.94e05 & 7.94e04 & 6.05e04 \\[0.5ex]
    LSL & 1.97e01 & 1.53e04 & 1.32e06 & \textbf{1.49e05} & \textbf{1.28e05}\\[0.5ex]
     LSLX & 1.98e01 & \textbf{1.54e04} & \textbf{1.38e06} & 1.98e04 & 1.55e04\\[0.5ex]
    QMC & \multicolumn{5}{c}{Halton: 3.80e00 \qquad Sobol: 2.60e01} \\
     \bottomrule
    \end{tabular}
    }
    \vspace{-0.25cm}
   \caption{\label{tab:mse_g4} Statistical efficiency for $g_4$; see also Figure~\ref{fig:4}.}
\end{minipage}
\end{table*}

\section{Bayesian inference}\label{sec:bayes_inf}

In this section, we compare the different Monte Carlo estimates on Bayesian inference examples. Given some observed data $x$, the goal is to infer the parameter $\theta$ of a statistical model. We have some information through the prior distribution $\pi(\theta)$ and observe the model likelihood $\ell(x|\theta)$. Bayes' rule gives the posterior distribution as
$$
	p(\theta|x) = \frac{\ell(x|\theta)\pi(\theta)}{\int_{\Theta} \ell(x|\theta)\pi(\theta)\diff\theta} \cdot
$$
The normalizing constant in the denominator is called evidence and is of interest for Bayesian model selection:
$$
	Z = \int_{\Theta} \ell(x|\theta)\pi(\theta)\diff\theta.
$$
Typically, this integral is analytically intractable. It is also difficult to compute numerically if the dimension $d$ of the parameter space $\Theta$ is large. 

We consider the same datasets as in \citep{south+m+d+o:2018}: the European dipper capture-recapture data from \citep{marzolin1988polygynie} in Section~\ref{sec:dipper} and the sonar data from \citep{gorman1988analysis} in Section~\ref{sec:sonar}. The dimensions of the integration domains are $d = 12$ and $d = 61$, respectively. 

As in Section~\ref{sec:simu}, we consider multivariate control functions based on univariate orthogonal polynomials by forming tensor products of the form $h_{\ell}(x_{1}, \ldots, x_{d}) = \prod_{j=1}^{d} h_{\ell_{j}}(x_{j})$, for a multi-index $\ell = (\ell_{1}, \ldots, \ell_{d})$ in $\{0,\ldots, k\}^{d} \setminus \{(0, \ldots, 0)\}$. In both examples, the dimension $d$ is so large that considering all tensor products is infeasible. Instead, we focus on combinations where $\ell_j$ equals $0$ for all but one or two coordinates, leading to a total number of $m=kd$ and $m=kd + k^2 d(d-1)/2$ control variates, respectively. 

The different Monte Carlo estimates are compared on the basis of their mean squared errors (MSE). In contrast to Section~\ref{sec:simu}, the true value of the integral is unknown. An estimate of this value, referred to as the gold standard $Z^\star$, is obtained by naive Monte Carlo with sample size $n=10^8$. The variance of this estimate, computed on $20$ independent runs, is smaller than the variance of all the other considered methods. The different boxplots of Figure~\ref{fig:real_data} show the results obtained over $100$ independent runs of $\hat Z/Z^\star$ where $\hat Z$ is the estimate of the evidence. Tables~\ref{tab:mse_cap1} to \ref{tab:mse_son2} provide numerical values for the statistical efficiency $\widehat{MSE}(\mathrm{vanilla})/\widehat{MSE}(\cdot)$. We consider various settings and the parameter configuration is $n \in \{2\,000;\,5\,000\}$ for the Monte Carlo sample size with $N \in \{700;\,1\,000\}$ for the Monte Carlo subsample size for the LSLASSO. The regularization parameter $\lambda$ is chosen via dichotomic search (Algorithm~\ref{alg:dicho}).

\subsection{European dipper capture-recapture data}
\label{sec:dipper}

The data-set given in Table~\ref{tab:capture_data} was collected by \citep{marzolin1988polygynie} and describes the annual capture and recapture counts of the bird species \textit{Cinclus cinclus}, also known as the European dipper, in eastern France from 1981 to 1987. We observe count data $x_{i,j}$ with $i \in \{1, \ldots, I\}$ and $j \in \{i+1,\ldots,J\}$, where $x_{i,j}$ denotes the number of birds released in year $i$ and subsequently recaptured for the first time in year $j$. In the example, we have $I = 6$ and $J = 7$, where 1981 corresponds to year $i = 1$. Also observed is $R_i$, the number of marked birds released into the population in year $i$.

Following \citep{brooks2000bayesian,nott2018approximation,south+m+d+o:2018}, we consider a Bayesian approach for the Cormack--Jolly--Seber model \citep{lebreton1992modeling}. The model parameters are $\phi_i$, a bird's survival probability from year $i$ to $(i+1)$ for $i \in \{1, \ldots, I\}$, together with $p_j$, the probability of a bird being recaptured in year $j\in\{2,\ldots,J\}$. Let $\nu_{i,j}$ denote the probability that a bird captured and released in year $i$ gets recaptured for the first time in year $j$. Since the bird must survive from year $i$ to year $j$, not be recaptured in years $i+1$ to $j-1$ and then finally be recaptured in year $j$, the probability is modelled as
$$ 
	\nu_{i,j} = \phi_i p_j \prod_{k=i+1}^{j-1} [\phi_k(1-p_k)].
$$
The number of birds released at year $i$ that are never recaptured at all is equal to $r_i = R_i - \sum_{j=i+1}^{J} x_{i,j}$ while the probability that a bird released in year $i$ is never recaptured is $\chi_i = 1 - \sum_{j=i+1}^{J} \nu_{i,j}$. The resulting likelihood is equal to
$$
	\ell(x|\theta) = \prod_{i=1}^I \left\{ \chi_i^{r_i} \prod_{j=i+1}^{J} \nu_{i,j}^{x_{i,j}} \right\},
$$
where $\theta=(\phi_1,\ldots,\phi_6,p_2,\ldots,p_7) \in [0, 1]^{12}$.
The uniform distribution is chosen as prior and we use tensor products of Legendre polynomials with $k=10$ (Example~\ref{ex:legendre}) as controls.

\begin{table}[t]
   \begin{tabular}{cccccccc}
    \toprule
    \multirow{2}{*}{\shortstack{Release\\year}} 
     &\multirow{2}{*}{\shortstack{Birds\\released}}  & \multicolumn{6}{c}{Year of recapture: $1981+\cdots$} \\
    \cmidrule{3-8}
     &  & 1 & 2 & 3 & 4 & 5 & 6 \\
     \midrule
    1981 & 22 & 11 & 2  & 0  & 0  & 0  & 0 \\[0.5ex]
    1982 & 60 &    & 24 & 1  & 0  & 0  & 0 \\[0.5ex]
    1983 & 78 &    &    & 34 & 2  & 0  & 0 \\[0.5ex]
    1984 & 80 &    &    &    & 45 & 1  & 2 \\[0.5ex]
    1985 & 88 &    &    &    &    & 51 & 0 \\[0.5ex]
    1986 & 98 &    &    &    &    &    & 52 \\
     \bottomrule
    \end{tabular}
   \caption{\label{tab:capture_data} European dipper capture-recapture data \citep{marzolin1988polygynie}. The counts in the triangle refer to the number of birds released in a given year and recaptured for the first time in a later year.}
\end{table}

The results for the various integration methos are reported in the same way as in Section~\ref{sec:simu}. The boxplots and statistical efficiencies are given in Figures~\ref{fig:cap1} and~\ref{fig:cap2} and Tables~\ref{tab:mse_cap1} and~ \ref{tab:mse_cap2} respectively. Similarly to the synthetic data, Figures~\ref{fig:cap1} and~\ref{fig:cap2} reveal the success or failure of the OLS on the capture-recapture data when the number of control variates $m$ is larger than the Monte Carlo sample size $n$. The variance goes down as $m$ increases. Tables~\ref{tab:mse_cap1} and~\ref{tab:mse_cap2} show that for $n=2\,000$, the OLS estimate gives the best performance whereas for $n = 5\,000$, the LASSO-based methods profit from the large number of available control variates. In this case, the LASSO is most efficient while the LSLASSOX performs similarly but at a reduced computing time. 
\begin{figure*}
	\begin{subfigure}[h]{0.49\textwidth}
         \centering
         \includegraphics[width=\textwidth]{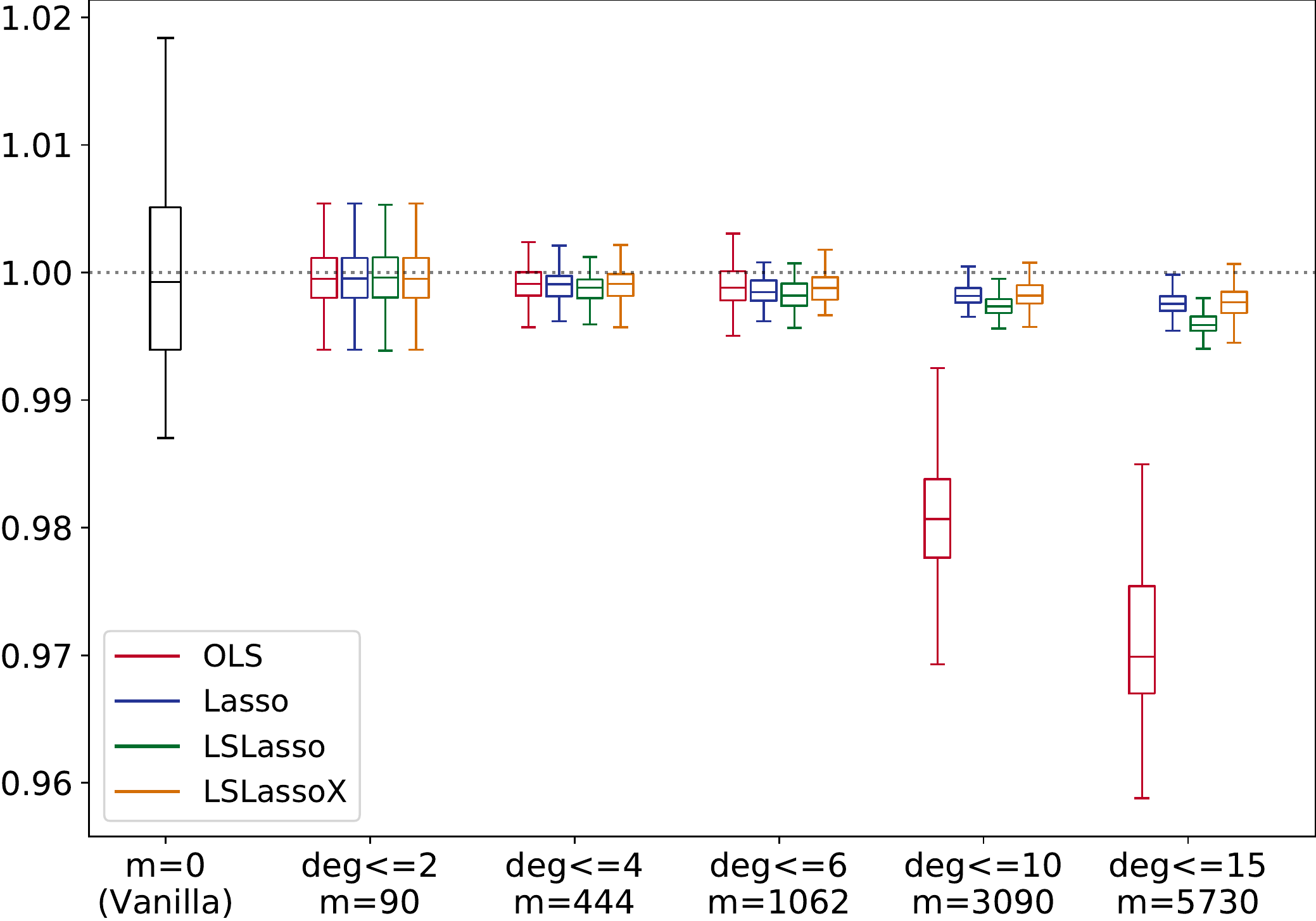}
	\caption{Boxplots for Capture dataset ($n=2000,\ N=700$)}
	\label{fig:cap1}
     \end{subfigure}
	\hfill
	\begin{subfigure}[h]{0.49\textwidth}
         \centering
	\includegraphics[width=\textwidth]{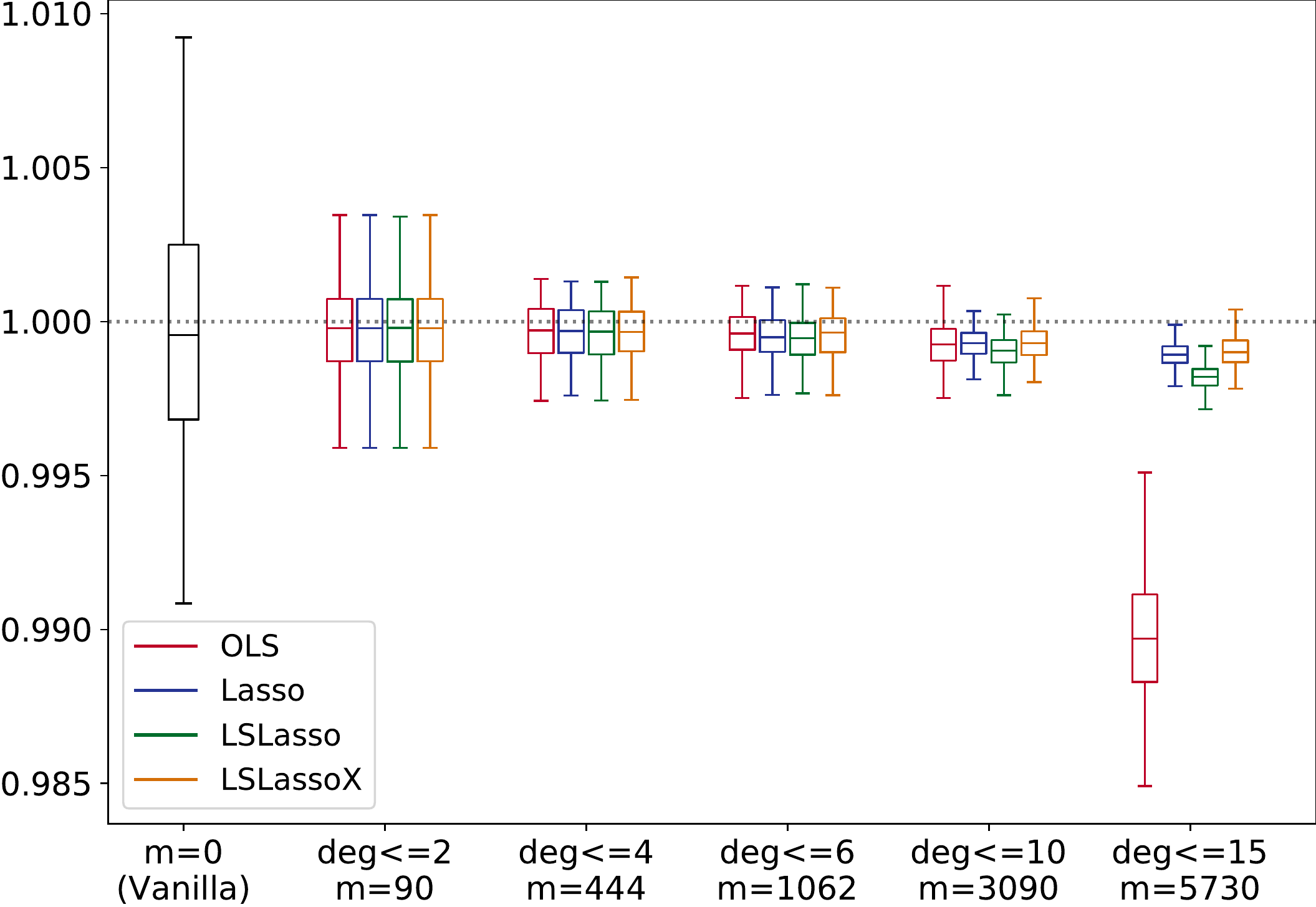}
	\caption{Boxplots for Capture dataset ($n=5000,\ N=1000$)}
	\label{fig:cap2}
	\end{subfigure}
	\begin{subfigure}[h]{0.49\textwidth}
         \centering
         \includegraphics[width=\textwidth]{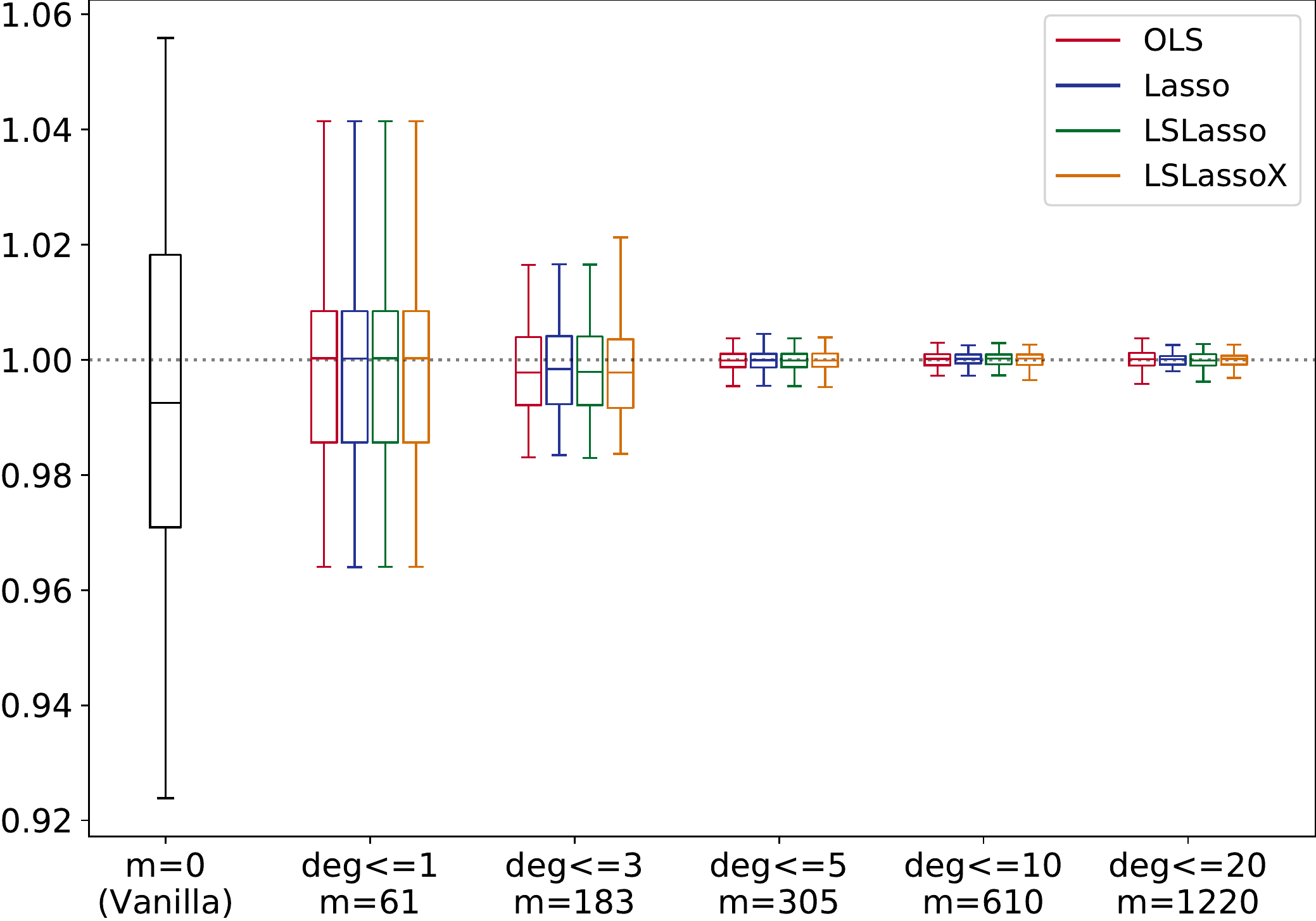}
	\caption{Boxplots for Sonar dataset ($n=2000,\ N=700$)}
	\label{fig:son1}
     \end{subfigure}
	\hfill
	\begin{subfigure}[h]{0.49\textwidth}
         \centering
	\includegraphics[width=\textwidth]{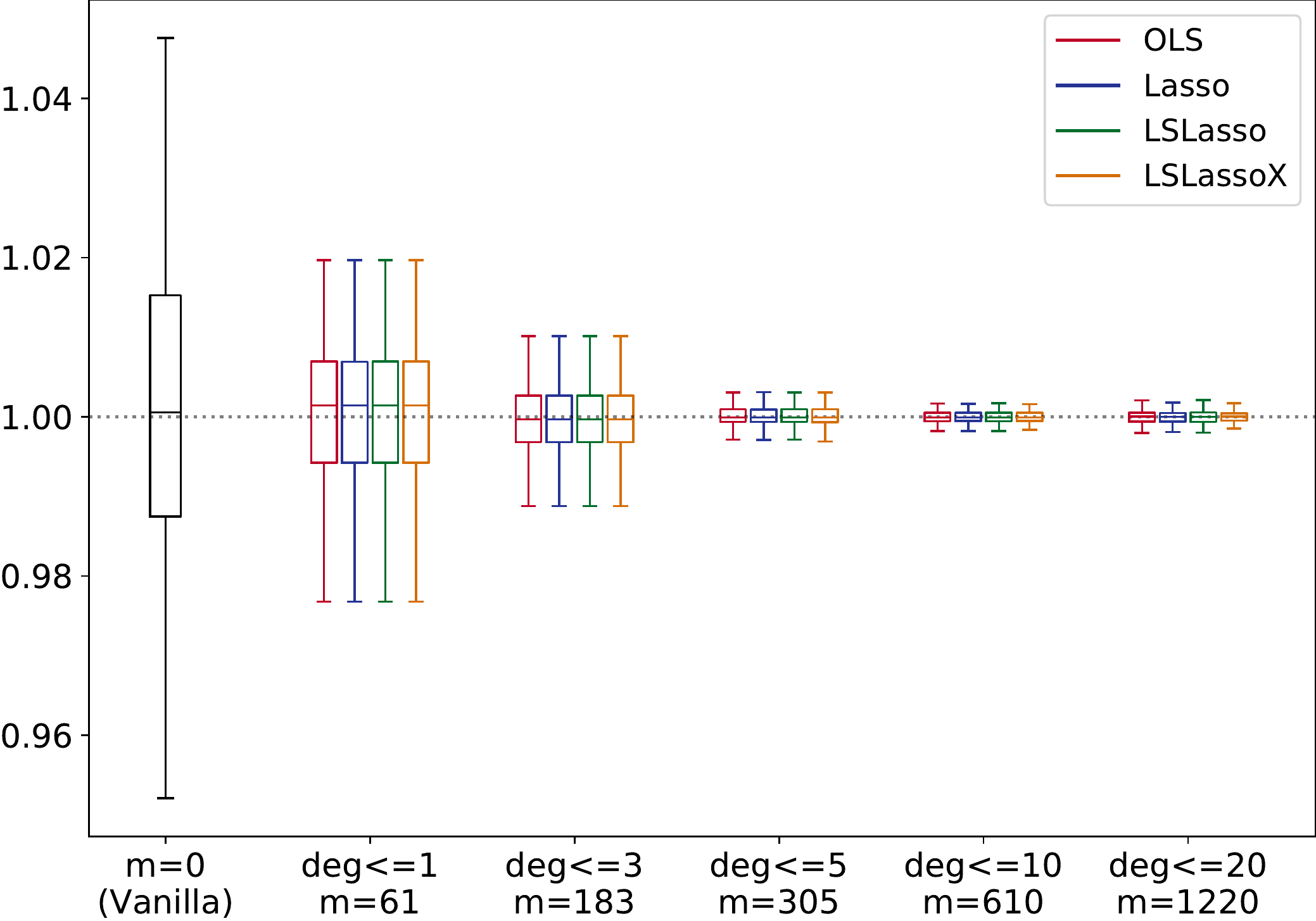}
	\caption{Boxplots for Sonar dataset ($n=5000,\ N=1000$)}
	\label{fig:son2}
	\end{subfigure}
	\caption{Boxplots (based on 100 runs) of $\hat Z/Z^\star$ returned by each of the methods for Capture-Recapture and Sonar examples.}
	\label{fig:real_data}
	\vspace{0.5cm}
\end{figure*}
\begin{table*}
\centering
\begin{minipage}[h]{\columnwidth}
\centering 
   \begin{tabular}{lccccc}
    \toprule
    \multicolumn{1}{c}{$m=$} & 90 & 444 & 1\,062 & 3\,090 & 5\,730 \\
    \midrule
     OLS & 9.33 & \textbf{20.7} & 14.7 & 0.14 & 0.06 \\[0.5ex]
    LASSO & 9.34 & 20.3 & 16.7 & \textbf{14.4} & \textbf{8.57} \\[0.5ex]
    LSL & 9.33 & 20.4 & 12.8 & 8.43 & 4.60\\[0.5ex]
    LSLX & 9.33 & 19.4 & \textbf{19.8} & 12.9 & 7.86 \\
     \bottomrule
    \end{tabular}
   \vspace{-0.25cm}
   \caption{\label{tab:mse_cap1} Capture data: statistical efficiency ($n=2000$)}
\end{minipage}\hfill
\begin{minipage}[h]{\columnwidth}
\centering 
   \begin{tabular}{lccccc}
    \toprule
    \multicolumn{1}{c}{$m=$} & 90 & 444 & 1\,062 & 3\,090 & 5\,730 \\
    \midrule
     OLS & 7.67 & 18.1 & 22.1 & 15.2 & 0.15 \\[0.5ex]
    LASSO & 7.67 & \textbf{18.4} & \textbf{22.3} & \textbf{22.8} & 12.8 \\[0.5ex]
    LSL & 7.67 & 18.0 & 21.3 & 13.3 & 5.24\\[0.5ex]
    LSLX & 7.67 & 17.8 & 21.4 & 21.6 & \textbf{13.2} \\
     \bottomrule
    \end{tabular}
   \vspace{-0.25cm}
   \caption{\label{tab:mse_cap2} Capture data: statistical efficiency ($n=5000$)}
\end{minipage}

\vspace{1cm}

\begin{minipage}[t]{\columnwidth}
\centering 
   \begin{tabular}{lccccc}
    \toprule
     \multicolumn{1}{c}{$m=$} & 61 & 183 & 305 & 610 & 1220 \\
    \midrule
     OLS & 3.39 & 13.3 & 246 & 548 & 330 \\ [0.5ex]
    LASSO & 3.39 & 13.6 & \textbf{250} & \textbf{673} & 680 \\[0.5ex]
    LSL & 3.39 & 13.3 & 246 & 564 & 499 \\[0.5ex]
    LSLX & 3.39 & \textbf{13.9} & 244 & 558 & \textbf{680} \\
     \bottomrule
    \end{tabular}
   \vspace{-0.25cm}
   \caption{\label{tab:mse_son1} Sonar data: statistical efficiency ($n=2000$)}
\end{minipage}\hfill
\begin{minipage}[t]{\columnwidth}
\centering 
   \begin{tabular}{lccccc}
    \toprule
    \multicolumn{1}{c}{$m=$} & 61 & 183 & 305 & 610 & 1220 \\
    \midrule
     OLS & 4.48 & 17.0 & 235 & 801 & 601 \\[0.5ex]
     LASSO & 4.49 & 17.0 & \textbf{240} & 821 & 721 \\[0.5ex]
    LSL & 4.48 & 17.0 & 235 & 804 & 629 \\[0.5ex]
     LSLX & 4.48 & 17.0 & 241 & \textbf{833} & \textbf{734} \\
     \bottomrule
    \end{tabular}
   \vspace{-0.25cm}
   \caption{\label{tab:mse_son2} Sonar data: statistical efficiency ($n=5000$)}
\end{minipage}
\end{table*}

\subsection{Sonar data}
\label{sec:sonar}

The data were collected by \citep{gorman1988analysis} and are available from the UCI Machine Learning Repository \citep{asuncion2007uci}. The data matrix $X$ represents $208$ sonar signals, each one composed of $60$ attributes within the binary classification framework. A column of $1$'s is added to the matrix $X$ to represent the intercept so that $X \in \reals^{208 \times 61}$. The goal is to assess whether the sonar signal bounces off a metal cylinder (label $y=1$) or a roughly cylindrical rock (label $y=-1$). The different covariates represent the energy within particular frequency bands, integrated over a certain period of time. Using the encoding $y \in \{-1,+1\}$ and following a logistic regression model, the resulting log-likelihood is
$$
	\log \ell(X,y|\theta) 
	= -\sum_{i=1}^{208} 
	\log\left\{
		1 + \exp\left( - y_i \sum_{j=1}^{61} X_{ij} \theta_j \right)
	\right\},
$$
where the model coefficient $\theta \in [-1,1]^{61}$ has a uniform prior distribution. We use the family of Legendre polynomials as control functions with $k=20$. The boxplots and statistical efficiencies are presented in Figures~\ref{fig:son1} and~\ref{fig:son2} and Tables~\ref{tab:mse_son1} and~\ref{tab:mse_son2}, respectively. Once again, the Lasso-based methods, with their selection strategy, are able to benefit from a larger control variates space. The winner of this competition is LSLASSOX as it offers the best performance combined with a smaller computation time compared to the LSLASSO.

\section{Conclusion and perspective}
\label{sec:conclusion}

The use of high-dimensional control variates with the help of a LASSO-type procedure has been shown to be efficient in order to reduce the variance of the basic Monte Carlo estimate. The method, called LSLASSO(X), that first selects appropriate control variates by the LASSO, possibly on a smaller subsample, and then estimates the control variate coefficients by least squares performs excellently considering the modest computing time required. Several avenues for further research are now discussed.

The construction of control variates by a change of measure (Remark~\ref{rk:construction_variates}) presupposes some knowledge on the underlying integration measure in order to choose an appropriate sampling distribution. For instance, if the support of the sampling measure does not cover the whole integration domain then the method will certainly fail. \textit{Adaptive importance sampling} (see, e.g., \citep{owen+z:2000,portier2018asymptotic}) offers a possible solution, involving online estimates of the appropriate sampling policy and the optimal linear combination of control variates. 

Assumption~\ref{ass_1_sub_gauss} on the sub-Gaussianity of the residuals is key to obtain concentration inequalities. For certain applications, it might by too restrictive, however. In the absence of such an assumption or more generally of suitable bounds on the tails of the residual distribution, other types of results such as almost sure convergence rates might still be pursued.

In the random design setting, the estimators of coefficient vector $\beta^\star(f)$ are all biased, even the OLS estimator. The bias may be removed by sample splitting \citep{avramidis1993}, but at the cost of an increased variance, especially if the number of control variates is large. For the Lasso-based methods, debiasing methods are studied in \citep{javanmard2018} and the references therein. The merits of these techniques for Monte Carlo control variate methods remain to be investigated.

We have presented different control variate methods from the point of view of estimation only. Equally important questions are that of model evaluation and Monte Carlo sample size calculation, assessing the accuracy of the estimate. Several ways can be imagined such as sample splitting (e.g., cross-validation) and plug-in estimation of the residual variance $\sigma^2(f)$, using for instance the estimated residuals.

\paragraph{Acknowledgments.} The authors are grateful to the Associate Editor and two anonymous Reviewers for their many valuable comments and interesting suggestions. Further, the authors gratefully acknowledge the exceptionally careful proofreading done by Aigerim Zhuman (UCLouvain).

\appendix
\normalsize

\section{Auxiliary results}
\label{app:aux}

\begin{lemma}(Sub-Gaussian)
	\label{lemma:concentration}
	Let $X_1,\ldots,X_n $ be independent and identically distributed random variables in $(\mathcal{X}, \mathcal{A})$ with distribution $P$. Let $\varphi_1,\ldots,\varphi_p $ be real-valued functions on $\mathcal{X}$ such that $P(\varphi_k) = 0$ and $\varphi_k \in \mathcal{G}(\tau^2)$ for all $k=1,\ldots,p$. Then for all $\delta > 0$, we have with probability at least $1-\delta$,
	\begin{ceqn}
	\begin{align*}
	\max_{k=1, \ldots,p}
	\left\lvert \sum_{i=1}^n \varphi_k(X_i) \right\rvert
	\leq  \sqrt{2n\tau^2\log(2p/\delta)}.
	\end{align*}
	\end{ceqn}
\end{lemma}

\begin{proof}
	For each $k=1,\ldots,p$, the centered random variable $\sum_{i=1}^n \varphi_k(X_i)$ is sub-Gaussian with variance factor $n \tau^2$. By the union bound and by Chernoff's inequality, we have, for each $t > 0$,
	\begin{ceqn}
	\begin{align*}
	\prob\left( \max_{k=1, \ldots,p} \left| \sum_{i=1}^n \varphi_k(X_i) \right|> t\right)
	&\le \sum_{k=1}^p \prob\left( \left| \sum_{i=1}^n \varphi_k(X_i) \right|> t\right) \\
	&\le 2p \exp\left(\frac{-t^2}{2n\tau^2}\right).
	\end{align*}
	\end{ceqn}
	Set $t =  \sqrt{2n\tau^2\log(2p/\delta)}$ to find the result.
\end{proof}

\begin{lemma}(Smallest eigenvalue lower bound)
	\label{lemma:smallest_eigen_value}
	Let $X_1,\ldots,X_n$ be independent and identically  distributed random variables in $(\mathcal{X}, \mathcal{A})$ with distribution $P$. Let $g = (g_1, \ldots, g_p)^T$ in $L_2(P)^p$ be such that the $p \times p$ Gram matrix $G = P(g g^T)$ satisfies $\lambda_{\min}(G) > 0$. Define the transformation $\tilde g = G^{-1/2} g$ and put $B_{\tilde g} := \sup_{x \in \mathcal{X}} \norm{\tilde g(x)}_2^2$. Let $\delta, \eta \in (0, 1)$. For $\delta \in (0, 1)$, the empirical Gram matrix $\hat{G}_n = P_n(g g^T)$ satisfies, with probability at least $1-\delta$,
	$$
		\lambda_{\min}(\hat G_n) 
		> \left(1 - \sqrt{2 B_{\tilde{g}} n^{-1} \log(p/\delta)}\right) \lambda_{\min}(G).
	$$
\end{lemma}

\begin{proof}
Suppose that the result is true in the special case that $G$ is the identity matrix. In case of a general Gram matrix $G$, we could then apply the result for the special case to the vector of functions $\tilde g = G^{-1/2} g$, whose Gram matrix is the identity matrix. We would get that $\lambda_{\min}( P_n(\tilde g \tilde g^T)) > 1-\eta$ with probability at least $1-\delta$. Since $P_n(\tilde g \tilde g^T) = G^{-1/2} \hat{G}_n G^{-1/2}$ and since $u^T G^{-1} u \le 1 / \lambda_{\min}(G)$ for every unit vector $u \in \reals^p$, we would have
\begin{ceqn}
	\begin{align*}
	\lambda_{\min} \left(P_n (\tilde g \tilde g^T)\right)
	&=
	\min_{u^Tu = 1} \left\{ u^T P_n(\tilde g \tilde g^T) u \right\} \\
	&=
	\min_{u^Tu = 1} \left\{
	\frac{(G^{-1/2}u)^T \hat{G}_n G^{-1/2}u}{(G^{-1/2}u)^T G^{-1/2} u} 
	u^T G^{-1} u 
	\right\} \\
	&\le \lambda_{\min}(\hat{G}_n) / \lambda_{\min}(G).
	\end{align*}
\end{ceqn}
	It would then follow that 
	$$ 
	\lambda_{\min}(\hat{G}_n) 
	\ge \lambda_{\min}(P_n (\tilde g \tilde g^T)) \, \lambda_{\min}(G)
	\ge (1-\eta) \lambda_{\min}(G),
	$$
as required. Hence we only need to show the result for $G = I$, in which case $\tilde g = g$.
	
	We apply the matrix Chernoff inequality in \citep[Theorem~5.1.1]{tropp2015introduction} to the random matrices $n^{-1} g(X_i)g(X_i)^T$. These matrices are independent and symmetric with dimension $p \times p$. Their minimum and maximum eigenvalues are between $0$ and $L = B_g/n$, with $B_g = \sup_{x \in \mathcal{X}} \lambda_{\max} (g(x) g(x)^T) = \sup_{x \in \mathcal{X}} \norm{g(x)}^2_2$. Their sum is equal to $P_n(g g^T) = \hat{G}_n$, whose expectation is $G = I$ by assumption. In the notation of the cited theorem, we have $\mu_{\min} = \lambda_{\min}(G) = 1$, and thus, by Eq.~(5.1.5) in that theorem, we have, for $\eta \in [0, 1)$,
	$$
		\prob\{\lambda_{\min}(\hat{G}_n) \le 1 - \eta\} \le p \left[\frac{\exp(-\eta)}{(1-\eta)^{1-\eta}}\right]^{n/B_g}.
	$$
	The term in square brackets is bounded above by $\exp(-\eta^2/2)$. Indeed, we have, for $\eta \in [0, 1)$,
	$$
		\frac{e^{-\eta}}{(1-\eta)^{1-\eta}} 
		= \exp \{-\eta - (1-\eta) \log(1-\eta) \}
	$$
	and
	\begin{ceqn}
	\begin{align*}
		\eta + (1-\eta) \log(1-\eta) 
		&= \eta - (1-\eta) \int_0^\eta \frac{\diff t}{1-t} \\
		&= \int_0^{\eta} \left(1 - \frac{1-\eta}{1-t}\right) \, \diff t  \\
		&= \int_0^{\eta} \frac{\eta - t}{1-t} \, \diff t \\
		&\ge \int_0^\eta (\eta - t) \, \diff t 
		= \frac{\eta^2}{2}.
	\end{align*}
	\end{ceqn}
	It follows that
	$$
		\prob\{\lambda_{\min}(\hat{G}_n) \le 1 - \eta\} 
		\le p \exp \left( - \frac{\eta^2 n}{2 B_g} \right).
	$$
		Solving $p \exp \left( - \frac{\eta^2 n}{2 B_g} \right) = \delta$ in $\eta$, we find that, with probability at least $1-\delta$,
	$$
		\lambda_{\min}(\hat{G}_n) > 1 - \sqrt{2B_g n^{-1} \log(p/\delta)}.
	$$
\end{proof}

\begin{lemma}[Upper bound of moments]
	\label{lem:94}
	Let $X$ be a random variable such that $\expec(|X|^{2p}) \le 2^{p+1} p!$ for every integer $p \ge 1$. Then
	\begin{ceqn}
	\begin{equation}
	\label{eq:52}
		\forall \lambda \in \reals, \qquad
		1 + \sum_{k=2}^\infty \frac{\lambda^k}{k!} \expec(|X|^k)
		\le \exp(9 \lambda^2 / 4),
	\end{equation}
	\end{ceqn}
	in which it is implicitly understood that the series on the left-hand side converges.
\end{lemma}

\begin{proof}

	Let $\lambda \in \reals$. We split the series in terms with even and odd indices $k$, leading to
	\begin{align*}
		&1 + \sum_{k=2}^\infty \frac{\lambda^k}{k!} \expec(|X|^k)
		\\
		&= 1 + \sum_{p=1}^\infty \frac{\lambda^{2p}}{(2p)!} \expec(|X|^{2p})
		+ \sum_{p=1}^\infty \frac{\lambda^{2p+1}}{(2p+1)!} \expec(|X|^{2p+1}).
	\end{align*}
	We will bound the series on the odd indices in terms of the series on the even indices.
	
	Since the geometric mean of two nonnegative numbers is bounded by their arithmetic mean, we have, for all $x \ge 0$ and all $a > 0$,
	\begin{ceqn}
	\begin{align*}
		|x| 
		= \sqrt{ \frac{1}{a} \cdot a x^2 } 
		\le \frac{1}{2} \left( \frac{1}{a} + a x^2 \right).
	\end{align*}
	\end{ceqn}
	Applying the previous inequality to $x = \lambda X$ and scalars $a_p > 0$ to be chosen later,
	\begin{align*}
		&\sum_{p=1}^\infty \frac{\lambda^{2p+1}}{(2p+1)!} \expec(|X|^{2p+1}) \\
		&\le \sum_{p=1}^\infty \frac{\lambda^{2p}}{(2p+1)!} \expec \left[ |X|^{2p} \frac{1}{2} \left(\frac{1}{a_p} + a_p (\lambda X)^2\right) \right] \\
		&= \sum_{p=1}^\infty \frac{\lambda^{2p}}{2a_p} \frac{\expec ( |X|^{2p})}{(2p+1)!} 
		+ \sum_{p=1}^\infty \frac{a_p}{2} \frac{\lambda^{2p+2}}{(2p+1)!} \expec ( |X|^{2p+2}) \\
		&= \sum_{p=1}^\infty \frac{\lambda^{2p}}{2a_p} \frac{\expec ( |X|^{2p})}{(2p+1)!} 
		+ \sum_{p=2}^\infty \frac{a_{p-1}}{2} \frac{\lambda^{2p}}{(2p-1)!} \expec ( |X|^{2p}) \\
		&= \sum_{p=1}^\infty \left( \frac{1}{2a_p(2p+1)} + p a_{p-1} \1_{\{p \ge 2\}} \right) \frac{\lambda^{2p}}{(2p)!} \expec(|X|^{2p}).
	\end{align*}
	Here, $\1$ denotes an indicator function.
	We obtain
	\begin{align*}
	& \sum_{k=2}^\infty \frac{\lambda^k}{k!} \expec(|X|^k) \\
	&\le \sum_{p=1}^\infty \left( 1 + \frac{1}{2a_p(2p+1)} + p a_{p-1} \1_{\{p \ge 2\}} \right) \frac{\lambda^{2p}}{(2p)!} \expec(|X|^{2p}).
	\end{align*}
	Define $b_p = a_p (2p+1)$ and use the hypothesis on $\expec(|X|^{2p})$ to see that, for any constants $b_p > 0$,
	\begin{align*}
	&1 + \sum_{k=2}^\infty \frac{\lambda^k}{k!} \expec(|X|^k) \\
	&\le 1 + \sum_{p=1}^\infty \left( 1 + \frac{1}{2b_p} + \frac{p}{2p-1} b_{p-1} \1_{\{p \ge 2\}} \right) \frac{\lambda^{2p}}{(2p)!} 2^{p+1} p!.
	\end{align*}
	The objective is to find a constant $c > 0$ as small as possible and such that the right-hand side is bounded by $\exp(c \lambda^2) = 1 + \sum_{p=1}^\infty c^p \lambda^{2p} / p!$. Comparing coefficients, this means that we need to determine scalars $b_p > 0$ and $c > 0$ in such a way that for all  $p = 1, 2,\ldots$
	\begin{ceqn}
	\begin{equation*}
		\left( 1 + \frac{1}{2b_p} + \frac{p}{2p-1} b_{p-1} \1_{\{p \ge 2\}} \right) \frac{2^{p+1} p!}{(2p)!} 
		\le \frac{c^p}{p!},
	\end{equation*}
	\end{ceqn}
	or, equivalently,
	\begin{ceqn}
	\begin{equation*}
	\label{eq:cpp}
		\left(2 + \frac{1}{b_p} + \frac{2p}{2p-1} b_{p-1} \1_{\{p \ge 2\}} \right)
		\prod_{j=1}^p \frac{j}{p+j} \le (c/2)^p.
	\end{equation*}
	\end{ceqn}
	The case $p = 1$ gives 
	\begin{ceqn}
	\begin{equation}
	\label{eq:cpp1}
		2 + \frac{1}{b_1} \le c,
	\end{equation}
	\end{ceqn}
	showing that, with this proof technique, we will always find $c > 2$. Setting $b_p \equiv b > 0$ for all integer $p \ge 1$ and $c = 2 + 1/b$, inequality~\eqref{eq:cpp1} is automatically satisfied, so it remains to find $b > 0$ such that forall $p = 2, 3, \ldots$
	\begin{align*}
		\left(2 + \frac{1}{b} + \frac{2p}{2p-1} b\right) \prod_{j=1}^p \frac{j}{p+j} \le (c/2)^p \quad \text{with } c = 2 + \frac{1}{b}.
	\end{align*}
	The left-hand side is decreasing in $p$ whereas the right-hand side is increasing in $p$. It is thus sufficient to have the inequality satisfied for $p = 2$, i.e.,
	\begin{ceqn}
	\begin{equation}
	\label{eq:cpp2}
		\left( 2 + \frac{1}{b} + \frac{4b}{3} \right) \frac{1}{6} 
		\le \left(1 + \frac{1}{2b}\right)^2.
	\end{equation}
	\end{ceqn}
	Equating both sides leads to a nonlinear equation in $b$ that can be solved numerically, giving the root $b \approx 4.006156$. With $b = 4$, inequality~\eqref{eq:cpp2} is satisfied, as can be checked directly ($91/72 \le 81/64$). We conclude that $c = 2 + 1/4 = 9/4$ is a valid choice.
\end{proof}

 Note that the series in \eqref{eq:52} starts at $k = 2$. If also $\expec(X) = 0$, the left-hand side in \eqref{eq:52} is an upper bound for $\expec(\exp(\lambda X))$, and we obtain the following corollary.

\begin{corollary}
	\label{cor:5}
	Let $Z$  be a centered random variable such that $$\forall p \in \mathbb{N}^\star, \quad \expec(|Z|^{2p}) \le 2^{p+1} p!$$
Then $\log \expec(\exp(\lambda Z)) \le 9 \lambda^2/4$ for all $\lambda \in \reals$, i.e., $Z \in \cG(9/2)$.
\end{corollary}

\begin{lemma}
	\label{lem:subG}
	Let $(X, Y)$ be a pair of uncorrelated random variables. If $X \in \cG(\nu)$ and $|Y| \le \kappa$ for some $\nu > 0$ and $\kappa > 0$, then $XY \in \cG((9/2)\kappa^2 \nu)$.
\end{lemma}

\begin{proof}
	The random variable $X/\sqrt{\nu}$ is sub-Gaussian with variance factor $1$. As on page~25 in \citet{boucheron+l+m:2013}, this implies that $\prob(|X/\sqrt{\nu}| > t) \le 2 \exp(-t^2/2)$ for all $t \ge 0$ and thus $\expec[|X/\sqrt{\nu}|^{2p}] \le 2^{p+1} p!$ for all integer $p \ge 1$ (see \cite[Theorem~2.1]{boucheron+l+m:2013}).
	
	Let $Z = XY / (\sqrt{\nu} \kappa)$.	Since $X$ is centered and $X$ and $Y$ are uncorrelated, $XY$ is centred too, and therefore also $Z$. From the previous paragraph, we have $\expec(|Z|^{2p}) \le \expec(|X/\sqrt{\nu}|^{2p}) \le 2^{p+1} p!$ for all integer $p \ge 1$. Corollary~\ref{cor:5} gives for all $\lambda \in \reals$ that $\log \expec(\exp(\lambda Z)) \le 9\lambda^2/4$, from which
	\begin{ceqn}
	\begin{align*}
		\log \expec(\exp(\lambda XY))
		= \log \expec(\exp(\lambda \sqrt{\nu} \kappa Z))
		\le \frac{9}{4} \lambda^2 \nu \kappa^2. 
	\end{align*}
	\end{ceqn}
\end{proof}

\begin{lemma}[Upper bound for norm-subGaussian random vector] \label{lemma:up_bound_nsgd}
	Let $X$ be a $d$-dimensional random vector with zero-mean and such that $\prob(\norm{X}_2 \ge t) \le 2 \exp \left( -t^2/(2\sigma^2) \right)$ for all $t \ge 0$. Then the random matrix $Y$ defined by
	\begin{ceqn}
	\begin{align} \label{eq:def_Y}
	Y =\begin{bmatrix} 0 & X^T \\ X & 0 \end{bmatrix} \in \reals^{(d+1) \times (d+1)}
	\end{align}
	\end{ceqn}
	 satisfies $\expec(\exp(\theta Y)) \preceq \exp(c\theta^2\sigma^2) I$ for any $\theta \in \reals$, with $c = 9/4$, where $I$ denotes the identity matrix.
\end{lemma}

\begin{proof}
	The non-zero eigenvalues of $Y$ are $\norm{X}$ and $-\norm{X}$. The non-zero eigenvalues of $Y^k$ are thus $\norm{X}^k$ and $(-\norm{X})^k$ for integer $k \ge 1$. It follows that $Y^k \preceq \norm{X}^k I$ for all integer $k \ge 1$, and therefore also $\expec(Y^k) \preceq \expec(\norm{X}^k) I$ for all integer $k \ge 1$. Furthermore, the operator norm of $Y^k$ is bounded by $\norm{Y^k} \le \norm{X}^k$.
		
Since $\expec(Y) = 0$, we get, for any $\theta \in \reals$,
	\begin{align*}
		&\expec(\exp(\theta Y))
		= I + \sum_{k=2}^\infty \frac{\theta^k}{k!} \expec(Y^k) \\
		&\preceq \left( 1 + \sum_{k=2}^\infty \frac{\theta^k}{k!} \expec(\norm{X}^k) \right) I
		= \left( 1 + \sum_{k=2}^\infty \frac{(\theta\sigma)^k}{k!} \expec(\xi^k) \right) I,
	\end{align*}
	where $\xi = \norm{X}/\sigma$. The first series converges in operator norm since $\norm{ \expec(Y^k) } \le \expec(\norm{Y^k}) \le \expec(\norm{X}^k)$. 
	
	By assumption, $\prob(\xi > t) = \prob(\norm{X} \ge \sigma t) \le 2e^{-t^2/2}$ for all $t \ge 0$ and thus $\expec(|\xi|^{2p}) \le 2^{p+1} p!$ for all integer $p \ge 1$ 
	 But then we can apply Lemma~\ref{lem:94} with $\lambda = \theta \sigma$ and $X = \xi$, completing the proof.
\end{proof}

The following result is a special case of \citet[Corollary~7]{jin+etal:2019}. Our contribution is to make the constant $c$ in the cited result explicit. In passing, we correct an inaccuracy in the proof of \citet[Lemma~4]{jin+etal:2019}, in which it was incorrectly claimed that the odd moments of a certain random matrix Y as in our Lemma \ref{lemma:up_bound_nsgd} are all zero.

\begin{lemma}[Hoeffding inequality for norm-subGaussian random vectors]
	\label{lem:HffnsbG}
	Let the $d$-dimensional random vectors $Z_1,\ldots,Z_n$ be independent, have mean zero, and satisfy
	\begin{ceqn}
	\begin{equation}
	\label{eq:HffnsbG}
		\forall t \ge 0, \forall i = 1,\ldots,n, \quad
		\prob(\norm{Z_i}_2 \ge t) \le 2 \exp \left( - \tfrac{t^2}{2\sigma^2} \right)
	\end{equation}
	\end{ceqn}
	for some $\sigma > 0$. Then for any $\delta > 0$, with probability at least $1-\delta$, we have
	$$
		\norm{\sum_{i=1}^n Z_i}_2 \le 3 \sqrt{n \sigma^2 \log(2d/\delta)}.
	$$
\end{lemma}

\begin{proof}
	Given Corollary~7 in \citet{jin+etal:2019}, the only thing to prove is that their constant can be set equal to $3$. Their Corollary~7 follows from their Lemma~6 
in which it is shown that when the matrix $Y$ defined in \eqref{eq:def_Y} satisfies
	$$
		\forall \theta \in \reals, \qquad
		\expec[\exp(\theta Y)] \preceq \exp(c \theta^2 \sigma^2) I,
	$$
then we have for any  $\theta>0$, with probability at least $(1-\delta)$,
\begin{ceqn}
\begin{align*}
\norm{\sum_{i=1}^n Z_i}_2 \le c \cdot \theta n \sigma^2 + \frac{1}{\theta}\log(2d/\delta).
\end{align*}	
\end{ceqn}
Taking $\theta = \sqrt{\log(2d/\delta)/(cn\sigma^2)}$ yields
\begin{ceqn}
\begin{align*}
\norm{\sum_{i=1}^n Z_i}_2 \le 2\sqrt{c} \sqrt{n \sigma^2 \log(2d/\delta)},
\end{align*}	
\end{ceqn}
and we conclude with Lemma \ref{lemma:up_bound_nsgd} ($c=9/4,2\sqrt{c}=3$).
\end{proof}

\section{Proof of Theorem \ref{th:ols}}
\label{app:th:ols}

The proof is organized as follows. We first provide an upper bound on the error (Step~1). This bound involves the norm of the error made on the rescaled coefficients and is controlled in Step~2. Then (Step~3), we construct an event that has probability at least $1-\delta$ on which we can control the terms that appear in the upper bound of Step~2. Collecting all the inequalities, we will arrive at the stated bound (Step~4).
\smallskip

\emph{Step~1.} ---
Since $f = P(f) + \beta^\star(f)^T h + \epsilon$, the oracle estimate of $P(f)$, which uses the unknown, optimal coefficient vector $\beta^\star(f)$, is
$$
\hat{\alpha}_n^{\mathrm{or}}(f)
= P_n[f - \beta^\star(f)^T h]
= P(f) + P_n(\epsilon).
$$
The difference between the OLS and oracle estimates is
$$
\hat{\alpha}_n^{\mathrm{ols}}(f) - \hat{\alpha}_n^{\mathrm{or}}(f)
= \bigl(\beta^\star(f) - \bols(f)\bigr)^T \, P_n(h).
$$
Let $G = P(h h^\top)$ be the $m \times m$ Gram matrix. By assumption, $G$ is positive definite. Write
\begin{ceqn}
\begin{align*}
\eta^\star &= G^{1/2} \beta^\star(f), & \hat{\eta} &= G^{1/2} \bols(f), & \hbar &= G^{-1/2} h(f).
\end{align*}
\end{ceqn}
The estimation error of the OLS estimator can thus be decomposed as
\begin{ceqn}
\begin{align*}
&n \bigl( \hat{\alpha}_n^{\mathrm{ols}}(f) -  P(f) \bigr) \\
&= n \bigl(\hat{\alpha}_n^{\mathrm{or}}(f) - P(f) \bigr) 
+ \bigl(\beta^\star(f) - \bols(f)\bigr)^T \, n \, P_n(h) \\
&= \sum_{i=1}^n \epsilon(X_i) + 
\bigl(\beta ^*(f) - \bols(f)\bigr)^T \sum_{i=1}^n h(X_i) \\
&= \sum_{i=1}^n \epsilon(X_i) +
(\eta^\star - \hat{\eta})^\top \sum_{i=1}^n \hbar(X_i).
\end{align*}
\end{ceqn}
By the triangle and Cauchy--Schwarz inequalities,
\begin{ceqn}
\begin{equation}
\label{eq:alpha:ols:bound:eta}
n \abs{ \hat \alpha_n^{\mathrm{ols}} (f) -  P(f) }
\le \abs{ \sum_{i=1}^n \epsilon(X_i) }
+ \norm{\eta^\star - \hat{\eta}}_2 \norm{ \sum_{i=1}^n \hbar(X_i) }_2.
\end{equation}
\end{ceqn}

\emph{Step~2.} --- 
We will show that, if $\lambda_{\min}(P_n(\hbar \hbar^\top)) > \norm{P_n(\hbar)}_2^2$, then
\begin{ceqn}
\begin{equation}
\label{ineq:etaols}
	\norm{\hat{\eta} - \eta^\star}_2
	\le
	\frac%
	{\norm{P_n(\hbar \epsilon)}_2 + \norm{P_n(\hbar)}_2 \abs{P_n(\epsilon)}}%
	{\lambda_{\min}(P_n(\hbar \hbar^\top)) - \norm{P_n(\hbar)}_2^2}.
\end{equation}
\end{ceqn}
and thus, by \eqref{eq:alpha:ols:bound:eta},
\begin{multline}
	\label{eq:alpha:ols:bound:pieces}
	\abs{ \hat \alpha_n^{\mathrm{ols}} (f) -  P(f) } \\
	\le \abs{P_n(\epsilon)} + 
	\frac%
	{\norm{P_n(\hbar \epsilon)}_2 + \norm{P_n(\hbar)}_2 \abs{P_n(\epsilon)}}%
	{\lambda_{\min}(P_n(\hbar \hbar^\top)) - \norm{P_n(\hbar)}_2^2}
	\norm{P_n(\hbar)}_2 
\end{multline}
\smallskip

\emph{Step~2.1} ---
Considered the column-centered $n \times m$ design matrices
\begin{ceqn}
\begin{align*}
H_c &= H - \1_n P_n(h)^\top 
= \bigl(h_j(X_i) - P_n(h_j)\bigr)_{i,j}, \\
\Hbar_c = H_c G^{-1/2} 
&= \Hbar - \1_n P_n(\hbar)^\top = \bigl(\hbar_j(X_i) - P_n(\hbar_j)\bigr)_{i,j}.
\end{align*}
\end{ceqn}
Since $\Hbar^\top \1_n = n P_n(\hbar)$, we have
\begin{ceqn}
\begin{align*}
\Hbar_c^\top \Hbar_c
&= \Hbar^\top \Hbar - n P_n(\hbar) P_n(\hbar)^\top \\
&= n \left( P_n(\hbar \hbar^\top) - P_n(\hbar) P_n(\hbar)^\top\right).
\end{align*}
\end{ceqn}
As a consequence, for $u \in \reals^m$,
\begin{ceqn}
\begin{align*}
u^\top \Hbar_c^\top \Hbar_c u
&= n \left( u^\top P_n(\hbar \hbar^\top) u - (P_n(\hbar)^\top u)^2 \right) \\
&\ge n \left( \lambda_{\min}(P_n(\hbar \hbar^\top)) - \norm{P_n(\hbar)}_2^2 \right) \norm{u}_2^2. 
\end{align*}
\end{ceqn}
by the Cauchy--Schwarz inequality. In particular, $u^\top \Hbar_c^\top \Hbar_c u$ is non-zero for non-zero $u \in \reals^m$, so that $\Hbar_c^\top \Hbar_c$ is invertible, and so is the matrix 
$$ 
H_c^\top H_c = G^{1/2} \Hbar_c \Hbar_c G^{1/2}. 
$$
Also, the smallest eigenvalue of $\Hbar_c^\top \Hbar_c$ is bounded from below by
$$
\lambda_{\min}(\Hbar_c^\top \Hbar_c)
\ge n \left( \lambda_{\min}(P_n(\hbar \hbar^\top)) - \norm{P_n(\hbar)}_2^2 \right) > 0.
$$
The largest eigenvalue of the inverse matrix $(\Hbar_c^\top \Hbar_c)^{-1}$ is then bounded from above by
\begin{ceqn}
\begin{equation}
\label{eq:lambdamaxHbc}
\lambda_{\max}\bigl((\Hbar_c^\top \Hbar_c)^{-1}\bigr)
\le \frac{1}{n \left( \lambda_{\min}(P_n(\hbar \hbar^\top)) - \norm{P_n(\hbar)}_2^2 \right)}.
\end{equation}
\end{ceqn}
\smallskip

\emph{Step~2.2.} ---
Write $\epsilon_c^{(n)} = (\epsilon(X_i) - P_n(\epsilon))_{i=1}^n$ for the centered vector of error terms. Recall $f_c^{(n)} = (f(X_i) - P_n(f))_{i=1}^n$, the centered vector of samples from the integrand. As $f = P(f) + h^\top \beta^*(f) + \epsilon$, we have
$$
f_c^{(n)} = H_c \beta^\star(f) + \epsilon_c^{(n)}.
$$
From the characterization \eqref{eq:ols:c} of the OLS estimate of the coefficient vector and since $H_c^\top H_c$ is invertible,
\begin{ceqn}
\begin{align*}
\bols(f) 
&= (H_c^\top H_c)^{-1} H_c^\top f_c^{(n)} \\
&= (H_c^\top H_c)^{-1} H_c^\top \left( H_c \beta^\star(f) + \epsilon_c^{(n)} \right) \\
&= \beta^\star(f) + (H_c^\top H_c)^{-1} H_c^\top \epsilon_c^{(n)}.
\end{align*}
\end{ceqn}
We obtain
\begin{ceqn}
\begin{align}
\nonumber
\hat{\eta} - \eta^\star 
&= G^{1/2} \left(\bols(f) - \beta^\star(f)\right) \\
\nonumber
&= G^{1/2} (H_c^\top H_c)^{-1} H_c^\top \epsilon_c^{(n)} \\
\label{eq:etaepsilon}
&= (\Hbar_c^\top \Hbar_c)^{-1} \Hbar_c^\top \epsilon_c^{(n)}.
\end{align}
\end{ceqn}
\smallskip

\emph{Step~2.3.} ---
We combine the results from Steps~2.1 and~2.2. From the upper bound \eqref{eq:lambdamaxHbc} and the identity~\eqref{eq:etaepsilon}, we obtain
\begin{ceqn}
\begin{equation*}
\norm{\hat{\eta} - \eta^\star}_2
\le
\frac{\norm{\Hbar_c^\top \epsilon_c^{(n)}}_2}%
{n \left( \lambda_{\min}(P_n(\hbar \hbar^\top)) - \norm{P_n(\hbar)}_2^2 \right)}
\end{equation*}
\end{ceqn}
Finally, as $\Hbar_c = (\hbar_j(X_i) - P_n(\hbar_j))_{i,j}$, we find
\begin{ceqn}
\begin{align*}
n^{-1} \norm{\Hbar_c^\top \epsilon_c^{(n)}}_2
&= n^{-1} \norm{\sum_{i=1}^n \hbar(X_i) \epsilon(X_i) - P_n(\hbar) \sum_{i=1}^n \epsilon(X_i)}_2 \\
&= \norm{P_n(\hbar \epsilon) - P_n(\hbar) P_n(\epsilon)}_2 \\
&\le \norm{P_n(\hbar \epsilon)}_2 + \norm{P_n(\hbar)}_2 \abs{P_n(\epsilon)}.
\end{align*}
\end{ceqn}
Equation~\eqref{ineq:etaols} follows.
\smallskip

\emph{Step 3. ---}
In view of \eqref{eq:alpha:ols:bound:pieces}, we need to ensure that $\abs{P_n(\epsilon)}$, $\norm{P_n(\hbar)}$ and $\norm{P_n(\hbar \epsilon)}_2$ are small and that $\lambda_{\min}(P_n(\hbar \hbar^\top))$ is large. Let $\delta>0$. We construct an event with probability at least $1-\delta$ on which four inequalities hold simultaneously. Recall $B = \sup_{x \in \mathcal{X}} \norm{\hbar(x)}_2^2$, defined in \eqref{eq:B}.
\smallskip

\emph{Step 3.1. ---}
Because $\epsilon \in \subG(\tau^2)$, Chernoff's inequality (or Lemma~\ref{lemma:concentration} with $p = 1$) implies that with probability at least $1 - \delta/4$,
\begin{ceqn}
\begin{equation}
\label{eq:3_eps}
	\abs{\sum_{i=1}^n \epsilon (X_i)}
	\leq \sqrt{2n \tau^2\log(8/\delta)}.
\end{equation}
\end{ceqn}
\smallskip

\emph{Step 3.2. ---}
For the term $\norm{\sum_{i=1}^n \hbar(X_i)}_2$, we apply the vector Bernstein bound in \cite[Lemma~9]{hsu+etal:2014}. On the one hand $\sup_{x \in \mathcal{X}} \norm{\hbar(x)}_2 \le \sqrt{B}$ and on the other hand
$$
	\sum_{i=1}^n \expec[ \norm{\hbar(X_i)}_2^2 ] 
	= \sum_{i=1}^n \sum_{j=1}^m P(\hbar_j^2)
	= nm.
$$
The cited vector Bernstein bound gives
\begin{ceqn}
\begin{equation*}
	\forall t \ge 0,
	\prob\left[ 
	\norm{\sum_{i=1}^n \hbar(X_i)}_2 
	> \sqrt{nm} \left(1 + \sqrt{8t}\right) + \tfrac{4}{3} t \sqrt{B} 
	\right] 
	\le e^{-t}.
\end{equation*}
\end{ceqn}
Setting $t = \log(4/\delta)$, we find that, with probability at least $1-\delta/4$, we have
$$
	\norm{\sum_{i=1}^n \hbar(X_i)}_2 
	\le \sqrt{nm} \left(1 + \sqrt{8\log(4/\delta)}\right) + \tfrac{4}{3} \log(4/\delta) \sqrt{B}. 
$$
Since $\log(4/\delta) \ge \log(4)$, we have 
\begin{ceqn}
\begin{align*}
1 + \sqrt{8 \log(4/\delta)} \le 4 \sqrt{\log(4/\delta)}
\end{align*}
\end{ceqn}
and thus
\begin{ceqn}
\begin{align*}
	\norm{\sum_{i=1}^n \hbar(X_i)}_2 
	&\le 4 \sqrt{nm \log(4/\delta)} + \tfrac{4}{3} \log(4/\delta) \sqrt{B} \\
	&= 4 \sqrt{\log(4/\delta)} \left( \sqrt{nm} + \tfrac{1}{3} \sqrt{B \log(4/\delta)} \right). 
\end{align*}
\end{ceqn}
The condition on $n$ easily implies that
$$
	\tfrac{1}{3} \sqrt{B \log(4/\delta)} \le \tfrac{1}{4} \sqrt{nm}
$$
and thus
\begin{ceqn}
\begin{equation}
\label{eq:Pnhb}
	\norm{\sum_{i=1}^n \hbar(X_i)}_2 \le 5 \sqrt{nm \log(4/\delta)}.
\end{equation}
\end{ceqn}
\smallskip

\emph{Step 3.3. ---}
To control $\norm{\sum_{i=1}^n \hbar(X_i) \epsilon(X_i)}_2$, we apply Lemma~\ref{lem:HffnsbG} with $Z_i = \hbar(X_i) \epsilon(X_i)$. The random vectors $\hbar(X_i) \epsilon(X_i)$ for $i = 1, \ldots, n$ are independent and identically distributed and have mean zero. Since $\norm{\hbar(X_i)}_2 \le \sqrt{B}$ by \eqref{eq:B} and since $\epsilon \in \mathcal{G}(\tau^2)$ by Assumption~\ref{ass_1_sub_gauss}, we have, for all $t > 0$,
\begin{ceqn}
\begin{align*}
	\prob[ \norm{\hbar(X_i) \epsilon(X_i)}_2 > t ]
	&\le
	\prob[ \sqrt{B} \abs{\epsilon(X_i)} > t ] \\
	&\le 
	2 \exp\left( - \tfrac{t^2}{2 B \tau^2} \right),
\end{align*}
\end{ceqn}
and \eqref{eq:HffnsbG} holds with $\sigma^2 = B \tau^2$. Lemma~\ref{lem:HffnsbG} then implies that, with probability at least $1-\delta/4$ and $c=3$ that
\begin{ceqn}
\begin{equation}
\label{eq:hbareps}
	\norm{ \sum_{i=1}^n \hbar(X_i) \epsilon(X_i) }_2
	\le c \sqrt{n B \tau^2 \log(8m/\delta)}. 
\end{equation}
\end{ceqn}

\smallskip

\emph{Step 3.4. ---}
Recall the $n \times m$ matrix $H = (h_j(X_i))_{i,j}$ and put 
$$
	\Hbar = H G^{-1/2} = (\hbar_j(X_i))_{i,j}. 
$$
The empirical Gram matrix of the vector $\hbar = (\hbar_1, \ldots, \hbar_m)^\top \in L_2(P)^m$ based on the sample $X_1, \ldots, X_n$ is 
$$ 
	P_n(\hbar \hbar^\top) = n^{-1} \Hbar^T \Hbar. 
$$
We apply Lemma~\ref{lemma:smallest_eigen_value} with $g = \tilde{g} = \hbar$, $p = m$, and $\delta$ replaced by $\delta/4$. We find that, with probability at least $1-\delta/4$,
\begin{ceqn}
\begin{align}
	\nonumber
	\forall u \in \reals^m, \
	\norm{\Hbar u}_2^2 
	&= n \, u^\top P_n(\hbar \hbar^\top) u \\
\label{eq:0_ev}	
	&\ge n \left( 1 - \sqrt{2 B n^{-1} \log(4m/\delta)} \right) \norm{u}^2_2.
\end{align}
\end{ceqn}
Since $P_n(\hbar \hbar^\top) = n^{-1} \Hbar^\top \Hbar$, it follows that
\begin{ceqn}
\begin{equation}
\label{eq:lambdaminPnhb}
	\lambda_{\min}(P_n(\hbar \hbar^\top))
	\ge 1 - \sqrt{2 B n^{-1} \log(4m/\delta)}
	\ge \tfrac{2}{3}
\end{equation}
\end{ceqn}
as the assumption on $n$ implies that $2 B n^{-1} \log(4m/\delta) \le 1/9$.
\smallskip

By the union bound, the inequalities \eqref{eq:3_eps}, \eqref{eq:Pnhb}, \eqref{eq:hbareps}, and \eqref{eq:0_ev} hold simultaneously on an event with probability at least $1-\delta$. For the remainder of the proof, we work on this event, denoted by $E$.
\smallskip

\emph{Step~4.} --- 
We combine the bound \eqref{eq:alpha:ols:bound:pieces} on the estimation error with the bounds valid on the event $E$ constructed in Step~3. By \eqref{eq:lambdaminPnhb}, we have
$$
	\lambda_{\min}(P_n(\hbar \hbar^\top)) - \norm{P_n(\hbar)}_2^2
	\ge \tfrac{2}{3} - 25 m n^{-1} \log(4/\delta)
	\ge \tfrac{1}{3} 
$$
since the assumption on $n$ implies that $25 m n^{-1} \log(4/\delta) \le 1/3$. As $B \ge m \ge 1$, we have

\begin{ceqn}
\begin{align*}
	&\lefteqn{
		\norm{P_n(\hbar \epsilon)}_2 + \norm{P_n(\hbar)}_2 \abs{P_n(\epsilon)}
	} \\
	&\le c \sqrt{n^{-1} B \tau^2 \log(8m/\delta)} + \\
	&5 \sqrt{n^{-1} m \log(4/\delta)} \cdot \sqrt{2 n^{-1} \tau^2 \log(8/\delta)} \\
	&\le \sqrt{n^{-1} B \tau^2 \log(8m/\delta)}
	\left( c + 5\sqrt{2n^{-1} \log(4/\delta)} \right) \\
	&\leq (c+\sqrt{2/3}) \sqrt{n^{-1} B \tau^2 \log(8m/\delta)},
\end{align*}
\end{ceqn}
since, by assumption, $n \ge 75m \log(4/\delta)$ which implies that $\sqrt{n^{-1} \log(4/\delta)} \leq 1/(5\sqrt{3})$. We find
\begin{ceqn}
\begin{align*}
	&\abs{\aols(f) - P(f)}
	\le \sqrt{2\tau^2 n^{-1} \log(8/\delta)} + \\
	& \frac{1}{1/3} \cdot (c+\sqrt{2/3}) \sqrt{n^{-1} B \tau^2 \log(8m/\delta)} 
	\cdot 5 \sqrt{m n^{-1} \log(4/\delta)} \\
	&= \sqrt{2\tau^2 n^{-1} \log(8/\delta)}
	+ \\ &15 (c+\sqrt{2/3}) n^{-1} \sqrt{B \tau^2 m \log(8m/\delta) \log(4/\delta)},
\end{align*}
\end{ceqn}
and the value $c=3$ gives $15 (c+\sqrt{2/3}) \approx 57.2 <58$
which is the bound stated in Theorem~\ref{th:ols}. \qed
\section{Proof of Theorem \ref{th:lasso}}

For a vector $\beta \in \reals^m$ and for a non-empty set $S \subset \{1, \ldots, m\}$, write $\beta_S = (\beta_k)_{k \in S}$.
For any matrix $A \in \reals^{n\times m}$ and $k\in  \{1,\ldots, m\}$, let $A_k $ denote its $k$-th column and {if $S = \{k_1,\ldots,k_\ell\} \subset \{1,\ldots,m\}$ with $k_1 < \ldots < k_\ell$, write $A_S = (A_{k_1},\ldots,A_{ k_\ell}) \in \reals^{n \times \ell}$.}

The proof is organized in a similar way as the one of Theorem~\ref{th:ols}. We first provide an initial upper bound on the error (Step~1). Then we construct an event that (Step~2) has probability at least $1-\delta$  and (Steps~3, 4, 5) on which we can control each of the terms of the previous upper bound. The combination of all steps to deduce the final statement is made clear in Step~6.
\smallskip

\emph{Step~1.} ---
As in the proof of Theorem \ref{th:ols}, with $\bols(f)$ replaced by $\blasson(f)$, the estimation error of the LASSO estimator can be decomposed as
\begin{multline*}
n \left( \hat{\alpha}_n^{\mathrm{lasso}}(f) -  P(f) \right)  \\
= \sum_{i=1}^n \epsilon(X_i) + 
\bigl(\beta ^*(f) - \blasson(f)\bigr)^T \sum_{i=1}^n h(X_i).
\end{multline*}
Writing $\hat{u} = \blasson(f) - \beta^\star(f)$, we get, by the triangle and Hölder inequalities,
\begin{ceqn}
\begin{equation}
\label{eq:alpha:lasso:bound}
n \left| \hat \alpha_n^{\mathrm{lasso}} (f) -  P(f) \right|
\le \left| \sum_{i=1}^n \epsilon(X_i) \right|
+ \norm{\hat{u}}_1 \max_{k=1,\ldots,m} \left\lvert\sum_{i=1}^n h_k(X_i)\right\rvert.
\end{equation}
\end{ceqn}
\smallskip

\emph{Step 2.} --- Let $\delta > 0$. We construct an event, $E$, with probability at least $1-\delta$ on which four inequalities, namely \eqref{eq:0_eps_lasso}, \eqref{eq:1_h_lasso}, \eqref{eq:2_heps_lasso} and \eqref{eq:3_ev_lasso_2}, hold simultaneously.

\begin{itemize}[leftmargin=1pc]
	
	\item
	Since $\epsilon \in \subG( \tau^2)$, we can apply Lemma \ref{lemma:concentration} with $p=1$ to get that, with probability at least $1 - \delta/4$,
	\begin{ceqn}
	\begin{equation}
	\label{eq:0_eps_lasso}
	\left|\sum_{i=1}^n \epsilon (X_i) \right|
	\leq \sqrt{2n \tau^2\log(8/\delta)}.
	\end{equation}
	\end{ceqn}
	\item
	In view of \citep[Lemma~2.2]{boucheron+l+m:2013} and Assumption~\ref{ass_2_bounded_control}, we have $h_k\in \subG(U_h^2)$ for all $k=1,\ldots, m$. Hence we can apply Lemma~\ref{lemma:concentration} with $p=m$ to get that, with probability at least $1-\delta/4$,
	\begin{ceqn} 
	\begin{equation}
	\label{eq:1_h_lasso}
	\max_{k=1,\ldots, m} \left| \sum_{i=1}^n h_k(X_i)  \right| 
	\leq \sqrt{2n U_h^2\log(8m/\delta)}.
	\end{equation}
	\end{ceqn}
	\item
	By virtue of Assumptions~\ref{ass_1_sub_gauss} and \ref{ass_2_bounded_control}, we can apply Lemma~\ref{lem:subG} to find $h_k \epsilon \in \subG(C \tau^2 U_h^2)$ with $C = 9/2$. Hence we can apply Lemma~\ref{lemma:concentration} to get that, with probability at least $1-\delta/4$,
	\begin{ceqn}
	\begin{equation}
	\label{eq:2_heps_lasso}
	\max_{k=1,\ldots, m} \left| \sum_{i=1}^n h_k(X_i) \epsilon(X_i)  \right| \leq \sqrt{2nC\tau^2 U_h^2\log(8m/\delta))}.
	\end{equation}
	\end{ceqn}
	\item In view of \citep[Lemma~2.2]{boucheron+l+m:2013} and Assumptions~\ref{ass_2_bounded_control} and \ref{ass_orth_control}, we have $h_{k} h_{l} -  P (h_k h_l) \in \subG(U_h^4)$ for all $k,l \in \{1,\ldots, m \}$.  Hence we can apply Lemma~\ref{lemma:concentration} with $p=m^2$ to get that, with probability at least $1-\delta/4$, 
	\begin{ceqn}
	\begin{equation*}
	\max_{\substack{1\leq k \leq m \\ 1\leq l \leq m}} \left| \sum_{i=1}^n \{ h_k(X_i)h_l(X_i) - P (h_k h_l)\}    \right| 
	\leq \sqrt{2n U_h^4\log(8m^2/\delta)}.
	\end{equation*}
	\end{ceqn}
Denote by $\Delta =  (P_n- P) \{h h^T\} $. Because by assumption $2 (\ell^\star /  \gamma^\star) \sqrt{2 U_h^4\log(8m^2/\delta)}  \leq \sqrt{n}$, we have that 
\begin{ceqn}
\begin{align*}
(\ell^\star /  \gamma^\star)   \max_ {1\leq k,l \leq m} |\Delta_{k,l}|\leq 1 / 2.
\end{align*}
\end{ceqn}
Remark that
\begin{ceqn}
\begin{align*}
\forall u\in \mathbb R^{m}, \quad n^{-1} \norm{Hu}_2^2   -  u^TG  u   &= u^T  \Delta u.
\end{align*}
\end{ceqn}
Then, following \cite[equation (3.3)]{bickel2009simultaneous}, use the inequality $ |u^T  \Delta u|\leq  \norm{u}_1^2  \max_ {1\leq k,l \leq m} |\Delta_{k,l}|$, to obtain that, with probability $1-\delta/4$, for all $u\in \mathcal C (S^\star ; 3)$,
\begin{ceqn}
\begin{align*}
\norm{Hu}_2^2 / n & \geq   u^TG  u   - \norm{u}_1^2  \max_ {1\leq k,l \leq m} |\Delta_{k,l}|\\
&\geq  u^TG  u   - \norm{u}_2^2 \ell^\star  \max_ {1\leq k,l \leq m} |\Delta_{k,l}|\\
&\geq  u^TG  u   - (u^TG  u)  (\ell^\star /  \gamma^\star)   \max_ {1\leq k,l \leq m} |\Delta_{k,l}|\\
&\geq  (u^TG  u )  / 2.
\end{align*}
\end{ceqn}
It follows that with probability at least $1-\delta/4$,
\begin{ceqn}
\begin{align}\label{eq:3_ev_lasso_2}
  \norm{H u}_2^2  \ge ( n\gamma^\star /2)\norm{u}_2^2  .
\end{align}
\end{ceqn}
\end{itemize}

\smallskip

\emph{Step 3.} ---
We claim that, on the event $E$, we have
\begin{ceqn}
\begin{align}\label{eq:application_lemma_3_lasso2}
\forall u\in \mathcal C (S^\star ; 3),\quad  \norm{H_c u}_2^2 \ge  (n \gamma^\star /4) \norm{u}_2^2  
\end{align}
\end{ceqn}
We have
$$ 
H_{c}^T H_{c} 
= H^T H - n \, P_n(h) \, P_n(h)^T 
$$
and thus,
\begin{ceqn}
\begin{align*}
\norm{H_{c} u}_2^2
\ge
\norm{H u}_2^2 - n \max_{k=1,\ldots, m} | P_n(h_k)| ^2 \norm{u}_1^2.
\end{align*}
\end{ceqn}
We treat both terms on the right-hand side. On the one hand, we just have obtained a lower bound for the first term. On the other hand, in view of \eqref{eq:1_h_lasso} and because $\norm{u}_1^2\leq 16 \norm{u_{S^\star}}_1^2\le 16  \ell^\star \norm{u}_2^2$, we have
\begin{ceqn}
\begin{align*}
\norm{u}_1^2 \max_{k=1,\ldots, m} | P_n(h_k)| ^2
	&= \norm{u}_1^2  n^{-2}\cdot \max_{k \in S^\star} \Bigl| \sum_{i=1}^n h_k(X_i) \Bigr|^2 \\
	&\le 16 \ell^\star \norm{u}_2^2 \cdot n^{-2}\cdot 2nU_h^2 \log(8m/\delta) \\
	&\le  \norm{u}_2^2 \gamma^\star / 4
\end{align*}
\end{ceqn}
as $n \ge (16 \times 8) \ell^\star (U_h^{2}/\gamma^\star) \log(8m/\delta)$ by assumption. In combination with \eqref{eq:3_ev_lasso_2}, we find
\begin{ceqn}
\begin{equation*}
	\norm{H_{c} u}_2^2
	\ge n(\gamma^\star/2) \norm{u}_2^2 - n(\gamma^\star/4) \norm{u}_2^2
	= n (\gamma^\star/4) \norm{u}_2^2.
\end{equation*}
\end{ceqn}

\emph{Step 4.} ---
We claim that, on the event $E$, we have

\begin{ceqn}
\begin{equation}
\label{eq:bound_cov_1}
	\norm{H_c^T \epsilon_c^{(n)}}_\infty 
	\le (3+\sqrt{2}/8) \sqrt{\log(8m/\delta)} U_h \tau \sqrt{n}.
\end{equation}
\end{ceqn}

Indeed, on the left-hand side in \eqref{eq:bound_cov_1} we have in virtue of \eqref{eq:0_eps_lasso}, \eqref{eq:1_h_lasso} and \eqref{eq:2_heps_lasso},
\begin{ceqn}
\begin{align*}
&\norm{H_c^T \epsilon_c^{(n)}}_\infty \\
&= \max_{k=1,\ldots, m} \left| 
\sum_{i=1}^n (h_k(X_i) - P_n(h_k)) (\epsilon(X_i) - P_n(\epsilon))
\right|
\\
&= \max_{k=1,\ldots, m} \left| 
\left(\sum_{i=1}^n h_k(X_i) \epsilon(X_i)\right) 
- n P_n(h_k)  P_n(\epsilon)   
\right|  \\
&\leq  \max_{k=1,\ldots, m} 
\left| \sum_{i=1}^n h_k(X_i) \epsilon(X_i) \right| 
+ n^{-1} \left|\sum_{i=1}^n \epsilon (X_i) \right|   \max_{k=1,\ldots, m} \left| \sum_{i=1}^n h_k(X_i) \right| \\
&\leq \sqrt{2n C\tau^2 U_h^2 \log(8m/\delta)} + \\
&n^{-1}\sqrt{2n \tau^2\log(8/\delta)}
\sqrt{2n U_h^2\log(8m/\delta)} \\
&= \sqrt{2n C\tau^2 U_h^2\log(8m/\delta)} 
\left( 1 + \sqrt{2 \log(8/\delta) / (Cn)} \right).
\end{align*}
\end{ceqn}
{Since $\ell^\star \ge 1$ and $\ell^\star U_h^{2} \ge \sum_{k\in S^\star} P(h_k^2) \ge \gamma^\star$, the assumed lower bound on $n$ implies that $n \ge 128 \log(8/\delta)$. As $C = 9/2$, the factor $\sqrt{2C}(1 + \sqrt{2 \log(8/\delta) / (Cn)})$ is bounded by $3 + \sqrt{2}/8$ and we get \eqref{eq:bound_cov_1}.
\smallskip

\emph{Step 5.} --- Recall $\hat{u} = \blasson(f) -  \beta^{\star}(f)$. We claim that, on the event $E$, we have
\begin{ceqn}
\begin{equation}
\label{eq:hatu48}
	\norm{\hat{u}}_1 \le 48 \lambda \ell^\star / \gamma^\star.
\end{equation}
\end{ceqn}
To prove this result, we shall rely on the following lemma. 

\begin{lemma}
	\label{lemma:cone_property}
	If $n \lambda \geq 2 \norm{H_c^T \epsilon_c^{(n)}}_{\infty}$ then, writing $\hat{u} = \blasson(f) -  \beta^{\star}(f)$, we have $\hat{u} \in \mathcal{C}(S^\star; 3)$ and
	\begin{ceqn}
	\begin{equation}
	\label{ineq:cone_property}
	\norm{H_c \hat{u}}_2^2
	\leq 3 n {\lambda} \norm{\hat{u}_{S^\star}}_1.
	\end{equation}
	\end{ceqn}
\end{lemma}

\begin{proof}
	This is just a reformulation of the reasoning on p.~298 in \citep{tibshirani+w+h;2015} {with a slightly sharper upper bound}. The vector $\hat{\nu}$ at the right-hand side of their Eq.~(11.23) {can be replaced by} $\hat{\nu}_S$. For the sake of completeness, we provide the details. 
	
	In the proof we use the short-cuts $\beta^\star = \beta^\star(f)$ and $\blasson = \blasson(f)$. Recall $\epsilon_c^{(n)} = f_c^{(n)} - H_c \beta^\star(f)$ and define
	\begin{ceqn}
	\begin{align*}
	G(u) &= \|f^{(n)}_c-H_c(\beta^\star +   u) \|_2^2 /(2n) + \lambda \| \beta^\star+   u\|_1 \\
	&= \|\epsilon_c^{(n)} - H_c   u \|_2^2 /(2n) + \lambda \| \beta^\star +   u\|_1.
	\end{align*}
	\end{ceqn}
	Because $G(\hat u )\leq G(0)$, we have
	\begin{ceqn}
	\begin{align*}
		\|H_c \hat u \| _2^2/(2n) 
		\leq \hat{u}^T H_c^T \epsilon_c^{(n)} / n + \lambda ( \|\beta^\star\|_1 -\| \beta^\star + \hat u \|_1 )
	\end{align*}
	\end{ceqn}
	From the triangle inequality
\begin{ceqn}
\begin{align*}
	\| (\beta^\star -  (-\hat  u))_{S^\star} \|_1\geq  | \| \beta^\star_{S^\star}\|_1 -  \| \hat  u_{S^\star} \|_1| \geq \| \beta^\star_{S^\star}\|_1 -  \| \hat  u_{S^\star} \|_1,
\end{align*}
\end{ceqn}
implying that
\begin{ceqn}
	\begin{align*}
	&\|\beta^\star\|_1 -\| \beta^\star + \hat u \|_1 \\
	&= \|\beta^\star\|_1 -\| (\beta^\star + \hat u)_{S^\star}  \|_1   - \|( \beta^\star + \hat u)_{\overline{S^\star}} \|_1 \\
	&\leq \|\beta^\star\|_1 - \| \beta^\star_{S^\star}\|_1 + \| \hat  u_{S^\star} \|_1 -  \|  (\beta^\star +\hat  u )_{\overline{S^\star}} \|_1\\
	&= \| \hat  u_{S^\star} \|_1 -  \| \hat  u _{\overline{S^\star}} \|_1.
	\end{align*}
	\end{ceqn}
	From H\"{o}lder's inequality, we get 
	\begin{ceqn}
	\begin{align*}
\abs{\hat{u}^T H_c^T \epsilon_c^{(n)}} \leq \norm{H_c^T \epsilon_c^{(n)}}_\infty \cdot \norm{\hat  u}_1,
	\end{align*}	
	\end{ceqn}
	 which leads to
	 \begin{ceqn}
	\begin{align*}
	\norm{H_c \hat u}_2^2/(2n) 
	\leq \| H_c^T \epsilon_c^{(n)}  \|_\infty \|\hat  u\|_1/n 
	+ \lambda (  \norm{\hat{u}_{S^\star}}_1 - \norm{\hat{u}_{\overline{S^\star}}}_1 ).
	\end{align*}
	\end{ceqn}
	Consequently, because $\norm{H_c^T\epsilon_c^{(n)}}_\infty / n \leq  \lambda/2$ by assumption, we obtain
	\begin{ceqn}
	\begin{align*}
	0\leq \|H_c \hat u \| _2^2/(2n)&\leq \lambda ( \|\hat  u\|_1/2 + \| \hat u_{S^\star} \|_1 -  \|  \hat  u _{\overline{S^\star}} \|_1 ) \nonumber \\
	& =  (\lambda/2) ( 3  \| \hat u_{S^\star} \|_1  -  \|  \hat  u _{\overline{S^\star}} \|_1 ).
	\end{align*}
	\end{ceqn}
	The right-hand side must be nonnegative, whence $\norm{\hat{u}_{\overline{S^\star}}}_1 \le 3 \norm{\hat{u}_{S^\star}}_1$, i.e., $\hat{u} \in \mathcal{C}(S; 3)$. The bound in \eqref{ineq:cone_property} follows as well.
\end{proof}

On the event $E$, the conclusion of Lemma~\ref{lemma:cone_property} is valid because the bound on $\norm{H_c^T \epsilon_c^{(n)}}_\infty$ in~\eqref{eq:bound_cov_1} and the assumption on $\lambda$ in Theorem~\ref{th:lasso} together imply that $\lambda \ge 2 \norm{H_c^T \epsilon_c^{(n)}}_\infty / n$.  The cone property of Lemma~\ref{lemma:cone_property} yields $\hat{u} \in \mathcal{C}(S^\star; 3)$ so that
\begin{ceqn}
\begin{equation}
\label{eq:hatu4}
	\norm{\hat{u}}_1 
	= \norm{\hat{u}_{S^\star}}_1 + \norm{\hat{u}_{\overline{S^\star}}}_1
	\le 4 \norm{\hat{u}_{S^\star}}_1.
\end{equation}
\end{ceqn}
 Thanks to \eqref{eq:application_lemma_3_lasso2} and Lemma \ref{lemma:cone_property}, and since $\abs{S^\star} = \ell^\star$, we get
 \begin{ceqn}
\begin{align*}
	\norm{\hat{u}_{S^\star}}_1^2
	&\le \ell^\star \norm{\hat{u}_{S^\star}}_2^2 \\
		&\le \ell^\star \norm{\hat{u}}_2^2 \\
	&\le \ell^\star \cdot n^{-1} (4/\gamma^\star) \norm{H_{c} \hat{u}}_2^2 \\
	&\le \ell^\star \cdot n^{-1} (4/\gamma^\star) \cdot 3n\lambda \norm{\hat{u}_{S^\star}}_1
	= 12 \ell^\star (\lambda/\gamma^\star) \norm{\hat{u}_{S^\star}}_1.
\end{align*}
\end{ceqn}
It follows that $\norm{\hat{u}_{S^\star}}_1 \le 12 \ell^\star \lambda/\gamma^\star$. In combination with \eqref{eq:hatu4}, we find \eqref{eq:hatu48}.

\smallskip

\emph{Step 6.} ---
Equation \eqref{eq:alpha:lasso:bound} gave a bound on the estimation error involving three terms. On the event $E$, these terms were shown to be bounded in \eqref{eq:0_eps_lasso}, \eqref{eq:1_h_lasso}, and \eqref{eq:hatu48}. It follows that, on $E$, we finally have
\begin{ceqn}
\begin{align*}
	&n \left| \hat \alpha_n^{\mathrm{lasso}} (f) -  P(f) \right| \\
	&\le
	\sqrt{2n\tau^2 \log(8/\delta)}
	+
	48 \lambda \ell^\star / \gamma^\star
	\cdot
	\sqrt{2nU_h^2 \log(8m/\delta)}.
\end{align*}
\end{ceqn}

Divide by $n$ and use $48\sqrt{2} < 68$ to obtain \eqref{ineq:alasso_general}.
\qed
\section{Proof of Theorem \ref{th:lasso_support_recovery}}

Recall that $S^\star = \{ j = 1, \ldots, m : \beta_j^\star(f) \ne 0\}$ with $\ell^\star = \abs{S^\star}$ and that $\overline{S^\star} = \{1,\ldots,m\} \setminus S^\star$. Further, $H_{c,S^\star}$ is the $n \times \ell^\star$ matrix having columns $H_{c,k}$ for $k \in S^\star$, where $H_{c,k}$ is the $k$-th column of $H_c$.
\smallskip

\emph{Step~1.} ---
We first establish some (non-probabilistic) properties of $\blasson(f)$. To this end, we consider the linear regression of the non-active control variates on the active ones: for $k \in \overline{S^\star} = \{ j = 1, \ldots, m : \beta_j^\star(f) = 0 \}$, this produces the coefficient vector
\begin{ceqn}
\begin{equation*}
\hat\theta_{n}^{(k)} \in 
\argmin_{\theta \in \reals^{\ell^\star}} 
\norm{ H_{c,k} - H_{c,S^{\star}} \theta }_2.
\end{equation*}
\end{ceqn}
Further, we consider the OLS oracle estimate $\hat{\beta}_{n}^{\star}$, which is the OLS estimator based upon the active control variables only, i.e.,
\begin{ceqn}
\begin{equation*}
\hat{\beta}_{n}^{\star} \in 
\argmin_{\beta \in \reals^{\ell^\star}} 
\| f^{(n)}_c  - H_{c,S^{\star}} \beta \| _2.
\end{equation*}
\end{ceqn}
Our assumptions will imply that, with large probability, $H_{c,S^{\star}}$ has rank $\ell^*$, in which case
\begin{ceqn}
\begin{align*}
\hat\theta_{n}^{(k)}
&= (H_{c,S^{\star}}^T H_{c,S^{\star}})^{-1} H_{c,S^{\star}}^T H_{c,k},
\\
\hat{\beta}_{n}^{\star}
&= (H_{c,S^{\star}}^T H_{c,S^{\star}})^{-1} H_{c,S^{\star}}^T f_c^{(n)}.
\end{align*}
\end{ceqn}
The following lemma provides a number of (non-probabilistic) properties of $\blasson(f)$, given certain conditions on $H_c$ and $\epsilon_c^{(n)}$. Recall that a norm $\norm{\,\cdot\,}$ on $\reals^p$ induces a matrix norm on $\reals^{p \times p}$ via $\norm{A} = \sup \{ \norm{A u} : u \in \reals^p, \norm{u} = 1 \}$ for $A \in \reals^{p \times p}$.

\begin{lemma}
	\label{lemma:tib_book_deterministic_result}
	If $H_{c,S^{\star}}$ has rank $\ell^\star$ and if there exists $\kappa \in (0, 1]$ such that
	\begin{ceqn}
	\begin{align}
	\label{eq:sub_model_1}
	\max_{k\in \overline{S^\star}} 
	\norm{\hat \theta_{n}^{(k)}}_1 
	&\leq 1 - \kappa, \\
	\label{eq:sub_model_2}
	\max_{k\in \overline{S^\star}}  
	\abs{(H_{c,k} - H_{c,S^{\star}} \hat{\theta}_{n}^{(k)})^T  \epsilon_c^{(n)}} 
	&\le \kappa \lambda n,
	\end{align}
	\end{ceqn}
	then the minimizer $\blasson(f)$ in \eqref{eq:lasso} is unique, with support $\supp(\blasson (f)) \subset S^\star$, and it satisfies
	\begin{multline}
	\label{eq:blassonk:bound:1}
	\max_{k\in S^{\star}} 
	\abs{\blassonk (f) - {\beta}_k^\star (f)} 
	\leq \\
	\max_{k\in S^{\star}} \abs{\hat{\beta}_{n,k}^{\star} - \beta^\star_k (f)} 
	+ n\lambda \norm{(H_{c,S^{\star}}^T H_{c,S^{\star}})^{-1}}_\infty.
	\end{multline}
\end{lemma}

\begin{proof}
	The proof of the previous result is actually contained in \cite{tibshirani+w+h;2015}. The uniqueness of the LASSO solution and the property that it does not select inactive covariates follows directly from the proof of their Theorem~11.3. The only difference is that, in our case, the inequality~\eqref{eq:sub_model_2} is an assumption whereas in \cite{tibshirani+w+h;2015} it is a property of the Gaussian fixed design model. The approach in \cite{tibshirani+w+h;2015} is based upon checking the \emph{strict dual feasibility condition}. The bound~\eqref{eq:blassonk:bound:1} is Eq.~(11.37) in \cite{tibshirani+w+h;2015}.
\end{proof}

We slightly modify Lemma~\ref{lemma:tib_book_deterministic_result} to make the conditions~\eqref{eq:sub_model_1} and~\eqref{eq:sub_model_2} easier to check and to make the bound~\eqref{eq:blassonk:bound:1} easier to use. 

\begin{lemma}
	\label{lemma:tib_book_deterministic_result_new}
	If there exists $\nu > 0$ such that
	\begin{ceqn}
	\begin{equation}
	\label{eq:emp_gram_ev_restricted} 
	\forall u\in \reals^{\ell^\star}, \qquad \norm{H_{c,S^{\star}} u}_2^2
	\ge n \nu \norm{u}_2^2,
	\end{equation}
	\end{ceqn}
	and if there exists $\kappa \in (0,1]$ such that
	\begin{ceqn}
	\begin{align}	
	\label{eq:sub_model_1_new}
	\frac{\ell^\star} {\nu n}
	\max_{k\in \overline{S^\star}} \max_{j\in S^{\star}} 
	\abs{H_{c,j}^T H_{c,k}}  
	&\leq 1 - \kappa, \\
	\label{eq:sub_model_2_new}
	\max_{k=1,\ldots, m } \abs{H_{c,k}^T \epsilon_c^{(n)}}  
	&\le \frac{1}{2} \kappa \lambda n,
	\end{align}
	\end{ceqn}
	then the minimizer $\blasson(f)$ in \eqref{eq:lasso} is unique, with support satisfying $\supp(\blasson (f)) \subset S^\star$, and it holds that
	\begin{ceqn}
	\begin{equation}
	\label{eq:blassonk:bound:2}
	\max_{k\in S^{\star}} 
	\abs{\blassonk (f) - {\beta}_k^\star (f)}
	\leq (1 + \kappa/2) \sqrt{\ell^\star} \lambda / \nu.
	\end{equation}
	\end{ceqn}
\end{lemma}

\begin{proof}

	By \eqref{eq:emp_gram_ev_restricted}, the smallest eigenvalue of the $\ell^\star \times \ell^\star$ matrix $H_{c,S^\star}^T H_{c,S^\star}$ is positive, so that it is invertible and $H_{c,S^\star}$ has rank $\ell^\star$.
	
	We show that \eqref{eq:sub_model_1_new} implies \eqref{eq:sub_model_1}. For each $k \in \overline{S^\star}$, the vector $\hat{\theta}_n^{(k)}$ has length $\ell^\star$, so that 
	$$ 
	\norm{\hat \theta_{n}^{(k)}}_1
	\leq \sqrt {\ell^{\star}} \norm{\hat \theta_{n}^{(k)}}_2. 
	$$
	 
	Because $\hat\theta_{n}^{(k)} $ is an OLS estimate, using that the largest eigenvalue of $(H_{c,S^{\star}}^T H_{c,S^{\star}})^{-1}$ being bounded from above by $(n \nu)^{-1}$, we obtain
	$$
	\norm{\hat\theta_{n}^{(k)}  }_2
	= \norm{(H_{c,S^{\star}}^T H_{c,S^{\star}})^{-1} H_{c,S^{\star}}^T H_{c,k} }_2
	\le \frac{1}{n \nu} \norm{H_{c,S^{\star}}^T H_{c,k} }_2
	$$
	Since $\norm{x}_2 \le \sqrt{m} \, \norm{x}_\infty$ for $x \in \reals^m$, we can conclude that
	\begin{ceqn}
	\begin{equation*}
	\norm{\hat \theta_{n} ^{(k)}}_2 
	\leq \frac{\sqrt {\ell^\star}}{\nu n} 
	\max_{j\in S^{\star}} \abs{H_{c,j}^T H_{c,k}}. 
	\end{equation*}
	\end{ceqn}
	Combining the two bounds, we find that \eqref{eq:sub_model_1_new} indeed implies \eqref{eq:sub_model_1}. 
	
	Next we show that \eqref{eq:sub_model_2_new} implies \eqref{eq:sub_model_2}. For $k \in \overline{S^\star}$, we have
	\begin{ceqn}
	\begin{align*}
	&\abs{(H_{c,k} - H_{c,S^{\star}} \hat\theta_{n}^{(k)})^T  \epsilon_c^{(n)}}
	\\	
	&\leq \abs{H_{c,k}^T \epsilon_c^{(n)}} 
	+ \abs{(\hat\theta_{n}^{(k)})^T H_{c,S^{\star}}^T \epsilon_c^{(n)}}\\
	&\leq \abs{H_{c,k}^T \epsilon_c^{(n)}} 
	+ \norm{\hat \theta_{n}^{(k)}}_1
	\max_{j \in S^{\star}} \abs{H_{c,j}^T \epsilon_c^{(n)}}.
	\end{align*}
	\end{ceqn}
	Using \eqref{eq:sub_model_1} and \eqref{eq:sub_model_2_new} we deduce \eqref{eq:sub_model_2}. 
	
	The conditions of Lemma \ref{lemma:tib_book_deterministic_result} have been verified, and so its conclusion holds. We simplify the two terms in the upper bound~\eqref{eq:blassonk:bound:1}. First, we use that 
\begin{ceqn}
\begin{align*}
	\norm{\hat{\beta}_{n}^\star  - \beta^\star(f)}_2 	
	&= \norm{(H_{c,S^{\star}}^T H_{c,S^{\star}})^{-1} H_{c,S^{\star}}^T \epsilon_c^{(n)} }_2 \\
	&\le  
	\frac{\sqrt{\ell^\star}}{ \nu n } \norm{H_c^T \epsilon_c^{(n)}}_\infty.
	\end{align*}	
	\end{ceqn}
	 Second, for any matrix $A \in \reals^{p\times p}$, we have $\norm{A}_\infty \leq  \sqrt{p} \norm{A}_2$ (e.g., \cite[page 365]{horn+j:2012}), and this we apply to $(H_{c,S^{\star}}^T H_{c,S^{\star}})^{-1}$. In this way, the upper bound in \eqref{eq:blassonk:bound:1} is dominated by
	\begin{multline*}
	\norm{\hat{\beta}_ {n}^{\star} - \beta^\star(f)}_2 
	+ n\lambda \cdot \sqrt{\ell^\star}
	\norm{ (H_{c,S^{\star}}^T H_{c,S^{\star}})^{-1} }_2
	\\
	\le \frac{\sqrt{\ell^\star}}{n \nu} \max_{k \in S^\star} \abs{H_{c,k}^T \epsilon_c^{(n)}}
	+ n\lambda \cdot \sqrt{\ell^\star} \cdot \frac{1}{n \nu},
	\end{multline*}
	since the largest eigenvalue of $(H_{c,S^{\star}}^T H_{c,S^{\star}})^{-1}$ is at most $(n \nu)^{-1}$. Use \eqref{eq:sub_model_2_new} to further simplify the right-hand side, yielding \eqref{eq:blassonk:bound:2}.
\end{proof}

\emph{Step 2.} ---
Let $\delta \in (0, 1)$ and $n = 1, 2, \ldots$. In a similar way as in the proof of Theorem~\ref{th:ols}, we construct an event of probability at least $1-\delta$. This time, we need five inequalities to hold simultaneously.

\begin{itemize}[leftmargin=1pc]
	
	\item
	Because $\epsilon \in \subG( \tau^2)$, with probability at least $1-\delta/5$,
	\begin{ceqn}
	\begin{equation}
	\label{eq:0_eps_support}
	\left|\sum_{i=1}^n \epsilon (X_i) \right|
	\leq \sqrt{2n \tau^2 \log(10/\delta)}.
	\end{equation}
	\end{ceqn}
	
	\item
	In view of \citep[Lemma~2.2]{boucheron+l+m:2013} and Assumption~\ref{ass_2_bounded_control}, we have $h_k\in \subG(U_h^2)$ for all $k=1,\ldots, m$. Hence we can apply Lemma \ref{lemma:concentration} with $p=m$ to get that, with probability at least $1-\delta/5$, 
	\begin{ceqn}
	\begin{equation}
	\label{eq:1_h_support}
	\max_{k=1,\ldots, m} \left| \sum_{i=1}^n h_k(X_i)  \right| 
	\leq \sqrt{2n U_h^2\log(10m/\delta)}.
	\end{equation}
	\end{ceqn}
	\item
	By virtue of Assumptions~\ref{ass_1_sub_gauss} and~\ref{ass_2_bounded_control}, we can apply Lemma~\ref{lem:subG} to have $h_k \epsilon \in \subG(C U_h^2 \tau^2 )$, where $C = 9/2$. Hence we can apply Lemma~\ref{lemma:concentration} to get that, with probability at least $1-\delta/5$,
	\begin{ceqn}
	\begin{equation}
	\label{eq:2_heps_support}
	\max_{k=1,\ldots, m} 
	\left| \sum_{i=1}^n h_k(X_i) \epsilon(X_i)  \right| 
	\leq \sqrt{2C n \tau^2 U_h^2\log(10m/\delta))}.
	\end{equation}
	\end{ceqn}
	
\item
Recall that $B^\star 
	= \sup_{x \in \mathcal{X}} h_{S^\star}^T(x) G_{S^\star}^{-1} h_{S^\star}(x)$ with
	\begin{ceqn}
\begin{equation*}
\label{eq:Bstar:zeta}
	B^\star 	\le \lambda_{\max}(G_{S^\star}^{-1}) 
	\sup_{x \in \mathcal{X}} h_{S^\star}^T(x) h_{S^\star}(x)  
	\le \ell^\star U_h^2 / \gamma^{\star\star},
\end{equation*}	
\end{ceqn}
The assumption on $n$ easily implies that $n \ge 8 B^{\star} \log(5\ell^\star/\delta)$. Applying Lemma \ref{lemma:smallest_eigen_value} with $p=\ell^\star$, $g = h_{S^\star}$, and $\delta$ replaced by $\delta/5$, we find that, with probability at least $1-\delta/5$,
	\begin{ceqn}
	\begin{equation}
	\label{eq:3_ev_support}
	\norm{H_{S^{\star}} u}_2^2 \ge n\gamma^{\star\star} \norm{u}^2_2/2,
	\qquad \forall u\in \reals^{\ell^*}.
	\end{equation}
	\end{ceqn}
	
	\item
	Finally, {because $\abs{h_j(x)} \le U_h$ for all $x \in \mathcal{X}$ and $j \in \{1,\ldots,m\}$ and because $P(h_k h_j) = 0$ for all $(k, j) \in \overline{S^\star} \times S^\star$}, we have $h_kh_j \in \subG (U_h^4)$ for such $k$ and $j$, and thus, with probability at least $1-\delta/5$,
	\begin{ceqn}
	\begin{equation}
	\label{eq:4_control_cross_prod}
	\max_{k\in \overline{S^\star}}  \max_{j\in S^{\star}}   
	\left| \sum_{i=1}^n h_k(X_i) h_j(X_i)  \right| 
	\leq \sqrt{2nU_h^4 \log(10 \ell^\star m  /\delta ) }.
	\end{equation}
	\end{ceqn}
\end{itemize}

By the union bound, the event, say $E$, on which \eqref{eq:0_eps_support}, \eqref{eq:1_h_support}, \eqref{eq:2_heps_support}, \eqref{eq:3_ev_support} and \eqref{eq:4_control_cross_prod}  are satisfied simultaneously has probability at least $1-\delta$. We work on the event $E$ for the rest of the proof.
\smallskip

\emph{Step 3.} ---
On the event $E$, we have
\begin{ceqn}
\begin{equation}
\label{eq:application_lemma_3_support}
\forall u\in \reals^{\ell^\star}, \qquad \norm{H_{c,S^{\star}} u}_2^2
\ge n \alpha \gamma^{\star\star} \norm{u}_2^2,
\end{equation}
\end{ceqn}
where $\alpha \in (0, 1/2)$ is an absolute constant whose value will be fixed in Step~6(ii). We have
$$
H_{c,S^{\star}}^T H_{c,S^{\star}} 
= H_{S^\star}^T H_{S^\star} - n \, P_n(h_{S^\star}) \, P_n(h_{S^\star})^T 
$$
and thus, by the Cauchy--Schwarz inequality and by \eqref{eq:3_ev_support},
\begin{ceqn}
\begin{align*}
\norm{H_{c,S^{\star}} u}_2^2
&\ge
\norm{H_{S^\star} u}_2^2 - n \norm{P_n(h_{S^\star})}_2^2 \norm{u}_2^2 \\
&\ge
n \left(\gamma^{\star\star} / 2 - \norm{P_n(h_{S^\star})}_2^2\right) \norm{u}_2^2.
\end{align*}
\end{ceqn}
In view of \eqref{eq:1_h_support}, we have
$$
\norm{P_n(h_{S^\star})}_2^2
\le \frac{\ell^\star}{n^2} 2nU_{h}^2 \log(10m/\delta)
= 2\ell^\star \log(10m/\delta) U_{h}^2 / n.
$$
We thus get
$$
\norm{H_{c,S^{\star}} u}_2^2
\ge
n \gamma^{\star\star} 
\left[
\frac{1}{2} -
\frac{2\ell^\star \log(10m/\delta) U_{h}^2 / \gamma^{\star\star}}{n} 
\right] 
\norm{u}_2^2
$$
A sufficient condition for \eqref{eq:application_lemma_3_support} is thus that the term in square brackets is at least $\alpha$, i.e.,
$$
n \geq \frac{2}{1/2 - \alpha} 
\ell^\star \log(10m/\delta) U_{h}^2 / \gamma^{\star\star}
$$
Since $\ell^\star \ge 1$ and $U_{h}^2 \ge \gamma^{\star\star}$, a condition of the form
\begin{ceqn}
\begin{equation}
\label{eq:beta}
n \ge \rho \log(10\ell^\star m/\delta) [\ell^\star(U_h^2 / \gamma^{\star\star})]^2
\end{equation}
\end{ceqn}
is thus sufficient, with much to spare, provided $\rho > 2/(1/2-\alpha)$. In Step~6(ii), we will choose $\alpha$ in such a way that the constant $\rho = 70$ appearing in the statement of the theorem is sufficient.
\smallskip

\emph{Step 4.} ---
On the event $E$, we have
\begin{multline}
\label{eq:application_lemma_2}
\max_{k\in \overline{S^\star}}  \max_{j\in S^{\star}} |  H_{c,j} ^T H_{c,k}| \\
\le
\sqrt{2nU_h^4 \log(10 \ell^\star m  /\delta ) }
+ 2 U_h^2\log(10m/\delta).
\end{multline}
Indeed, denote $A = \overline{S^\star} \times S^\star$, in virtue of \eqref{eq:1_h_support} and \eqref{eq:4_control_cross_prod}, the left-hand side is bounded by
\begin{ceqn}
\begin{align*}
&\lefteqn{
	\max_{(k,j) \in A}  \left| \left(\sum_{i=1}^n h_k(X_i) h_j(X_i)\right)  - n P_n(h_k)  P_n(h_j)    \right|  
} \\
&\leq \max_{(k,j) \in A}   
\left| \sum_{i=1}^n h_k(X_i) h_j(X_i)  \right| 
+ \frac{1}{n} \max_{k\in \overline{S^\star}}
\left|\sum_{i=1}^n h_k (X_i) \right| 
\max_{j\in S^{\star}}  \left| \sum_{i=1}^n h_j(X_i)  \right|      \\
&\leq \max_{(k,j) \in A} 
\left| \sum_{i=1}^n h_k(X_i) h_j(X_i)  \right| 
+ \frac{1}{n} \max_{k = 1,\ldots m} 
\left|\sum_{i=1}^n h_k (X_i) \right|^2      \\
&\leq \sqrt{2nU_h^4 \log(10 \ell^\star m  /\delta ) }
+ \frac{1}{n} 2n U_h^2\log(10m/\delta),
\end{align*}
\end{ceqn}
which is \eqref{eq:application_lemma_2}.
\smallskip

\emph{Step 5.} ---
On the event $E$, we have
\begin{multline}
\label{eq:application_lemma_1}
\norm{H_c^T \epsilon_c^{(n)}}_\infty
\leq \\
\sqrt{2n C \tau^2 U_h^2\log(10m/\delta)} 
\left( 1 + \sqrt{2 \log(10/\delta) / (Cn)} \right).
\end{multline}
The proof is the same as the {first part of the} one  \eqref{eq:bound_cov_1}.
\smallskip

\emph{Step 6.} ---
We will verify that on the event $E$, the three assumptions of Lemma~\ref{lemma:tib_book_deterministic_result_new} are satisfied with $\kappa = 1/2$ and $\nu = \alpha \gamma^{\star\star}$, with $\alpha$ as in Step~3.

\begin{enumerate}[label=(\roman*)]
	\item 
	Eq.~\eqref{eq:emp_gram_ev_restricted} with $\nu = \alpha \gamma^{\star\star}$ is just \eqref{eq:application_lemma_3_support}.
	
	\item 
	Eq.~\eqref{eq:sub_model_1_new} with $\nu = \alpha \gamma^{\star\star}$ and $\kappa = 1/2$ follows from \eqref{eq:application_lemma_2} provided we have
	\begin{align*}
	\frac{\ell^\star}{\alpha \gamma^{\star\star} n}
	\left( \sqrt{2nU_h^4 \log(10 \ell^\star m  /\delta ) }
	+ 2 U_h^2\log(10m/\delta) \right)
	\\
	\le 1 - \frac{1}{2}.	
	\end{align*}
	To check whether this is satisfied, we will make use of the elementary inequality\footnote{The convex parabola $x \mapsto a x^2 - b x - c$ has zeroes at $x_{\pm} = (b \pm \sqrt{b^2 + 4ac}) / (2a)$, and $x_- < 0 < x_+ < \sqrt{b^2 + 4ac}/a$.} 
	\begin{ceqn}
	\begin{equation*}
	\forall (a,b,c) \in (0, \infty)^3, \, 
	\forall x \ge \sqrt{b^2 + 4ac}/a,
	\qquad
	a x^2 \ge bx + c.
	\end{equation*} 
	\end{ceqn}
	with $x = \sqrt{n}$ and 
	\begin{ceqn}
	\begin{align*}
	a &= \alpha \gamma^{\star\star} / (2\ell^\star), \\
	b &= \sqrt{2 U_h^4 \log(10 \ell^\star m / \delta)}, \\
	c &= 2U_h^2 \log(10m/\delta).
	\end{align*} 
	\end{ceqn}
	Sufficient is that $n = x^2$ is bounded from below by $(b^2 + 4ac)/a^2 = (b/a)^2 + 4c/a$, which is
	\begin{multline*}
	\frac%
	{2 U_h^4 \log(10 \ell^\star m / \delta)}%
	{(\alpha \gamma^{\star\star} / (2\ell^\star))^2} 
	+ 4 \frac%
	{2U_h^2 \log(10m/\delta)}%
	{\alpha \gamma^{\star\star} / (2\ell^\star)} = \\
	\frac{8}{\alpha^2} \log(10 \ell^\star m/\delta) \left(\frac{\ell^\star U_h^2}{ \gamma^{\star\star}}\right)^2
	+
	\frac{16}{\alpha} \log(10 m/\delta) \left(\frac{\ell^\star U_h^2}{ \gamma^{\star\star}}\right).
	\end{multline*}
	But $\ell^\star \ge 1$ and $\gamma^{\star\star} \le (1/\ell^\star) \sum_{j \in S^\star} P(h_j^2) \le U_h^2$, so that a sufficient condition is that
	$$
	n \ge \left(\frac{8}{\alpha^2} + \frac{16}{\alpha}\right)
	\log(10 \ell^\star m/\delta) [\ell^\star(U_h^2 / \gamma^{\star\star})]^2.
	$$
	The constant $\rho$ in \eqref{eq:beta} must thus be such that
	$$
	\rho \ge \max\left( \frac{2}{1/2 - \alpha}, \frac{8}{\alpha^2} + \frac{16}{\alpha} \right).
	$$
	The minimum of the right-hand side as a function of $\alpha \in (0, 1/2)$ occurs at $\alpha = \sqrt{2}/3$ and is equal to $2/(1/2 - \sqrt{2}/3) \approx 69.9$. Taking $\rho = 70$ as in the assumption on $n$ is thus sufficient. 
	
	\item
	Eq.~\eqref{eq:sub_model_2_new} with $\kappa = 1/2$ follows from \eqref{eq:application_lemma_1}, since
	$$
	\sqrt{n \tau^2 U_h^2\log(10m/\delta)} 
	\left( \sqrt{2C} + 2\sqrt{ \log(10/\delta) / n} \right) \le \lambda n /4
	$$
	by the assumed lower bound on $\lambda$. Indeed, since $\ell^\star \ge 1$ and $U_{h}^{2} \ge \gamma^{\star\star}$, the assumed lower bounds on $n$ imply that $n \ge 70 \log(10/\delta)$, so that $ 2\sqrt{ \log(10/\delta) / n}$ is bounded by $2/\sqrt{70}$; recall $C = 9/2$. Since $4 \cdot (3 + 2/\sqrt{70}) \approx 12.9$, the assumed lower bound for $\lambda$ suffices.
\end{enumerate}
\smallskip

\emph{Step 7.} ---
By the previous step, the conclusions of Lemma~\ref{lemma:tib_book_deterministic_result_new} with $\kappa = 1/2$ and $\nu = \alpha \gamma^{\star\star}$ hold on the event $E$, where $\alpha = \sqrt{2}/3$ was specified in Step~6(ii). The minimizer $\blasson$ in \eqref{eq:lasso} is thus unique and we have $\supp(\blasson(f)) \subset S^\star$.

To show the reverse inclusion, we need to verify that $\abs{\blassonk(f)} > 0$ for all $k \in S^\star$. To this end, we apply \eqref{eq:blassonk:bound:2} with $\kappa = 1/2$ and $\nu = \alpha \gamma^{\star\star}$, which becomes
$$
\max_{k \in S^\star} \abs{\blassonk(f) - \beta_k^\star(f)} 
\le (5/4) \sqrt{\ell^{\star}} \lambda / (\alpha \gamma^{\star\star}).
$$
For any $k\in S^\star$, we thus have
$$
\abs{\blassonk(f)}
\ge \min_{j \in S^{\star}} \abs{\beta_j^\star(f)}  
- (5/(4\alpha)) \sqrt{\ell^{\star}} \lambda / \gamma^{\star\star}.
$$
But for $\alpha = \sqrt{2}/3$, we have approximately $5/(4\alpha) \approx 2.65$. Since $\min_{j \in S^\star} \abs{\beta_j^\star (f)} > 3 \sqrt{\ell^{\star}} \lambda / \gamma^{\star\star}$ by the assumed upper bound for $\lambda$, we find $\abs{\blassonk(f)} > 0$, as required. \qed
\section{Proof of Theorem \ref{th:lslasso}}
\label{app:th:lslasso}

Recall that the LSLASSO estimator is defined as an OLS estimate computed on the active variables selected by the LASSO based on a subsample of size $N \in \{1,\ldots,n\}$. 
Let $\hat{\beta}_N^{\mathrm{lasso}}(f)$ denote the LASSO coefficient vector in \eqref{eq:lasso} based on the subsample $X_1, \ldots, X_N$ and let 
$$
	\hat{S}_N 
	= \supp(\hat{\beta}_N^{\mathrm{lasso}}(f)) 
	= \{ k \in \{1, \ldots, m\} \,:\, 
		\hat{\beta}_{N,k}^{\mathrm{lasso}}(f)>0 \} 
$$
denote the estimated active set of $\hat{\ell} = |\hat{S}_N|$ control variates. The LSLASSO estimate $\hat{\alpha}_n^{\mathrm{lslasso}}(f)$ based on the full sample $X_1, \ldots, X_n$ is defined as the OLS estimator based on the control variates $h_k$ for $k \in \hat{S}_N$: writing $H_{\hat{S}_N}$ for the $n \times \hat{\ell}$ matrix with columns $(h_k(X_i))_{i=1}^n$ with $k \in \hat{S}_N$, we have
\begin{ceqn}
\begin{equation*}
\bigl(
\hat \alpha_n^{\mathrm{lslasso}}(f), 
\hat \beta_n^{\mathrm{lslasso}}(f)
\bigr)
\in \argmin_{(\alpha , \beta) \in \reals \times \reals^{\hat \ell}} 
\norm{f^{(n)}-\alpha \1_n-H_{\hat{S}_N} \beta}_{2}^{2},
\end{equation*}
\end{ceqn}
Therefore, we can derive a concentration inequality by combining the support recovery property (Theorem \ref{th:lasso_support_recovery}) along with the concentration inequality for the OLS estimate (Theorem \ref{th:ols}) using only the active control variates. 

Let $\delta >0$ and $n \ge 1$. We construct an event with probability at least $1-\delta$ on which the support recovery property and the concentration inequality for the OLS estimate hold simultaneously. Recall that $S^\star = \supp(\beta^\star(f))$ is the true set of $\ell^\star = \abs{S^\star}$ active control variables.
\begin{itemize}[leftmargin=1pc]
	\item Thanks to Theorem~\ref{th:lasso_support_recovery}, with probability at least $1-\delta/2$,
	\begin{ceqn}
	\begin{equation}
	\label{eq:support_equal}
	\hat S_N = S^\star.
	\end{equation}
	\end{ceqn}
	{Indeed, the conditions on $N$ and $\lambda$ in Theorem~\ref{th:lslasso} are such that we can apply Theorem~\ref{th:lasso_support_recovery} with $n$ and $\delta$ replaced by $N$ and $\delta/2$, respectively.}
	
	\item
	Thanks to Theorem~\ref{th:ols}, with probability at least $1-\delta/2$,
	\begin{multline}
	\label{eq:ols_restricted}
	\abs{\aols(f,h_{S^\star}) - P(f)}
	\le \sqrt{2 \log(16/\delta)} \frac{\tau}{\sqrt{n}} \\
	+ 58 \sqrt{B^\star \ell^\star \log(16 \ell^\star /\delta) \log(8/\delta)}\frac{\tau}{n}.
	\end{multline}
	where for any $S\subset \{1,\ldots, m\}$, $ \aols(f,h_{S})$ is the OLS estimate of $P(f)$ based on the control variates $h_{S}$. Indeed, we apply Theorem~\ref{th:ols} with $h$ and $\delta$ replaced by $h_{S^\star}$ and $\delta/2$, respectively. The required lower bound on $n$ is now $n \geq \max\left(18 B^\star \log(8 \ell^\star/\delta),75 \ell^\star \log(8/\delta) \right)$. By assumption we have $N \geq 75 [\ell^\star(U_h^2/\gamma^{\star\star})]^2 \log(20 \ell^\star/\delta)$. The required lower bound is already satisfied for $N$, and thus certainly by $n$.
\end{itemize}

By the union bound, the event on which \eqref{eq:support_equal} and \eqref{eq:ols_restricted} are satisfied simultaneously has probability at least $1-\delta$. 
On this event, we can, by definition of $\hat{\alpha}_n^\mathrm{lslasso}(f)$ and by \eqref{eq:support_equal}, write the integration error as
\begin{ceqn}
\begin{align*}
\abs{\hat \alpha_n^\mathrm{lslasso}(f)- P(f)}
&= \abs{\aols(f,h_{\hat{S}_N}) - P(f)} \\
&= \abs{\aols(f,h_{S^\star}) - P(f)}.
\end{align*}
\end{ceqn}
But the right-hand side is bounded by \eqref{eq:ols_restricted}, yielding \eqref{ineq:alslasso}, as required.

\bibliographystyle{spbasic}
\bibliography{main.bbl}
\end{document}